\documentclass[10pt,a4paper]{amsart}
\usepackage[latin1]{inputenc}
\usepackage{amsmath}
\usepackage{amsfonts}
\usepackage[T1]{fontenc}
\usepackage{amssymb}
\usepackage{graphicx}
\usepackage{amscd}
\usepackage{mathrsfs}
\usepackage[all]{xy}
\usepackage{overpic}
\usepackage{appendix}

\usepackage{tikz}
\usetikzlibrary{matrix}

\title[Courant algebroid automorphisms and moduli of generalized metrics]{The Lie group of automorphisms of a Courant algebroid and the moduli space of generalized metrics}

\author[R. Rubio]{Roberto Rubio}
\author[C. Tipler]{Carl Tipler}

\address{Weizmann Institute of Science, 234 Herzl St, Rehovot 76100, Israel}
\email{roberto.rubio@weizmann.ac.il} 
\address{LMBA, UMR CNRS 6205; D\'epartement de
  Math\'ematiques, Universit\'e de Bretagne Occidentale, 6, avenue
  Victor Le Gorgeu, 29238 Brest Cedex 3 France}
\email{carl.tipler@univ-brest.fr}

\thanks{RR was supported by CAPES, ISF grant 687/13 and a Minerva foundation grant. CT is partially supported by ANR project EMARKS No ANR-14-CE25-0010 and by CNRS grant PEPS jeune chercheur.}

\theoremstyle{plain}
\newtheorem{theorem}{Theorem}[section]
\newtheorem{lemma}[theorem]{Lemma}
\newtheorem{corollary}[theorem]{Corollary}
\newtheorem{proposition}[theorem]{Proposition}

\theoremstyle{definition}
\newtheorem{definition}[theorem]{Definition}
\newtheorem{definition-theorem}[theorem]{Definition-Theorem}
\newtheorem{example}[theorem]{Example}
\theoremstyle{remark}
\newtheorem{remark}[theorem]{Remark}

\newcommand{\Id}{\operatorname{Id}}

\newcommand{\ad}{\operatorname{ad}}

\newcommand{\Aut}{\operatorname{Aut}}
\newcommand{\Der}{\operatorname{Der}}

\newcommand{\RR}{{\mathbb R}}

\newcommand{\NN}{{\mathbb N}}
\newcommand{\rk}{\operatorname{rk}}

\newcommand{\surj}{\to\kern-1.8ex\to}

\newcommand{\cA}{\mathcal{A}}
\newcommand{\cC}{\mathcal{C}}

\newcommand{\cM}{\mathcal{M}}

\newcommand{\cF}{\mathcal{F}}
\newcommand{\cV}{\mathcal{V}}
\newcommand{\cW}{\mathcal{W}}

\newcommand{\cG}{\mathcal{G}}
\newcommand{\cL}{\mathcal{L}}
\newcommand{\cO}{\mathcal{O}}

\newcommand{\cR}{\mathcal{R}}
\newcommand{\cS}{\mathcal{S}}
\newcommand{\cT}{{\mathcal{T}}}

\newcommand{\cH}{\mathcal{H}} 

\newcommand{\dvol}{\mathrm{dvol}}

\newcommand{\Man}{M}

\def\ep{\varepsilon}
\def\om{\omega}
\def\Om{\Omega}

\def\Diff{\mathrm{Diff}}
\def\bfDiff{\mathrm{Diff}}
\def\GDiff{\mathrm{GDiff}}
\def\bfGDiff{\mathrm{GDiff}}
\def\bfOm{{\Omega}}
\def\bfGamma{{\Gamma}}

\def\GM{\mathcal{GM}}
\def\bfGM{\GM}

\def\GR{\mathcal{GR}}

\def\Isom{\mathrm{Isom}}
\def\Im{\mathrm{Im}}
\def\Id{\mathrm{Id}}

\def\cA{\mathcal{A}}

\def\cF{\mathcal{F}}
\def\cG{\mathcal{G}}
\def\cH{\mathcal{H}}

\def\cT{\mathcal{T}}

\def\cR{\mathcal{R}}
\def\cU{\mathcal{U}}
\def\cS{\mathcal{S}}

\def\bfE{\mathrm{E}}

\def\bfG{\mathrm{G}}
\def\bfF{\mathrm{F}}
\def\bfH{\mathrm{H}}
\def\bfA{\mathbf{A}}
\def\bfB{\mathbf{B}}

\def\gdiff{\mathfrak{gdiff}}

\newcommand{\st}{\;|\;}

\newcommand{\R}{{\mathbb{R}}}

\newcommand{\bbG}{{\bf{\mathbb{G}}}}

\newcommand{\cCi}{C^\infty}

\newcommand{\fg}{\mathfrak{g}}
\newcommand{\fh}{\mathfrak{h}}

\newcommand{\fm}{\mathfrak{m}}

\newcommand{\fgdiff}{\mathfrak{gdiff}}

\newcommand{\al}{\alpha}

\newcommand{\la}{\langle}
\newcommand{\ra}{\rangle}

\newcommand{\SO}{\mathrm{SO}}

\newcommand{\SSS}{\mathrm{S}}

\newcommand{\E}{\mathrm{E}}

\newcommand{\OO}{\mathrm{O}}

\renewcommand{\Im}{\mathrm{Im}}

\def\DH{\Diff_{[H]}}

\def\cGM{\mathcal{GM}}

\def\cR{\mathcal{R}}
\def\cRuH{\mathcal{R}^{[H]}}
\def\cGR{\mathcal{GR}}

\def\cGRuH{\mathcal{GR}^{H}}

\def\cBH{\mathcal{B}_{[H]}}

\def\s{\lambda}
\def\Spl{{\Lambda}}

\begin{document}

\begin{abstract}
  We endow the group of automorphisms of an exact Courant algebroid over a compact manifold with an infinite dimensional Lie group structure modeled on the inverse limit of Hilbert spaces (ILH).
  We prove a slice theorem for the action of this Lie group on the space of generalized metrics.
  As an application, we show that the moduli space of generalized metrics is stratified by ILH submanifolds and relate it to the moduli space of usual metrics. Finally, we extend these results to odd exact Courant algebroids.
\end{abstract}

\maketitle

\vspace{-1cm}

\tableofcontents


\section{Introduction}
\label{sec:intro}

This paper lays the foundations for a programme concerning the action of symmetries, automorphisms of a Courant algebroid, on geometric structures in generalized geometry. In the study of moduli spaces of differential geometric structures, the knowledge of infinite-dimensional Lie groups plays a fundamental role. In this work, we focus on the Lie group structure of the group of automorphisms of an exact or odd exact Courant algebroid over a compact manifold, and study the action of this group on the space of generalized metrics.

The bundle $TM+T^*M$, the Whitney sum of the tangent and cotangent bundles of a manifold, is canonically endowed with a pairing and a projection to the tangent bundle. Further equipped with the Dorfman bracket on its sections,
it provides not only the frame for the theory of Dirac structures \cite{Cou} but also the motivating example for the definition of a Courant algebroid \cite{LWX}. 

Generalized geometry, as initiated by Hitchin \cite{Hit1}, and more specifically generalized complex \cite{G1annals} and 
K\"ahler \cite{Gua14} geometry, were firstly developed on  $TM+T^*M$ and its twisted versions, 
known as exact Courant algebroids (see Section \ref{sec:definition courant alg} for definitions).
Generalized geometry is concerned with the definition and study of geometric structures on $TM+T^*M$ or a more general Courant algebroid. 
The automorphisms of this Courant algebroid become the symmetries of the theory: the generalized diffeomorphisms $\GDiff$. 
A remarkable feature of this group is that it includes new symmetries: the $B$-fields.

One of the first examples of a generalized geometric structure is that of a generalized 
metric, a maximally positive-definite subbundle. 
Generalized metrics were introduced by Gualtieri \cite{G1} in the context of generalized K\"ahler geometry, 
and have been recently used to restate some systems of coupled equations in a simpler fashion.
For instance, field equations originating in supergravity are interpreted as the vanishing of a suitable generalized Ricci tensor \cite{GF}. 
On the other hand, the Strominger system, whose solutions provide, as proposed by Yau \cite{Yau}, a generalization of Calabi-Yau structures, is reformulated as generalized Killing spinor equations \cite{grt}. Very recently, special holonomy metrics with torsion
have also been described by Killing spinors in generalized geometry \cite{GF2}. All these applications arise on Courant algebroids of the form $TM+\ad P+T^*M$, where $P$ is a principal $G$-bundle over the manifold and $G$ is a compact Lie group. Odd exact Courant algebroids correspond to $G=\SSS^1$ and their study in this work provides a toy model for future applications.

Motivated by the applications of the generalized Ricci tensor and generalized Killing spinors, and the fact that
generalized Riemannian quantities are natural under the $\GDiff$-action,
one is lead to consider the moduli of generalized metrics.
In this context, the present paper presents two main results. The first one concerns the symmetries of the theory.

\begin{theorem}
  \label{theo intro: Gdiff}
  The group $\GDiff$ of automorphisms of an exact or odd exact Courant algebroid over a compact manifold carries a strong ILH Lie group structure, i.e., it is modelled on an inverse limit of Hilbert manifolds.
\end{theorem}

There are many notions of infinite-dimensional Lie group depending on the model space for the underlying manifold. 
Choosing a particular Banach space as a local model for the charts leads to the notion of a Banach Lie group. Despite enjoying very nice analytical properties,  
Banach Lie groups turn out to be too restrictive: a Banach Lie group acting effectively and transitively on a finite-dimensional compact smooth manifold must be finite-dimensional \cite{Om78}. On the other hand,  Fr\'echet Lie groups give a too large category, a reason being the absence of fundamental results such as the local inverse theorem, Frobenius' theorem or the existence of local solutions to ODEs.
These considerations motivated the introduction of {\it strong ILH Lie groups} by Omori \cite{Om68} in his study of the group of diffeomorphisms of a compact manifold. The existence of a Frobenius' theorem in this category makes strong ILH Lie groups
suitable for the study of automorphisms of Courant algebroids.

In the classical setting, the existence of a strong ILH Lie group structure on the group of diffeomorphisms
of a compact manifold was proved by Omori \cite{Om68}, following the work by Ebin \cite{Eb}.
As a corollary, this result lead to interesting applications in the context of hydrodynamics \cite{EbMa}.
Indeed, the geodesics in the group of diffeomorphisms fixing a prescribed volume form give solutions
to the Euler equation of motion of a perfect fluid. Theorem \ref{theo intro: Gdiff}
could thus be of interest in a different context, as supergravity theories.

As an application of Theorem \ref{theo intro: Gdiff}, we initiate the study of the moduli space of
generalized metrics on a fixed exact or odd exact Courant algebroid $E$ under the action of $\GDiff$. Recall
that a motivation to study the moduli space of Riemannian metrics came from general relativity \cite{Fi}.
Indeed, an interesting class of solutions to the Einstein equations on a space-time $M^3\times\RR$ are interpreted as a particular class of geodesics in the space
of Riemannian structures on $M^3$. A good account of these facts is found in \cite{MaEbFi}. 

In this work, we extend a result by Bourguignon \cite{Bou} showing that the moduli space of Riemannian structures is stratified. Namely, there exists a partition into topological subspaces, or strata, in such a way that if one stratum intersects the closure of another stratum, the first is contained in the closure of the second. We prove the following:

\begin{theorem}
  \label{theo intro: GM}
  The moduli space of generalized metrics on $E$ is stratified by ILH manifolds,
  where strata are labelled by the set of conjugacy classes of generalized isometry groups of generalized metrics. Moreover, the moduli space of generalized metrics projects onto the moduli space of metrics, the preimage of a stratum being a union of strata.
\end{theorem}
The precise relation between the two set of strata is described in Section \ref{sec:relations moduli}.

The proof of Theorem \ref{theo intro: GM} relies on a description of generalized isometry groups and a slice result for the $\GDiff$ action on the space of generalized metrics $\GM$.
To state this construction, denote
by $\rho_{\GM}:\GDiff\times\GM\rightarrow\GM$ the action of $\GDiff$ on $\GM$, and, for each generalized metric $V_+\in\GM$,
denote by $\Isom(V_+)$ its isotropy group for $\rho_{\GM}$.

\begin{theorem}
  \label{theo intro: slice exact case full group}
  Let $V_+$ be a generalized metric on $E$. There exists an ILH submanifold $\cS$ of $\GM$ such that:
  \begin{enumerate}
  \item[a)] For all $F\in \Isom(V_+)$, $\rho_{\GM}(F,\cS)=\cS$.
  \item[b)] For all $F \in \GDiff$, if $\rho_{\GM}(F,\cS)\cap\cS\neq \emptyset$, then $F\in\Isom(V_+)$.
  \item[c)] There is a local cross-section $\chi$ of the map $F \mapsto \rho_{\GM}(F,V_+)$ on a neighbourhood
    $\cU$ of $V_+$ in $\rho_{\GM}(\GDiff, V_+)$ such that the map from $\cU\times\cS$ to $\GM$ given by
    $(V_1,V_2)\mapsto \rho_{\GM}(\chi(V_1),V_2)$ is a homeomorphism onto its image.
  \end{enumerate}
\end{theorem}

The ideas for the proof of Theorem \ref{theo intro: slice exact case full group} go back to Ebin \cite{Eb},
who first obtained a similar result for the action of diffeomorphisms on Riemannian metrics on a given
compact manifold. However, we should emphasize
that, unlike the Lie algebra of vector fields, the Lie algebra of $\GDiff$ is not described as the space
of smooth sections for a vector bundle on $M$, but rather by a subspace of sections satisfying a differential relation.
For this reason, standard elliptic operator theory, as used in \cite{Eb}, cannot be applied
directly. Instead, we need to establish comparison results on operators in the ILH category, which allow us to use an ILH Frobenius' Theorem \cite{Om97}.
It is, to our knowledge, the first example of a slice construction under these constraints.

A natural sequel to our work is the study of subspaces of the moduli of generalized metrics
given by the vanishing of natural curvature quantities, such as the generalized Ricci tensor
or the generalized scalar curvature (see e.g. \cite{GF2}). On the other hand, note that we have focused on the exact and odd exact cases for the sake of simplicity. We believe the same results hold for Courant algebroids of the form $TM+\ad P+T^*M$. The proof of this fact and its applications are work in progress.

The plan of the paper is as follows. In Section \ref{sec:ILHrev} we introduce the basics of ILH geometry and gather the results that we will need throughout the paper. 
In Section \ref{sec:GDIff} we review the definitions of Courant algebroid and generalized diffeomorphisms, 
and prove that the group of generalized diffeomorphisms $\GDiff$ is a strong ILH Lie group 
(Theorem \ref{theo:strong ILH subgroup exact case}). In Section \ref{sec:moduliGM} we 
introduce generalized metrics and prove a slice theorem for the action of $\GDiff$ 
(Theorem \ref{theo:slice exact case full group}). In Section \ref{sec:stratification} we use the slice theorem to describe the ILH stratification of the 
orbit space $\cGR$ of generalized metrics under the $\GDiff$-action (Theorem \ref{theo:Stratification ILH exact case}).  In Section \ref{sec:odd-exact}, we state the main results for odd exact Courant algebroids, focusing on the main differences with the exact case.
Finally, the appendix contains technical results on elliptic operator theory and ILH chains 
that are used in the proofs of Section~\ref{sec:moduliGM}.\\

\textbf{Acknowledgments:} We would like to thank H. Bursztyn and M. Garcia-Fernandez for stimulating discussions. We are very grateful to IMPA for hosting us as long and short term visitors, respectively. Finally, we would like to express our gratitude to the referee for her/his very careful reading of the manuscript.

\section{Review on ILH geometry}
\label{sec:ILHrev}

In this section, we review the basics of ILH geometry and state
some results that will be used throughout the paper. We refer to \cite{Om97} for an extensive introduction to the subject and the proofs of the results.

\subsection{ILH calculus}

We introduce inverse limits of Hilbert spaces as our model space for infinite dimensional manifolds. 
Denote by $\NN(d)$ the set of integers $k\geq d$ for an integer $d$. 
A set of complete locally convex topological vector spaces $$\lbrace \bfE, E^k\st k\in \NN(d) \rbrace$$ is called an  {\bf ILH chain} if, 
for each $k\in \NN(d)$,
\begin{itemize}
\item  $E^k$ is a Hilbert space with norm $\vert\vert\cdot \vert\vert_k$,
\item  $E^{k+1}$ embeds continuously in $E^k$ with dense image,
\item  and $\bfE=\bigcap_{k\in\NN(d)} E^k,$ endowed with the inverse limit topology.
\end{itemize}

Let $\lbrace {\rm F}, F^k\st k\in \NN(d) \rbrace$ be another ILH chain. A linear map $f:\bfE \to {\rm F}$ is called an {\bf ILH map} if, for each $k\in \NN(d)$, the map $f$ extends to a continous linear map $E^k\to F^k$. 

To talk about differentiability, consider  an open subset $U$ of $E^d$. A  map $$f:U\cap \bfE \rightarrow {\rm F}$$  is called a {\bf ${\bf \cC^r}$ ILH map}, for $r\in\NN$, if for each $k\in \NN(d)$, the map $f$ extends to a $\cC^r$ map from $U\cap E^k\to F^k$, also denoted by $f$. 
In order to have an implicit function theorem, we will need some uniform control on the norms of the derivatives $$df:(U\cap E^k)\times E^k \rightarrow F^k,$$
as we now explain. We say that a $\cC^\infty$ ILH map $f$ is a {\bf ${\bf \cC^{\infty,r}}$ ILH normal map} if, for every $x\in U\cap\bfE$, there is a convex neighbourhood $W\subset U$ of $x$, a constant $C$ and polynomials $P_k$, for each $k\in\NN(d)$, with the property that if $y\in\bfE$ is such that $x+y\in W$, then, for every $k\in \NN(d)$ and every integer $1\leq s\leq r$, we have
\begin{equation}
  \label{eq:linear estimates def}
  \begin{split}
    \vert\vert (d^sf)_{x+y}(v_1,\cdots,v_s)\vert\vert_k  \leq {} & C \vert\vert y \vert\vert_k \vert\vert v_1\vert\vert_d\cdots \vert\vert v_s\vert\vert_d\\
    & {} + C \sum_{i=1}^s \vert\vert v_1\vert\vert_d \cdots \vert\vert v_{i-1}\vert\vert_d \vert\vert v_i\vert\vert_k \vert\vert v_{i+1}\vert\vert_d \cdots \vert\vert v_s\vert\vert_d\\
    & {} + P_k(\vert\vert y\vert\vert_{k-1})\vert\vert v_1\vert\vert_{k-1}\cdots \vert\vert v_s\vert\vert_{k-1}.
  \end{split}
\end{equation}

Note that if $f:U\cap \bfE \rightarrow{\rm F}$ is a $\cC^{\infty,r}$ ILH normal map, then, for every $x\in U\cap \bfE$, there are constants $C$ and $D_k$, for each $k\in \NN(d)$, such that the map $\bfA:=df_x$ satisfies,
\begin{equation}
  \label{eq:linear estimates def linear}
  \vert \vert \bfA(x)\vert\vert_k\leq C\vert\vert x\vert\vert_k + D_k\vert\vert x\vert\vert_{k-1}.
\end{equation}
A map $\bfA$ satisfying \eqref{eq:linear estimates def linear} is said to be a {\bf linear ILH normal map}, and estimates like (\ref{eq:linear estimates def}) and (\ref{eq:linear estimates def linear}) are called {\bf linear estimates}.

The implicit function theorem for ILH calculus reads as follows.

\begin{theorem}[\cite{Om97}, I, Thm. 6.9]
  \label{theo:implicitfunctiontheorem}
  Let $U$ and $V$ be open neighbourhoods of zero in $E^d$ and $F^d$, respectively, and let
  $$\Phi: U\cap \bfE \rightarrow V\cap {\rm F}$$
  be a $\cC^{\infty,r}$ ILH normal map for $r\geq 2$. Assume that $\bfA=d\Phi_0$ has a right inverse $\bfB$ that is a linear ILH normal map. Set
  \begin{align}
    \bfE_1 & {} =\ker( \bfA : \bfE \rightarrow {\rm F}),
    &   E_1^k & {} = \ker (\bfA : E^k \rightarrow F^k).
  \end{align}
  Then, $\lbrace \bfE_1, E^k_1, k\in\NN(d)\rbrace$ is an ILH chain with
  \begin{align}
    \bfE & {} =\bfE_1\oplus \bfB{\rm F},& E^k & =E^k_1\oplus \bfB F^k,
  \end{align}
  and there are neighbourhoods $W_1$ and $V'$ of zero in $E_1^d$ and $F^d$, respectively, as well as a $\cC^{\infty,r}$ ILH normal map
  $$
  \Psi : (W_1\cap\bfE_1) \times (V'\cap {\rm F}) \rightarrow \bfB{\rm F},
  $$
  such that $$\Phi(u,\Psi(u,v))=v.$$
\end{theorem}

\subsection{ILH manifolds, groups and actions}
\label{sec:ILHman}

Let $\Man$ be a manifold modelled on a locally convex topological space $\bfE$. We call $\Man$ an {\bf ILH manifold} modelled on the ILH chain $\lbrace \bfE, E^k\st k\in \NN(d)\rbrace$  when the following are satisfied:
\begin{enumerate}
\item[a)] The manifold $\Man$ is the inverse limit of smooth Hilbert manifolds $\Man^k$ modelled on $E^k$ such that $\Man^l\subset \Man^k$ for all $l\geq k$.
\item[b)] For all $x\in\Man$, there exist open charts $(U_k,\phi_k)$ of $\Man^k$ containing $x$ such that if $l\geq k$, $U_l\subset U_k$ and $(\phi_k)_{\vert U_l}=\phi_l$.
\end{enumerate}
Moreover, if the inverse limit $U_\infty$ of $(U_k)_{k\in\NN(d)}$ in b) is an open neighbourhood of $x$ in $\Man$, then $\Man$ is called a {\bf strong ILH manifold}.

A sufficient condition for two different sets of open charts to define the same ILH manifold structure is the following: for each $x\in \Man$ with local open charts  $\{(U_k,\phi_k)\}$, $\{(V_k,\psi_k)\}$, the inverse limit of the maps $\psi_k^{-1}\circ \phi_k$, $\phi_k^{-1}  \circ \psi_k$, which are, by strongness, maps on the neighbourhood $U_\infty\cap V_\infty$,
\begin{align}\label{eq:intersections}
  \psi_\infty^{-1}\circ \phi_\infty, \phi_\infty^{-1}  \circ \psi_\infty : U_\infty\cap V_\infty \to V_\infty\cap U_\infty,
\end{align}
are $\mathcal{C}^{\infty,\infty}$ ILH normal.

A map $\Phi: \Man \rightarrow N$ between ILH manifolds is a {\bf $\cC^r$ ILH map} (of order $j$) if there is $j\in \NN$ such that for all suitable $k$, the map $\Phi$ extends to a $\cC^r$ map 
$\Phi_k: \Man^{j+k} \rightarrow N^k$
satisfying $(\Phi_k)_{\vert \Man^{j+k+1}}=\Phi_{k+1}$. 
If $\Phi$ is a $\cC^r$ ILH map for all $r$, then $\Phi$ is a {\bf smooth ILH map}.

\begin{remark}
  Note that in order to recover the notion of a $\cC^r$ ILH map for ILH chains from a $\cC^r$ ILH map between ILH manifolds, one needs to shift (by $j$) the indices of the ILH chains modelling the ILH manifolds.
\end{remark}

\begin{example}
  Let $(E,h)$ be a smooth Riemannian vector bundle over a compact
  Riemannian manifold $(M,g)$.  Assume that $(E,h)$ admits a metric-compatible connection $\nabla$.  Combining $g$ and $(h,\nabla)$, we
  build metrics $(\cdot,\cdot)_h$ on the bundles $(T^*M)^{\otimes
    p}\otimes E$, as well as compatible connections.  We define the
  norms, for $u\in \Gamma(E)$,
  \begin{equation}
    \label{eq:norms}     
    \vert\vert u \vert\vert_k:= \left( \sum_{i=0}^k \int_M (\nabla^iu,\nabla^iu)_h\:\dvol_g\right)^{\frac{1}{2}},
  \end{equation}
  where $\dvol_g$ is the volume element of $g$ and $\nabla^i$ the
  $ith$ covariant derivative associated to the extended connection on
  $E$.  As an example, given $(M,g)$, we can endow the spaces of forms
  $\Om^i$ or of symmetric tensors $\Gamma(S^2T^*M)$ with Hilbert
  norms. The space of Riemannian metrics on a compact manifold is then
  an example of ILH manifold, modelled on the space of symmetric
  $2$-tensors endowed with these norms.
\end{example}

A strong ILH Lie group $\bfG$ is a topological group with the structure of a   strong  ILH manifold 
such that the group operations are smooth ILH maps. For later purposes, we recall the formal definition from \cite{Om97}, formulated in terms of what happens around the identity at the infinitesimal level. 

\begin{definition}\label{def:ILHLiegroup}
  A {\bf strong ILH Lie group} modelled on an ILH chain $\lbrace \bfE, E^k\st
  k\in \NN(d)\rbrace$ is a topological group $\bfG$ such that there exist open
  neighbourhoods $0 \in U\subset E^d$, $e\in \widetilde{U}\subset
  \bfG$, an homeomorphism called {\em chart at the origin} $\xi:U\cap\bfE\to
  \widetilde{U}$ with $\xi(0)=e$, and an open neighbourhood $V$ of $0$
  in $E^d$ such that $$ \xi(V\cap \bfE)^2 \subset \xi(U\cap\bfE)
  \textrm{ and } \xi(V\cap\bfE)^{-1}=\xi(V\cap\bfE),
  $$ satisfying that, for any $u,v\in V\cap\bfE$, any $k\in\NN(d)$ and any $l\geq 0$,
  \begin{itemize}
  \item[a)] The local product
    $\eta(u,v)=\xi^{-1}(\xi(u)\cdot\xi(v))$ extends to a $\cC^l$ map
    $V\cap E^{k+l}\times V\cap\E^k\to U\cap E^k$.
  \item[b)] The local right translation $\eta_v(u)=\eta(u,v)$
    is a $\cC^\infty$ map $V\cap E^k \rightarrow U\cap E^k$.
  \item[c)] Its differential $\theta(w,u,v)=(d\eta_v)_uw$ extends to a
    $\cC^l$ map $E^{k+l}\times V\cap E^{k+l}\times V\cap E^k\to E^k$.
  \item[d)] The local inverse $\iota(u)=\xi^{-1}(\xi(u)^{-1})$
    extends to a $\cC^l$ map $V\cap E^{k+l}\to V\cap E^{k+l}$.
  \item[e)] For each $g\in G$, the local conjugation
    $A_g(u)=\xi^{-1}(g^{-1}\xi(u)g)$, defined on a neighbourhood $W$ of
    $0\in E^d$ such that $g^{-1}\xi(W\cap \bfE)g\subset \xi(U\cap
    \bfE)$, extends to a $\cC^\infty$ map $W\cap E^k\to U\cap
    E^k$.
  \end{itemize}
\end{definition}

Given a strong ILH Lie group $\rm G$ and $g\in \rm G$, consider the map $\xi_g(u)=\xi(u)g$ for $u\in V$, then
$$\lbrace (W\cap \bfE, \xi_g), g\in \rm G,\, W \textrm{ open in } V \rbrace$$
is an atlas for $\rm G$. One can define Hilbert manifolds $G^k$, with the structure of a topological group, covered by the atlas
$$\lbrace (W\cap E^k, \xi_g), g\in \rm G,\, W \textrm{ open in } V \rbrace$$
such that $\rm G$ is the inverse limit of the topological groups $G^k$ and the axioms of the definition can be stated in terms of the $G^k$s (see \cite[III, Thm. 3.7]{Om97}). We will refer to $\rm G$ as $\lbrace {\rm G}, G^k\st k\in \NN(d) \rbrace$ when necessary.

Hilbert Lie groups, or finite-dimensional Lie groups, are strong ILH Lie groups. As in these cases, the spaces $\bfE$ and $E^k, k\in \NN(d)$, are endowed with Lie algebra structures. We will denote them by $\fg$ and $\fg^k$, which are, respectively the Lie algebras of $\bfG$ and $G^k$ for all $k\in\NN(d)$.

\begin{example}\label{ex:chart-Diff}
  A typical example of a strong ILH Lie group is $\bfDiff$, the space
  of diffeomorphisms of an $n$-dimensional compact manifold $M$ (see, e.g.,
  \cite{Om97}).  This Lie group is modelled on $\lbrace {\Gamma(TM)},
  \Gamma(TM)^k, k\geq n+5 \rbrace$, where ${\Gamma(TM)}$ is the space
  of smooth vector fields on $M$, and $\Gamma(TM)^k$ is its completion
  with respect to the norm defined by  (\ref{eq:norms}). The choice of $k\geq n+5$ is made to make use of Sobolev embedding theorems and deal with regularity issues \cite[Ch. V]{Om97}. Choose a Riemannian metric $g_0$ on $M$. The chart at the origin $(\xi,U)$ is built using the exponential map associated
  to the geodesics provided by $g_0$:
  \begin{equation}
    \label{eq:definitionxiDiff}
    \begin{array}{cccc}
      \xi: & U \cap {\Gamma(TM)} & \rightarrow & \bfDiff\\
      & u & \mapsto & (x \mapsto \exp_x(u(x)) ),
    \end{array}
  \end{equation}
  where $t\mapsto \exp_x(tu(x))$ is the geodesic starting from
  $(x,u(x))$, and $U$ is a small neighbourhood of $0$ in
  $\Gamma(TM)^{n+5}$. Note that by the Sobolev embedding theorem, $\xi$
  is well defined. Moreover, the differentiable structure just defined does not depend on the choice of Riemannian metric \cite[Rem. 4]{Les}.
\end{example}

We will also need the notion of action in this category:
\begin{definition}
  An {\bf ILH (right) action} of a strong ILH Lie group
  $\lbrace {\rm G}, G^k\st k\in \NN(d)\rbrace$ on 
  an ILH manifold $\Man$ is a map
  $$\rho: \Man \times \bfG \rightarrow \Man$$ such that:
  \begin{enumerate}
  \item[a)] For any $v\in \Man$, $g,h\in G$, $\rho(\rho(v,g),h)=\rho(v,gh)$ and   $\rho(v,e)=v$.
  \item[b)] For every $k\in\NN(d)$ and $l\geq 0$, $\rho$ extends to a $\cC^l$ map from
    $\Man^{k+l}\times G^k$ to $\Man^k$.
  \end{enumerate}
  If $\lbrace \Man, \Man^k\st k\in \NN(d)\rbrace$ is simply an ILH chain, and if $\rho$ is linear with respect to $v$, $\rho$ is called
  an ILH representation.
\end{definition}
A fundamental example of an ILH representation of $\bfDiff$ is given by its action on differential forms.
Consider the ILH chain $\lbrace {\Om^j},\Om^{j,k}, k\geq n+5 \rbrace$, where $\Om^{j,k}$
is the completion of the space of smooth $j$-forms ${\Om^j}$ on a compact manifold $M$ with respect to the norm $\vert\vert \cdot \vert\vert_k$
defined as in equation (\ref{eq:norms}). 
Consider the pullback map
\begin{equation}
  \begin{array}{cccc}
    \rho : & \Om^* \times \bfDiff & \rightarrow & \Om^*\\
    &         (\om,g) & \mapsto & g^*\om.
  \end{array}
\end{equation}
Notice the shift in indices in the ILH chain for $\bfDiff$ in the following proposition:
\begin{proposition}[\cite{Om97}, VI, Thm. 6.1, Cor. 6.3]
  \label{prop:Diff action forms}
  The map $\rho$ gives a representation of the strong ILH group $\lbrace \bfDiff, \Diff^{k+1}, k\geq n+5\rbrace$
  on $\lbrace {\Om^*},\Om^{*,k}, k\geq n+5 \rbrace$.
  Moreover, for every fixed $\om\in {\Om^{j}}$, the map 
  $$\Psi_\om: U\cap {\Gamma(TM)} \rightarrow {\Om^j}$$
  defined by
  $\Psi(u)=\rho(\om,\xi(u))$ is a $\cC^{\infty,\infty}$ ILH normal map.
\end{proposition}
We state another result related to the action of
$\bfDiff$ on forms that will be used in Section \ref{sec:GDIff}.
First, for any $u\in U\cap{\Gamma(TM)}$ denote by $D\xi(u)$ the differential of the diffeomorphism $\xi(u)$. 
For any $x\in M$, set $\tau(\exp_xu(x))$ to be the parallel
transport on $\Lambda^j T^*M$ along the geodesic $t\mapsto \exp_x(tu(x))$, from $t=0$ to $t=1$, with respect to the metric induced by $g$.
Then define the map
\begin{equation}
  \label{eq:psioperator}
  \begin{array}{cccc}
    \Psi_{-1} : & {\Om^j} \times U\cap{\Gamma(TM)} & \rightarrow & {\Om^j}\\
    &  (\om,u) & \mapsto & \om( D(\xi(u))^{-1}\tau(\exp_xu(x))\cdot,\ldots, D(\xi(u))^{-1}\tau(\exp_xu(x))\cdot)
  \end{array}
\end{equation}
where more precisely $\Psi_{-1}(\om,u)$ is the form
$$
(x,(V_1,\ldots,V_j))\mapsto\om(x)( D(\xi(u))^{-1}\tau(\exp_xu(x))V_1,\ldots, D(\xi(u))^{-1}\tau(\exp_xu(x))V_j)).
$$
\begin{lemma} \label{lem:psioperator}
  The map $\Psi_{-1}$ is a $\cC^{\infty,\infty}$ ILH normal map from
  $\lbrace {\Om^j},\Om^{j,k} , k\geq n+5\rbrace\times  \lbrace\bfGamma(TM), \Gamma^{k+1}, k\geq n+5 \rbrace$ to
  $\lbrace \bfOm^j,\Om^{j,k} , k\geq n+5 \rbrace$.
\end{lemma}
\begin{proof}
  The proof of Lemma \ref{lem:psioperator} follows from the proof of \cite[VI, Lemma 6.2]{Om97} and
  is an application of \cite[V, Thm. 3.1]{Om97}. The operator $\Psi_{-1}$ here is slightly different from the one denoted $\Psi$ in Omori's book,
  but the argument works equally as the inverse map $D\phi\mapsto (D\phi)^{-1}$ is defined by a smooth fibre preserving bundle map.
\end{proof}
Lastly, a subgroup $\bf H\subset \bfG$ is called a {\bf strong ILH Lie subgroup} of $\lbrace {\rm G}, G^k\st k\in \NN(d) \rbrace$ if, for a chart at the origin $(\xi,U)$,
\begin{enumerate}
\item[a)] there is a decomposition $\fg=\fh\oplus \fm$ which extends to $\fg^k=\fh^k\oplus \fm^k$ for every $k\in\NN(d)$,
  where $\fh, \fh^k$ and $\fm, \fm^k$ are closed subspaces of the spaces $\bfE, E^k$.
\item[b)] There is a neighbourhood $V$ of $0$ in $\fg^d$ and a map $\xi':V\cap \fg \rightarrow \rm G$
  such that $\xi'\xi^{-1}$ is a $\cC^\infty$ ILH diffeomorphism on a neighbourhood of zero and $\xi'(V\cap \fh)\subset \bf H$.
\item[c)] The pair $(\xi'_{V\cap \fh}, V\cap \fh)$ is a chart at the origin for $\bf H$. 
\end{enumerate}

As explained in \cite{Om97}, examples of strong ILH Lie subgroups of $\bfDiff$ are the group of symplectomorphisms
of a given symplectic structure, the subgroup of hamiltonian symplectomorphisms, or the group of
volume preserving diffeomorphisms for a given volume form.

\subsection{ILH Frobenius' theorem}
\label{sec:ILHFrob}
In order to build strong ILH Lie subgroups, we will need to integrate distributions on strong ILH Lie groups.
This can be done by means of a Frobenius' theorem in this category.
We first define vector bundles on strong ILH Lie groups. The definition is based on the properties of the tangent bundle $T\bfG$ of a strong ILH Lie group $\bfG$.
As it happens for  Lie groups, $T\rm G$ is a trivial bundle, thanks to the map
$$dR: \fg\times {\rm G} \rightarrow T\rm G$$ defined by $dR(u,g)=d(R_g)_eu$, where $R_g(h)=hg$ is the right translation.
Considering generalisations of the differential of the right translation map leads to a definition of vector bundles that will be trivial over the ILH Lie group $G$ and carry a suitable ILH structure \cite[Section IX.1]{Om97}:
\begin{definition}
  \label{def:ILH group vector bundles}
  An {\bf ILH vector bundle $B(\bfF,{\rm G},\tilde T)$ over a strong ILH Lie group} $\lbrace {\rm G}, G^k\st k\in \NN(d) \rbrace$  consists of an ILH chain  $\lbrace \bfF, F^k\st k\in \NN(d) \rbrace$ and a {\em defining map} 
  \begin{equation*}
    \begin{array}{cccc}
      \tilde T : & \bfF \times (\tilde U \cap {\rm G}) \times (\tilde U\cap {\rm G}) & \rightarrow & \bfF\\
      &  (u,g,h)                                         &  \mapsto   &    \hat T(g,h)u
    \end{array}
  \end{equation*}
  for $\tilde U$ a neighbourhood of $e$ in ${\rm G}$, linear on $F$ and satisfying, for any $k\in\NN(d)$, for any $l\geq 0$ and $(g,h,h')\in {\rm G}^3$ :
  \begin{enumerate}
  \item[a)] $\hat T(g,e)=Id$.
  \item[b)] $\hat T(gh,h')\hat T(g,h)=\hat T(g,hh')$.
  \item[c)] $\tilde T$ extends to a $\cC^l$ map $F^{k+l}\times \tilde U\cap G^{k+l} \times \tilde U\cap G^k\to F^k$.
  \item[d)] $(u,g)\mapsto \tilde T(u,g,h)$ is a $\cC^\infty$ map $F^k\times \tilde U\cap G^k\to F^k$.
  \end{enumerate}
\end{definition}
\begin{remark}
  This definition deserves some explanation.
  Fix a triple $(\bfF,{\rm G},\tilde T)$ as in Definition \ref{def:ILH group vector bundles},  an open neighbourhood $V$ of $0$ in $E^d$, as in Definition \ref{def:ILHLiegroup}, 
  and set $\tilde V=\xi(V)$. One can define transition functions from the chart $(\tilde V\cap {\rm G}) g$
  to the chart $(\tilde V\cap {\rm G})h$, via the maps:
  \begin{equation}
    \label{eq:charts}
    \begin{array}{cccc}
      t_{h,g} : &   (\tilde V\cap {\rm G})g \times \bfF & \rightarrow & \bfF \\
      &    (xg,u) & \mapsto & \hat T(x,gh^{-1})u  
    \end{array}
  \end{equation}
  One can check that the family $\lbrace t_{h,g} , (g,h)\in {\rm G}^2\rbrace$ defines a $\cC^\infty$ vector bundle 
  $B(F^k,G^k,\tilde T)$ on $G^k$ for all $k\in \NN(d)$, with fibre $F^k$ (in the Hilbert category). Then $B(\bfF,{\rm G},\tilde T)$
  is the projective limit of these bundles.
\end{remark}

As an example, recall the maps $\theta$ and $\xi$ from the definition  of a strong ILH Lie group ${\rm G}$.
Then \begin{equation}
  \label{eq:defining map tangent}
  \tilde T_\theta(u,g,h)=\theta(u,\xi^{-1}(g),\xi^{-1}(h))
\end{equation}
satisfies the above axioms,
and $B(\bfE,{\rm G},\tilde T_\theta)$ is the tangent bundle of ${\rm G}$. 

For any $g\in {\rm G}$, the right translation map $R_g$ on ${\rm G}$ extends to a fibre preserving map 
$\tilde R_g$ on $B(\bfF,{\rm G},\tilde T)$ (see \cite[IX]{Om97}).
Consider another ILH vector bundle $B(\bfF', {\rm G}, \tilde T')$ for an ILH chain $\lbrace \bfF', F_2^k, k\in\NN(d) \rbrace$. Then any linear operator $\bfA: \bfF \rightarrow \bfF'$
can be extended to an ILH bundle homomorphism
$$
\tilde \bfA : B(\bfF, {\rm G}, \tilde T) \rightarrow B(\bfF', {\rm G}, \tilde T')
$$
by setting $\tilde \bfA= \tilde R_g \bfA \tilde R_g^{-1}$ on the fibre over $g\in \rm G$. However, even if $\bfA$ is continuous, $\tilde \bfA$ may have little
regularity.
The regularity of $\tilde \bfA$ is checked by means of its {\it local expression}.
By definition, if $(U,\xi_G)$ is a chart at the origin for ${\rm G}$, the local expression of $\tilde \bfA$ is the map:
\begin{equation}
  \label{eq:local expression}
  \begin{array}{cccc}
    \Phi_\bfA: & U\cap\fg \times \bfF & \rightarrow & \bfF' \\
    &   (  u , w )           & \mapsto     & \hat T'(e,\xi_G(u))\circ \tilde \bfA \circ \hat T(e,\xi_G(u))^{-1}w.
  \end{array}
\end{equation}
If $\Phi_\bfA$ satisfies some linear estimates, when taking derivatives with respect to
its first variable, we say that $\tilde \bfA$ is a {\bf $\cC^{\infty,r}$ ILH normal bundle homomorphism}.
We refer to \cite[IX, Def. 1.3]{Om97} for the precise definition.

Frobenius' theorem in this context is given by the following result.
\begin{theorem}{(Frobenius' theorem), \cite[IX, Thm. 3.4]{Om97}}
  \label{theo:Frobenius}
  Let ${\rm G}$ be a strong ILH Lie group, with tangent bundle $B(\fg,{\rm G},\tilde T_\theta)$. 
  Let $B(\bfF,\bfG, \tilde T')$ be another ILH-vector bundle over ${\rm G}$ and let
  $$
  \tilde \bfA : B(\fg,{\rm G},\tilde T_\theta) \rightarrow B(\bfF,\bfG, \tilde T')
  $$
  be a right invariant $\cC^{\infty,r}$ ILH normal bundle homomorphism, with $r\geq 1$.
  Suppose that the restriction $\bfA:\fg \rightarrow \bfF$ of $\tilde \bfA$ to the fibre at the identity
  satisfies for all $k\in \NN(d)$:
  \begin{enumerate}
  \item[a)] The subspace $\fh = \ker \bfA$ is a Lie subalgebra of $\fg$, and $\bfA\fg$ is a closed subspace of $\bfF$.
  \item[b)] There are closed subspace $\bfE_2$ and $\bfF_2$ such that $\fg=\fh\oplus \bfE_2$ and $\bfF=\bfA\fg\oplus \bfF_2$.
  \item[c)] There exist constants $C$ and $D_k$, for $k\in\NN(d)$ with $D_d=0$, such that, for all $u\in \bfE_2$,
    $\vert\vert \bfA u\vert\vert_k\geq C\vert\vert u\vert\vert_k - D_k \vert\vert u \vert\vert_{k-1}.$
  \item[d)] There exist constants $C'$ and $D'_k$, for $k\in\NN(d)$ with $D'_d=0$, such that, the projection $p:\bfF \rightarrow \bfA\fg$ with respect to the  decomposition in b) satisfies, for all $v\in\bfF$,
    $\vert\vert pv\vert\vert_k\leq C' \vert\vert v\vert\vert_k + D_k'\vert\vert v\vert\vert_{k-1}.$
    
  \end{enumerate}
  Then, there are neighbourhoods $0\in V_1\subset\fh^d$, $0\in V_2\subset E_2^d$, $e\in W'\subset G^d$
  and a $\cC^\infty$
  diffeomorphism $\xi':V_1\cap\fh\times V_2\cap \bfE_2\rightarrow W'\cap \bfG$ satisfying, for all $k\in\NN(d)$:
  \begin{enumerate}
  \item[a)] The map $\xi'$ extends to a $\cC^\infty$ diffeomorphism
    $
    (V_1\cap\fh^k)\times (V_2\cap E_2^k) \rightarrow W'\cap G^k.
    $
  \item[b)] For every $w\in V_2\cap E_2^k$, $\xi'((V_1\cap \fh^k) \times \lbrace w \rbrace)$ is an integral  submanifold of the involutive subbundle $\fh^k=\ker(\tilde \bfA: B(\fg^k,G^k,\tilde T_\theta)\rightarrow B(F^k,G^k,\tilde T'))$.
  \item[c)] The kernel $\tilde \fh=\ker \tilde \bfA$ is an ILH subbundle of $B(\fg,\bfG,\tilde T_\theta)$, and the maximal integral submanifold
    $\bfH$
    of $\tilde \fh$ through $e$ is a strong ILH Lie subgroup of $\bfG$.
  \end{enumerate} 
\end{theorem}
\begin{corollary}
  \label{cor:Frobenius}
  In the notation of Theorem \ref{theo:Frobenius}, $\xi'(\lbrace 0 \rbrace \times (V_2\cap \bfE_2) ) $
  provides a slice (at the origin) to the integral submanifolds of $\tilde \fh$. Thus
  the quotient $\bfG/\bfH$ naturally inherits an ILH manifold structure.
\end{corollary}

\begin{remark}\label{rem:hypo+trivial}
  Note that in the case of the trivial bundle $\R^d\times {\rm G}$, hypothesis (d) of Theorem \ref{theo:Frobenius} is trivially satisfied, as well
  as the fact that $\bfA \fg$ is closed and admits a closed complement.
\end{remark}

\begin{remark}
  Let $G$ be a Banach Lie group and $K$ a subbundle of $TG$. To prove Frobenius' theorem in the Banach setting, one describes the subbundle $K$ of $TG$
  in a local chart $x\in V\subset T_x$ by 
  $$
  K_y=\lbrace w + J(y)w, w\in K_x \rbrace, \: y\in V
  $$
  with $y \mapsto J(y)\in \cL(K_x,K_x')$ a smooth map, where $K_x'$ is a closed complement of $K_x$. In the setting of Theorem \ref{theo:Frobenius},
  the map $J$ is given by $J(u)w=-\chi(u)\circ p \circ \Phi(u)w$, with $\Phi$ the local expression of $\tilde A$ at $e$ and
  $\chi(u)$ the inverse for $p\circ \Phi(u): \bfE_2 \rightarrow \Im(\bfA)$. The hypotheses of the theorem
  ensure that the map $J$ is $\cC^{\infty,r}$ normal. One can apply Frobenius' theorem for the groups $G^k$ for each $k$, and obtain a local integration
  of $\ker A_k$. The normality condition ensures that this local result remains in the inverse limit procedure.
\end{remark}

Next, we settle some examples of bundles and right-invariant $\cC^{\infty,r}$ ILH normal bundle homomorphisms
that will be used in Section \ref{sec:GDIff}.
We consider $\bfDiff$ endowed with its strong ILH Lie group structure built by means of the exponential map. Recall
the open set $V\subset  \Gamma(T)^d$, the maps $\xi, \eta$ as in Definition \ref{def:ILHLiegroup}, and set $\tilde V=\xi(V)$.
Let $F$ be a smooth Riemannian bundle over $M$, endowed with a metric compatible connection.
With these data, we can consider the ILH chain $\lbrace {\Gamma(F)}, \Gamma(F)^k, k\geq n+5\rbrace$
where $\Gamma(F)^k$ is the completion of ${\Gamma(F)}$ with respect to the norm
defined as in equation (\ref{eq:norms}).
\begin{definition}
  \label{def:vector bundle parallel}
  The vector bundle $B({\Gamma(F)},\bfDiff,\tilde T_F)$ is defined to be the bundle over $\bfDiff$, with fibre ${\Gamma(F)}$
  and function $\tilde T_F$:
  \begin{equation}
    \label{eq:tilde TF}
    \begin{array}{cccc}
      \tilde T_F : &{\Gamma(F)} \times (\tilde V\cap \bfDiff) \times (\tilde V\cap \bfDiff) & \rightarrow & {\Gamma(F)}
    \end{array}
  \end{equation}
  with
  $$
  \tilde T_F(w,\xi(u),\xi(v))(x) =  \tau(\exp_x\eta(u,v)(x))^{-1}\tau(\exp_{\exp_x(v(x))} (u(\exp_xv(x))))w(\exp_x(v(x))),
  $$
  where $\tau(\exp_yu(y))$ denotes the parallel transport along the geodesic $t\mapsto \exp_y(tu(y))$ from
  $t=0$ to $t=1$.
\end{definition}
Let $F'$ be another Riemannian bundle. Then from any continuous linear operator 
$$\bfA: {\Gamma(F)} \rightarrow {\Gamma(F')},$$
one can build a right-invariant bundle morphism 
$$
\tilde \bfA : B({\Gamma(F)},\bfDiff,\tilde T_F) \rightarrow B({\Gamma(F')},\bfDiff,\tilde T_{F'})
$$ 
by setting $\tilde \bfA=\tilde R_g \bfA\tilde R_g^{-1}$ on the fibre over $g\in \Diff$.
We gather the following results from \cite[IX, Thm. 4.3, Thm. 5.6, Prop. 6.1]{Om97}:

\begin{theorem}
  \label{theo:extension theorem}
  Let $\bfA$, $F$ and $F'$ be as before. Then,
  \begin{enumerate}
  \item[a)] 
    If there is a bundle homomorphism $\gamma$ such that $(\bfA u)(x)=\gamma(u(x))$ for any $u\in\Gamma(F)$, 
    then $\tilde \bfA$ is a $\cC^{\infty,\infty}$ ILH normal bundle homomorphism.
  \item[b)] 
    If $\bfA$ is a differential operator of order $r$, then, up to a shift of indices,
    $\tilde \bfA$ is a $\cC^{\infty,\infty}$ ILH normal bundle
    homomorphism. More precisely, $\tilde \bfA$ extends as a smooth bundle homomorphism from
    $B(\Gamma^{k+r}(F),\Diff^{k+r},\tilde T_F)$ to $B(\Gamma^{k}(F'),\Diff^{k+r},\tilde T_{F'})$.
  \item[c)] 
    Let $h\in\Gamma(F)$ and assume that 
    $$
    \bfA(u)=\langle u, h \rangle_{L^2_g}.
    $$
    Then $\tilde \bfA$ is a $\cC^{\infty,\infty}$ ILH normal bundle homomorphism from
    $B({\Gamma(F)},\bfDiff,\tilde T_F)$ to the trivial bundle $\bfDiff\times \RR$.
  \end{enumerate}
\end{theorem}

\section{The automorphism group of an exact Courant algebroid}
\label{sec:GDIff}

In this section, we show that the group of automorphisms $\Aut(E)$ of an exact Courant algebroid $E$ over a compact $n$-dimensional manifold $M$ carries the structure of a strong ILH Lie group. We will also consider the subgroup of exact automorphisms.

\subsection{Courant algebroids, automorphisms and splittings}
\label{sec:definition courant alg}
A \textbf{Courant algebroid}  over a
manifold $M$ is a tuple $(E,\la\cdot,\cdot\ra,[\cdot,\cdot],\pi)$ consisting of a vector bundle $E\to M$ together with a
non-degenerate symmetric bilinear form $\la\cdot,\cdot\ra$ on $E$, a Dorfman
bracket $[\cdot,\cdot]$ on the sections $\Gamma(E)$
and a bundle map $\pi:E\to TM$ such that the following properties
are satisfied for any $e,e',e'' \in \Gamma(E)$:
\begin{itemize}
\item[(C1):] $[e,[e',e'']]=[[e,e'],e''] + [e',[e,e'']]$,
\item[(C2):] $\pi(e)\la e', e'' \ra = \la [e,e'], e'' \ra + \la e', [e,e'']
  \ra$,
\item[(C3):] $[e,e]=D\la e,e\ra$, or, equivalently, $[e,e']+[e',e']=2D\la e,e'\ra$,
\end{itemize}
with $D:\cC^\infty(M)\to \Gamma(E)$ defined, for $\phi\in \cCi(M)$, by $D\phi=\pi^*d\phi$, where we use $\la,\ra$ to identify $E^*$ and $E$. Note that, as a consequence of (C2), we also have the properties \cite{Uchino} \begin{itemize}
\item[(C4):] $[e,\phi e']=\phi[e,e']+(\pi(e)\phi)e'$,
\item[(C5):] $\pi([e,e'])=[\pi(e),\pi(e')]$.
\end{itemize}

\begin{example}\label{ex:Courant-algebroid}
  The best-known example of a Courant algebroid is, for a closed $3$-form
  $H\in \Omega^3_{cl}$, the tuple
  \begin{equation}
    (TM+T^*M)_H :=(TM+T^*M, \la\cdot,\cdot\ra, [\cdot,\cdot]_H,\pi),\label{eq:def-Courant-H}
  \end{equation}
  where the pairing is given by
  $\la u+\al,u+\al\ra = i_u\al$, the anchor map $\pi$ is the projection to $TM$, and the Dorfman bracket is given by
  $$[u+\al,v+\gamma]_H=[u,v]+\cL_u\gamma - i_v d\al + i_ui_v H.$$ These Courant algebroids are the framework for Dirac
  structures, which encompass presymplectic and (possibly twisted) Poisson geometry.
\end{example}

\begin{remark}
	Note that,  for the sake of simplicity,  we use the notation $TM+T^*M$ for the Whitney sum $TM\oplus T^*M$.
\end{remark}

Denote by $\OO(E)$ the group consisting of smooth bundle maps $F:E\to E$, covering a diffeomorphism $\phi:M\to M$, such that, for $u,v\in \Gamma(E)$, 
$$\la F(u),F(v) \ra = \phi_* \la u,v\ra.$$
The subgroup of anchor-preserving orthogonal maps is given by
$$ \OO_\pi(E)  = \{ F\in \OO(E) \st \pi(F(u)) = \phi_*\pi(u) \textrm{ for } u\in \Gamma(E)\}.$$
The \textbf{automorphism group} $\Aut(E)$ of the Courant algebroid $E$ consists of those maps that are moreover bracket-preserving:
$$\Aut(E) = \{ F\in \OO(E) \st \pi(F(u)) = f_*\pi(u),\; [F(u),F(v)]=F([u,v]),  \textrm{ for } u,v\in \Gamma(E) \}.$$

Finally, we denote by $\cG(E)$ the subgroup of $\Aut(E)$ covering the identity diffeomorphism.

\begin{example}\label{ex:first-groups}
  In Example \ref{ex:Courant-algebroid}, we have $\OO_\pi((TM+T^*M)_H)=\Diff
  \ltimes \Omega^2$ for any $H$. For $H=0$, we have $\Aut(TM+T^*M)=\Diff
  \ltimes \Omega^2_{cl}$, and $\cG(TM+ T^*M)=\Omega^2_{cl}$, where $\Omega^2_{cl}$ denotes the closed $2$-forms. A diffeomorphism $\phi\in
  \Diff$ acts by pushforward $\phi_*$ both in vector fields and forms,
  and $B\in \Omega^2_{cl}$, known as a $B$-field, acts by $$u+\al \to
  u+ \al + i_uB.$$ The group product in $\OO_\pi(TM+ T^*M)$ and $\Aut(TM+
  T^*M)$ is given by
  \begin{equation}
    (\phi,B)(\psi,B')=(\phi\circ\psi,\psi^* B + B'),\label{eq:product-rule}
  \end{equation}
  whereas $\cG(TM+ T^*M)$ is the usual abelian additive group of
  $2$-forms.
\end{example}

We say that two Courant algebroids $(E,\la\cdot,\cdot\ra,[\cdot,\cdot],\pi)$ and $(E',\la\cdot,\cdot\ra',[\cdot,\cdot]',\pi')$ are isomorphic when there exists an isomorphism of vector bundles $F:E\to E'$ satisfying, for $e_1,e_2\in \Gamma(E)$, $$\pi'\circ F = \pi, \qquad \la Fe_1, Fe_2 \ra'=\la e_1, e_2\ra, \qquad [Fe_1,Fe_2]'=F[e_1,e_2].$$

By considering $\pi^*:T^*M\to E^*$ and identifying $E^*\cong E$ using the non-degenerate pairing, we always have a sequence
\begin{equation*}
  T^*M \to E \to TM
\end{equation*}
associated to any Courant algebroid. We say that $E$ is an \textbf{exact Courant algebroids} when this sequence is exact, i.e., 
\begin{equation}
  0\to  T^*M \to E \to TM\to 0.\label{eq:exact-Courant}
\end{equation}

It is always possible to split the sequence \eqref{eq:exact-Courant} as a sequence of vector bundles. The splitting $\s:TM\to E$ can be chosen to be isotropic, meaning that $\s(TM)\subset E$ is an isotropic subbundle. Indeed, for any splitting $\s'$, the splitting $\s:u\mapsto \s'(u) - \frac{1}{2}\pi^* \la \s'(u),\cdot\ra$ is isotropic. The \textbf{space of isotropic splittings} of a Courant algebroid
\begin{equation}
  \Spl := \{ \s:TM\to E \st \s \textrm{ is injective and } \s(TM)\subset E \textrm{ is isotropic} \}\label{eq:space-of-splittings}
\end{equation}
is an affine space modelled on $\Omega^2$, i.e., an $\Omega^2$-torsor, where the action, for $B\in \Omega^2$, which we simply denote by $\s+B$, is given by 
$$(\s+B):u\mapsto \s(u)+\pi^*(i_u B).$$

Choosing an isotropic splitting  $\s:TM\to E$ helps us know $E$ better, as it provides a Courant algebroid isomorphism
\begin{equation}
  \begin{array}{ccc}
    E &\simeq_\s &(TM+T^*M)_H\\
    \lambda(u)+\al & \leftrightarrow & u+\al,\phantom{M}
    \label{eq:iso-splitting}
  \end{array}
\end{equation}
where the $3$-form $H$ of Example \ref{ex:Courant-algebroid} is given, for $u,v,w\in \Gamma(TM)$, by $$H(u,v,w)=\la [\s(u),\s(v)], \s(w)\ra.$$

By the isomorphism \eqref{eq:iso-splitting} and Example \ref{ex:first-groups}, the group $\OO_\pi(E)$ is always isomorphic to $\Diff\ltimes \Omega^2$. The subgroup $\Omega^2$ sits naturally inside $\OO_\pi(E)$, acting by $C(e)=\pi^*(i_{\pi(e)}C)$ for $C\in \Omega^2$, $e\in \Gamma(E)$, whereas we have a natural projection to the diffeomorphisms, yielding the sequence
\begin{equation}
  0\to \Omega^2 \to \OO_\pi(E) \to \Diff \to 0.\label{eq:exact-seq-OpiE}
\end{equation}

The way we regard the diffeomorphisms inside  $\OO_\pi(E)$ depends on the splitting of $E$: if, under a splitting $\s$,  the element $F\in \OO_\pi(E)$ corresponds to the diffeomorphism $(\phi,0)$, under a different splitting $\s+C$, the element $F$ corresponds, for the product rule given in \eqref{eq:product-rule}, to 
\begin{equation}
  (\Id,-C)(\phi,0)(\Id,C)=(\phi,C-\phi^*C).\label{eq:conjugation-splitting}
\end{equation}

To describe $\Aut(E)$, we look first at $\Aut((TM+T^*M)_H)$, which we denote by $\GDiff_H$ and whose elements will be referred to as generalized diffeomorphisms. We have
\begin{equation}
  \GDiff_H= \{ (\phi,B)\in \Diff\ltimes \Omega^2 \st \phi^*H-H=dB \},\label{eq:defining equation exact case}
\end{equation}
whose elements act on the left on $\Gamma((TM+T^*M)_H)$ by
$$u+\al\mapsto \phi_* u + \phi_*(\al + i_u B).$$
By \eqref{eq:defining equation exact case}, the map $\phi$ belongs to the diffeomorphisms preserving the cohomology class $[H]$, which we denote by $\Diff_{[H]}$, and the elements covering the identity are just closed forms. The automorphism group $\Aut(E) \subset \OO_\pi(E)$ is thus an extension 
\begin{equation*}
  \label{eq:ext exactcase}
  0 \to \Omega^2_{cl} \to \Aut(E) \to \Diff_{[H]}\to 0,
\end{equation*}
which, unlike \eqref{eq:exact-seq-OpiE}, is not a semidirect product (unless $[H]=0$).

\begin{remark}
 Note that a different splitting $\lambda+B$ yields, via \eqref{eq:iso-splitting}, the $3$-form $H+dB$. Indeed, equivalence classes of exact Courant algebroids up to isomorphism are parametrized by $[H]\in H^3(M,\R)$, known as the \v{S}evera class of the Courant algebroid.
\end{remark}

\subsection{The ILH Lie group structure}

We can endow the groups $O_\pi(E)$ and $\Aut(E)$ with the $\cCi$-Whitney topology, and we will show that they carry strong ILH Lie group structures. First, we show that $\OO_\pi(E)$ carries such a structure and then we will use the implicit function Theorem \ref{theo:implicitfunctiontheorem} for $\GDiff_{H}$, in a similar fashion as one does for the group of symplectomorphisms \cite{Om97}.
\begin{remark}
  We should mention other important categories of infinite dimensional Lie groups that could have been used in the context of generalized geometry. Tamed Fr\'echet Lie groups were introduced by Hamilton \cite{Ha}. These groups enjoy the Nash-Moser inverse function theorem. However this theorem
  requires the existence of a local tamed family of inverses for the differential of a map to be invertible. In the ILH category,
  one only needs to invert the differential of a map at a single point to apply an inverse function theorem.
  Another category is the one of convenient Lie groups, as developed by Michor and Kriegl \cite{KrMi}. This category enables one to deal with non-compact situations. However, to our knowledge, there is no general Frobenius'-type theorem for these groups.
  We refer to \cite{Ne} for an introduction to infinite-dimensional Lie groups.
\end{remark}
Throughout this section, we fix a Riemannian metric $g_0$  on $M$, giving the chart at the origin $(\xi,U)$ for the group $\bfDiff$, as in Example \ref{ex:chart-Diff}.

\begin{proposition}\label{prop:ILH-on-Opi}
  The group $O_\pi(E)\simeq_\s \Diff\ltimes \Omega^2$ carries a strong ILH Lie group structure modelled on
  $\lbrace \bfGamma(TM)\times\bfOm^2, \Gamma(TM)^{k+1}\times\Om^{2,k}, k\geq n+5 \rbrace$.  A chart at the origin $(\xi',U')$ on $U'=U\times \Om^{2,n+5}$ is given by
  \begin{equation}
    \label{eq:chart exact case}
    \begin{array}{cccc}
      \xi': & U'\cap(\bfGamma(TM) \times \bfOm^2) & \rightarrow &  \bfDiff\ltimes \bfOm^2 \\
      & (u,b) & \mapsto & (\xi(u),b).
    \end{array}
  \end{equation}
  Moreover, the ILH structure is independent of the choice of splitting $\s$, and its Lie algebra is given by $\Gamma(TM)\times \Omega^2$ with the bracket, for $(X,b),(Y,c)\in  \Gamma(TM)\times \Omega^2$, 
  $$ [(X,b),(Y,c)]=([X,Y],\cL_X c - \cL_Y b).$$
\end{proposition}
\begin{proof}
  The proof of the first part of the proposition follows directly from Proposition   \ref{prop:Diff action forms} and the fact that $\bfDiff$ is a strong ILH Lie group.
  For the independence of the choice of splitting, consider, at a point $x$, charts $\{U_k,\phi_k)\}$, $\{V_k,\psi_k)\}$ for the splittings $\s$ and $\s+C$. The map $\phi_k^{-1}\circ \psi_k$ and $\psi_k^{-1}\circ \phi_k$ in (\ref{eq:intersections}) are given by conjugation by $C$ and $-C$,
  \begin{equation}
    (\Id,-C)(\phi,B)(\Id,C)=(\phi,B+C-\phi^*C),\label{eq:conjugation-splitting-2}
  \end{equation} which are  $\mathcal{C}^{\infty,\infty}$ ILH normal by Proposition \ref{prop:Diff action forms}, and hence the ILH structure does not depend on the splitting. The expression for the bracket follows from the action 
  $$(X,b)(Y+\eta)=\cL_X (Y+\eta) - i_Yb.$$
\end{proof}

As we have shown the independence from the choice of splitting, in what follows we will work with $(TM+T^*M)_H$ and $\GDiff_H$.  In the remaining of this section, the ILH chains associated to the spaces $\bfOm^2$, $\bfGamma(TM)$ and $\bfOm^3$
are $\lbrace \bfOm^2, \Om^{2,k}, k\geq n+5 \rbrace$, $\lbrace \bfGamma(TM), \Gamma^{k+1}, k\geq n+5 \rbrace$ and
$\lbrace \bfOm^3, \Om^{3,k-1}, k\geq n+5 \rbrace$.

We will build the strong ILH Lie group $\bfGDiff_{H}$ as a strong ILH Lie subgroup of $\bfDiff \ltimes \bfOm^2$.
Introduce the map:
\begin{equation}
  \label{eq:defining map exact case}
  \begin{array}{cccc}
    \tilde\rho : & \bfOm^3 \times (\bfDiff \ltimes \bfOm^2)  & \rightarrow & \bfOm^3\\
    &  (H', (\phi,B)) & \mapsto & \phi^*H'-dB
  \end{array}
\end{equation}
Then $\GDiff_{H}=\tilde\rho(H,\cdot)^{-1}(H)$.
Define, using $(\xi',U')$ from Proposition \ref{prop:ILH-on-Opi},
\begin{equation*}
  \begin{array}{cccc}
    \Phi : &   U'\cap  (\bfDiff \ltimes \bfOm^2) & \rightarrow & \bfOm^3 \\
    & (u,b) & \mapsto &\tilde\rho(H,\xi'(u,b)).
  \end{array}
\end{equation*}
from the ILH chain $\lbrace  \bfGamma(TM)\times\bfOm^2, \Gamma^{k+1}\times\Om^{2,k}, k\geq n+5 \rbrace$ to
$\lbrace \bfOm^3, \Om^{3,k-1}, k\geq n+5 \rbrace$.

From Proposition \ref{prop:Diff action forms},
the map $\Phi$ is a $\cC^{\infty,\infty}$ ILH normal map from $ U\cap\bfDiff \times \bfOm^2$ to $\bfOm^3$,
with respect to these ILH structures. Using $(d\xi)_0=Id$, the derivative of $\Phi$ at zero is
\begin{equation}
  \label{eq:dphi exact case}
  d\Phi_0 (u,b)=d(\iota_uH - b).
\end{equation}
In order to apply the implicit function theorem, we need a right inverse
for $d\Phi_0$ in the category of $\cC^{\infty,2}$ ILH normal maps.
Consider the map
\begin{equation}\label{eq:right-inverse}
  \begin{array}{cccc}
    \bfB: & \bfOm^3 & \rightarrow & \bfGamma(TM)\times \bfOm^2\\
    &  h    & \mapsto     & (0,-d^*\bbG h),
  \end{array}
\end{equation}
where $d^*$ is the adjoint of $d$ with respect to the $L^2$ inner pairing on forms provided by $g_0$, and
$\bbG$ is the Green operator for the Hodge Laplacian.
In the following lemma, we consider the space of exact $3$-forms $d\Om^2$ as a subspace of $\Om^3$ endowed with the ILH chain structure $\lbrace \bfOm^3, \Om^{3,k-1}, k\geq n+5 \rbrace$.
\begin{lemma}
  \label{lem:B right inverse exact case}
  The restriction of the operator $\bfB$ to the space of exact $3$-forms $d\Om^2$ is a linear ILH normal map that is a right-inverse for the map $d\Phi_0$.
\end{lemma}
For a proof of this lemma, in particular the linear estimates, see 
\cite[VII, Lemma 5.10]{Om97}. We can then apply the implicit function theorem, Theorem \ref{theo:implicitfunctiontheorem}, and we have thus proved the following.
\begin{theorem}
  \label{theo:strong ILH subgroup exact case}
  The group $\Aut(E)$ is a strong ILH Lie subgroup of $\OO_\pi(E)$, or, equivalently in terms of a splitting, the group $\bfGDiff_{H}$ is a strong ILH Lie subgroup of $\bfDiff \ltimes \bfOm^2$.
\end{theorem}

The Lie algebra of $\GDiff_H$, or algebra of derivations of $(TM+T^*M)_H$, coincides with the kernel of $d\Phi_0$ in \eqref{eq:dphi exact case} and is explicitly given by 
\begin{equation*}
  \label{eq:Lie algebra exact case}
  \gdiff_{H}=\lbrace (u,b)\in \Gamma(TM)\times \Omega^2\; \vert\: d(\iota_uH-b)=0 \rbrace.
\end{equation*}
The subalgebra $\gdiff_{H}^e\subset \gdiff_{H}$ of exact derivations is given by
\begin{equation}
  \label{eq:liegdiffexact exact case}
  \gdiff_{H}^e=\lbrace (u,b)\in  \Gamma(TM)\times \Om^2 \:\vert\:  \iota_uH-b=da \textrm{ for some }  a\in \Om^1   \rbrace.
\end{equation}
This subalgebra integrates to the subgroup of exact generalized diffeomorphisms $\GDiff_H^e$. 

If we start by an exact Courant algebroid $E$ with two different splittings $\lambda$, $\lambda+C$, so that
$$\GDiff_H \simeq_\s \Aut(E) \simeq_{\s+C} \GDiff_{H+dC},$$
the Lie algebras of the first and third groups are isomorphic by $(u,b)\mapsto (u,b+i_udC)$, which interchanges exact derivations. Hence, the Lie algebra of exact derivations $\Der^e(E)$ and the subgroup $\Aut^e(E)$ of exact automorphisms are well defined. 

We will show that $\Aut^e(E)$  is a strong ILH Lie subgroup of $\Aut(E)$ by fixing a splitting $\s$, so that we work with $\GDiff_H$, using the simpler version of Frobenius' Theorem \ref{theo:Frobenius}, stated in Remark \ref{rem:hypo+trivial}, and following the ideas in \cite{Om97} for the group of  exact symplectomorphisms.

As a subalgebra of $\gdiff_{H}$, the algebra $\gdiff^e_H$ can be described as the space of pairs $(u,b)$ such that $\iota_uH-b$ is orthogonal to the harmonic $2$-forms for the fixed metric $g_0$ on $M$. Let $s=\dim H^2(M)$, and set $\{e_i\}_{1\leq i\leq s}$ a basis of the harmonic $2$-forms.
Define a map:
\begin{equation}
  \begin{array}{cccc}
    I : & \bfOm^2 & \rightarrow & \RR^s\\
    &  \om  &  \mapsto    &  \langle \om , e_i \rangle_{L^2_{g_0}}
  \end{array}
\end{equation}
where $\langle \cdot,\cdot \rangle_{L^2_{g_0}}$ denotes the $L^2$ inner product on forms defined by $g_0$.
Set also

\begin{equation}
  \begin{array}{cccc}
    \kappa : & \gdiff_{H} & \rightarrow & \bfOm^2 \\
    &  (u,b)   & \mapsto     &  \iota_uH-b.
  \end{array}
\end{equation}
Note that $\gdiff_H^e=(I\circ\kappa)^{-1}(0)$. 
We will consider an extension of $I\circ\kappa$ to $T\bfGDiff_{H}$ and show
that this extension is a $\cC^{\infty,\infty}$ ILH normal homomorphism to the trivial bundle over $\bfGDiff_H$.
Let $V$ be an open neighbourhood of zero in $\Gamma(TM)^{n+5}$ as in Definition \ref{def:ILHLiegroup}
and set $\tilde V=\xi(V)$. In order to obtain a simple description of $T\bfGDiff_{H}$, and in view of the composition law (\ref{eq:product-rule}), we define the map
\begin{equation}
  \begin{array}{cccc}
    T_{\bfOm^2} : & \bfOm^2\times (\tilde V\cap\bfDiff) \times (\tilde V\cap \bfDiff) & \rightarrow & \bfOm^2\\
    &  (b, \xi(u) ,\xi(v)) & \mapsto & \xi(v)^*b.
  \end{array}
\end{equation}
It is straightforward to check that $T_{\bfOm^2}$ satisfies
the requirements of Definition \ref{def:ILH group vector bundles}, so we have
introduced the bundle $B(\bfOm^2,\bfDiff,T_{\bfOm^2})$ over $\bfDiff$.
\begin{lemma}
  \label{lem:tangent bundle exact case}
  The tangent bundle $T\bfGDiff_{H}$ is the pullback of the Whitney sum
  $$B(\bfOm^2,\bfDiff,T_{\bfOm^2})\oplus T\bfDiff$$ by the projection map $\pi:\bfGDiff_{H} \rightarrow \bfDiff$.
\end{lemma}
\begin{proof}
  The proof of this lemma follows directly from the expression of the defining maps for the bundles that are considered.
  Let $T_{\theta}$ be a defining map for $T\bfDiff$ and $T_{\theta_H}$
  be a defining map for $T\bfGDiff_{H}$. Then for $(u,b)$ in $\gdiff_{H}$ and $(v,B)$, $(w,B')$ in $U'\cap\GDiff_{H} $,
  a direct computation using (\ref{eq:defining map tangent}) and (\ref{eq:product-rule}) leads to:
  $$
  T_{\theta_H}((u,b), (\xi(v),B) ,(\xi(w),B'))=(T_\theta(u,\xi(v),\xi(w)),\xi(w)^*b).$$ \end{proof}

Next, consider the right-invariant extensions of $I$ and $\kappa$:
\begin{equation}
  B(\bfOm^2,\bfDiff,T_{\bfOm^2})\oplus T\bfDiff \overset{\tilde \kappa}{\rightarrow} 
  B(\bfOm^2,\bfDiff,\tilde T_{\bfOm^2}) \overset{\tilde I}{\rightarrow}  \bfDiff \times \RR^d
\end{equation}
defined by setting $\tilde \kappa=R_\phi\circ\kappa\circ R_\phi^{-1}$ and $\tilde I=R_\phi\circ I\circ R_\phi^{-1}$
on the fibre over $\phi\in\bfDiff$,
and where $\tilde T_{\bfOm^2}$ is defined as in Definition \ref{def:vector bundle parallel}
(the Riemannian structure on $\Lambda^2 T^*M$ being induced by the fixed metric $g_0$).
\begin{proposition}
  \label{prop:extension exact case}
  The maps $\tilde I$ and $\tilde k$ are $\cC^{\infty,\infty}$ ILH normal bundle homomorphisms.
\end{proposition}
\begin{proof}
  For $\tilde I$, the result follows from (3) in Theorem \ref{theo:extension theorem}.
  For $\tilde \kappa$, we need to consider the local expression, say $\Phi_\kappa$, of this bundle homomorphism.
  By definition, this is given, for $u\in U\cap\bfDiff$ and $(v,b)\in \bfGamma(TM)\times\bfOm^2$, by:
  \begin{equation*}
    \Phi_\kappa(u)(v,b)(x)=\tau(\exp_x(u(x)))^{-1}((\iota_{v'}H)-(\xi(u)^{-1})^*b )(\exp_x(u(x)),
  \end{equation*}
  where $v'= T_\theta(\xi(u),e)^{-1}(v)$. By definition, 
  $$(u,v)\mapsto\tau(\exp_x(u(x)))^{-1}(\iota_{v'}H)(\exp_xu(x))$$
  is the local expression for the right-invariant extension from
  $T\bfDiff$ to $B(\bfOm^2,\bfDiff,\tilde T_{\bfOm^2})$ of the homomorphism $v\mapsto \iota_vH$. 
  Thus, by (1) in Theorem \ref{theo:extension theorem},
  this component of $\tilde \kappa$ is $\cC^{\infty,\infty}$ ILH normal.
  The map 
  $$(u,b)\mapsto \tau(\exp_x(u(x)))^{-1}(\xi(u)^{-1})^*b$$
  can be rewritten as
  $$
  \tau(\exp_x(u(x)))^{-1}(\xi(u)^{-1})^*b(\exp_xu(x))=\Psi_{-1}(u,b)(x)
  $$
  with $\Psi_{-1}$ defined in (\ref{eq:psioperator}).
  By Lemma \ref{lem:psioperator}, this is a $\cC^{\infty,\infty}$ ILH normal
  homomorphism. Thus, $\Phi$ satisfies the required estimates and the result follows.
\end{proof}

Finally, we are ready to prove the last result of this section.
\begin{proposition}
  \label{prop:exact subgroup exact case}
  The group  $\Aut^e(E)$ is a strong ILH Lie subgroup of $\Aut(E)$, or, equivalently in terms of a splitting, $\bfGDiff_{H}^e$ is a strong ILH Lie subgroup of $\bfGDiff_{H}$. 
\end{proposition}
\begin{proof}
  From Proposition \ref{prop:extension exact case}, the operator $\tilde I\circ\tilde\kappa$ is $\cC^{\infty,\infty}$ ILH normal.
  We can extend $I\circ \kappa$ to a right-invariant $\cC^{\infty,\infty}$ ILH normal
  bundle homomorphism
  $$\pi^*(B(\bfOm^2,\bfDiff,T_{\bfOm^2})\oplus T\bfDiff)\to T\bfGDiff_{H}\times\RR^d.$$
  By Lemma \ref{lem:tangent bundle exact case}, this gives an extension from
  $T\bfGDiff_{H}$ to a trivial bundle. As $\ker(I\circ \kappa)=\gdiff_{H}^e$, the result
  follows from Remark \ref{rem:hypo+trivial}.
\end{proof}

\section{A slice theorem for generalized metrics}
\label{sec:moduliGM}

\subsection{Generalized metrics and statement of the slice theorem}

Let $E$ be an exact Courant algebroid. We define a generalized metric as follows.

\begin{definition}\label{def:gen-metric}
  A generalized metric on an exact Courant algebroid $E$ over an $n$-dimensional manifold $M$ is a rank $n$ subbundle
  $$
  V_+\subset E
  $$
  such that $\la\cdot,\cdot\ra_{\vert V_+}$ is positive definite.
\end{definition}

\begin{example}
  Consider $E=(TM+T^*M)_H$ for any $H\in \Omega^3_{cl}$. A usual metric $g$ on $M$ defines a generalized metric on $E$ by its graph $V_+=\{u+i_ug\st u\in TM\}$.
\end{example}

For any generalized metric $V_+\subset E$, we have that the projection $\pi_{V_+} : V_+ \rightarrow TM$ is an isomorphism and induces a metric $g$ on $TM$ by
$$g(u,v)=\la \pi_{V_+}^{-1}(u), \pi_{V_ +}^{-1}(v) \ra.$$
Moreover, the map $\s:TM\to E$ given by
$$
\s : u \mapsto \pi_{V_+}^{-1}(u) -\iota_ug
$$
defines an isotropic splitting of $E$. Conversely, a pair $(g,\s)$ consisting of a metric and an isotropic splitting defines a generalized metric by the subbundle
$$
V_+=\lbrace \lambda(u)+\iota_ug \st  u\in TM\rbrace\subset E.
$$

Let $\cGM$ denote the set of generalized metrics on $E$. With the notation of \eqref{eq:space-of-splittings} for the space of isotropic splittings $\Spl$, and the notation 
$$\cM := \{ g\in \Gamma(S^2 T^*M)\st g \textrm{ is positive definite} \}$$
for the space of metrics, the argument above is summed up in the isomorphism
$$\cGM \cong \cM \times \Spl.$$

By choosing any splitting $\s\in \Spl$, we have an isomorphism $\Spl\simeq_\s \Omega^2$, which shows that the space $\GM$ is an ILH manifold modelled on the ILH chain 
$$\lbrace  \bfOm^2\times \bfGamma(S^2T^*M),\Om^{2,k} \times \Gamma(S^2T^*M)^k, k\geq n+5 \rbrace,$$
as $\GM$ is regarded, indeed, as an open subspace of $\bfOm^2\times\bfGamma(S^2T^*M)$.

Since $E$ is fixed and no confusion is possible, we denote $\OO_\pi(E)$ by $\OO_\pi$. The strong ILH Lie group $\OO_\pi$ preserves the pairing $\la\cdot,\cdot \ra$ and thus acts on the right on the ILH manifold $\bfGM$, with ILH action, by pull-back:
\begin{equation}
  \label{eq:action on metrics}
  \begin{array}{cccc}
    \rho_{\GM}:&  \OO_\pi(E)\times \GM &  \rightarrow & \GM \\
    & (F,V_+) & \mapsto & F^{-1}(V_+).
  \end{array}
\end{equation} 
Likewise, we denote the groups of automorphisms $\Aut^e(E)$ and $\Aut(E)$ by $\GDiff^e$ and $\GDiff$. The restriction of \eqref{eq:action on metrics} defines ILH actions of $\bfGDiff$ and $\bfGDiff^e$ on $\GM$. For any $F\in \GDiff$, if $S$ is a subset of $\GM$, we will denote by $F\cdot S$ the image of $S$ by the action of $F$, namely $\rho_{\GM}(F,S)$.

By taking a splitting $\s\in \Spl$, we have isomorphisms 
\begin{align}\label{eq:with-splitting}
  \OO_\pi\simeq_\s \Diff\ltimes
  \Omega^2,&& \cGM\simeq_\s \cM \times \Omega^2,&&  V_+\simeq_\s\{u+i_ug +i_u \omega\st u\in TM\}.
\end{align}

The action then reads \begin{equation*}
  \label{eq:action exact case}
  \begin{array}{cccc}
    \rho_{\GM}:& ( \Diff\ltimes \Omega^2 )\times\GM& \rightarrow & \GM \\
    &((\phi,B),(g,\omega)) & \mapsto & (\phi^*g, \phi^*\omega-B),
  \end{array}
\end{equation*}
which restricts to actions of $\bfGDiff_{H}$ and $\bfGDiff_{H}^e$ on $\GM$. 

In this section we prove a slice theorem for these actions, inspired by the classical result of Ebin \cite{Eb}. The main difference in the proof is that we need to use the abstract Frobenius' Theorem (Theorem \ref{theo:Frobenius})
in order to endow the orbits under the  action of $\GDiff^e$ with an ILH manifold structure. We will only sketch a proof of the result, focusing on the differences with \cite{Eb}. The strategy of the proof is as in the finite-dimensional case. We fix a generalized metric $V_+$ on $E$, and let $\Isom^e(V_+)$ be its isotropy subgroup under the $\GDiff^e$ action. The ILH Frobenius' theorem enables us to endow the space $\GDiff^e/ \Isom^e(V_+)$ with an ILH structure. 
The orbit $\GDiff^e\cdot V_+$ is homeomorphic to this quotient, and carries
an ILH structure. Then, by means of an invariant metric on $\GM$, we can consider the normal bundle to $\GDiff^e\cdot V_+$, as well as an exponential map on $\GM$. The slice is then obtained by exponentiating small vectors on the normal bundle to the orbit at $V_+$.

\begin{remark}
  The space of generalized metrics on a given $M$ is a fairly complicated object. 
  The slice result enables us to have a better description of this space.
  It would be interesting to have a further decomposition of the space of generalized metrics as in \cite{koi78}.
  This would rely on a solution to the Yamabe problem for the generalized scalar curvature in a generalized conformal class of metrics.
\end{remark}

In all this section, we fix an auxiliary metric $g_0$ on $M$, and thus endow the spaces of sections $\Gamma(TM^{\otimes p}\otimes (T^*M)^{\otimes q})$ with
$L^{2,k}$ inner products as defined in (\ref{eq:norms}).

\begin{definition}
  The {\bf group of  generalized isometries} (resp. exact generalized isometries)
  of  $V_+\in\GM$, denoted $\Isom(V_+)$ (resp. $\Isom^e(V_+)$), is the isotropy group of $V_+$ under the $\GDiff$ action (resp. under the $\GDiff^e$ action). When a splitting $\s$ is chosen, so that
  \begin{align}
    V_+\simeq_\s (g,\omega),&& \GDiff\simeq_\s \GDiff_H,&& \GDiff^e\simeq_\s \GDiff^e_H,\label{eq:with-splitting-2}
  \end{align}
  we will refer to it as $\Isom_H(g,\omega)$ (resp. $\Isom^e_H(g,\omega)$).
\end{definition}

\begin{proposition}
  \label{prop:isotropy compact exact case}
  Let $V_+ \in \cGM$ with induced metric $g\in \cM$. The groups $\Isom(V_+)$ and $\Isom^e(V_+)$ are isomorphic to compact subgroups of  $\Isom(g)$.
\end{proposition}
\begin{proof}
  Choose a splitting $\s$ so that we have $V_+\simeq_\s (g,\omega)$. The isotropy subgroup $\widetilde\Isom_H(g,\omega)$ of $(g,\omega)$ under the action of $\Diff\ltimes\bfOm^2$  is isomorphic, as a topological group, to $\Isom(g)$ via the map $(\phi,B)\mapsto \phi$. Indeed, $(\phi,B)$ fixes $(g,\omega)$ if and only if $\phi^*g=g$, i.e., $g\in \Isom(g)$, and $B$ is uniquely determined by $B=\phi^*\omega-\omega$. The images under this map of the closed subgroups $\widetilde\Isom_H(g,\om)\cap\GDiff_H$ and  $\widetilde\Isom_H(g,\om)\cap\GDiff_H^e$ give the result.
\end{proof}

The main result of this section is the following.
\begin{theorem}
  \label{theo:slice exact case}
  Let $V_+$ be a generalized metric on $E$. There exists an ILH submanifold $\cS$ of $\GM$ such that:
  \begin{enumerate}
  \item[a)] For all $F\in \Isom^e(V_+)$, $F\cdot \cS=\cS$.
  \item[b)] For all $F \in \GDiff^e$, if $(F\cdot \cS) \cap\cS\neq \emptyset$, then $F\in\Isom^e(V_+)$.
  \item[c)] There is a local cross-section $\chi$ of the map $F \mapsto \rho_{\GM}(F,V_+)$ on a neighbourhood
    $\cU$ of $V_+$ in $\GDiff^e\cdot V_+$ such that the map from $\cU\times\cS$ to $\GM$ given by
    $(V_1,V_2)\mapsto \rho_{\GM}(\chi(V_1),V_2)$ is a homeomorphism onto its image.
  \end{enumerate}
\end{theorem}
A similar statement holds for the full group of generalized diffeomorphisms:
\begin{theorem}
  \label{theo:slice exact case full group}
  Let $V_+$ be a generalized metric on $E$. There exists an ILH submanifold $\cS$ of $\GM$ such that:
  \begin{enumerate}
  \item[a)] For all $F\in \Isom(V_+)$, $F\cdot \cS=\cS$.
  \item[b)] For all $F \in \GDiff$, if $(F\cdot \cS)\cap\cS\neq \emptyset$, then $F\in\Isom(V_+)$.
  \item[c)] There is a local cross-section $\chi$ of the map $F \mapsto \rho_{\GM}(F,V_+)$ on a neighbourhood
    $\cU$ of $V_+$ in $\GDiff\cdot V_+$ such that the map from $\cU\times\cS$ to $\GM$ given by
    $(V_1,V_2)\mapsto \rho_{\GM}(\chi(V_1),V_2)$ is a homeomorphism onto its image.
  \end{enumerate}
\end{theorem}

These theorems provide slices to the actions of $\GDiff^e(E)$ and $\GDiff(E)$ on $\GM$, thus generalizing the results of Ebin \cite{Eb}. 

\subsection{Proof of the slice theorem}

This section will consist in proving Theorems \ref{theo:slice exact case} and \ref{theo:slice exact case full group}, for which we choose a splitting $\s$ in order to have $$E\simeq_\s (TM+T^*M)_H$$ and the identifications \eqref{eq:with-splitting} and \eqref{eq:with-splitting-2}. For simplicity, we use the notation $\Isom_H(V_+)$ for $\Isom_H(g,\omega)$ when $V_+\simeq_\s (g,\omega)$, and similarly for exact isometries.

The first step will be to build an ILH structure on the orbits. For a technical reason, we will first address the case of exact automorphisms of $E$.

\begin{proposition}
  \label{prop:isotropy ILH exact case}
  Let $V_+\in\GM$. The group $\Isom^e_H(V_+)$ is a strong ILH Lie subgroup of $\bfGDiff^e_H$. Moreover,
  the quotient space $\bfGDiff^e_H/\Isom^e_H(V_+)$ carries an ILH manifold structure.
\end{proposition}

\begin{proof}
  We will apply Frobenius' Theorem (Theorem \ref{theo:Frobenius}) to show that  $\Isom^e_H(V_+)$ is an ILH Lie subgroup. The ILH manifold structure in the quotient will then be a direct consequence of Corollary \ref{cor:Frobenius}. Consider the map
  \begin{equation}\label{eq:Phi-action}
    \begin{array}{cccc}
      \Phi : & \GDiff_H^e  & \rightarrow & \bfGamma(S^2T^*M) \times \bfOm^2\\
      &  (\phi,B) & \mapsto & \rho_{\GM}((\phi,B),(g,\om)),
    \end{array}
  \end{equation}
  for which $\Phi^{-1}(V_+)=\Isom^e_H(V_+)$. The differential of $\Phi$ defines a right-invariant $\cC^{\infty,\infty}$ ILH normal bundle homomorphism from $T\bfGDiff_H^e$ to the bundle $B(\bfGamma(S^2T^*M)\times\bfOm^2,\bfGDiff_H^e,T_{gm})$, where $T_{gm}$ is defined by
  \begin{equation}
    T_{gm}((\dot g, \dot\om),(\xi(u),B),(\xi(v),B))=\xi(v)^*(\dot g,\dot\om),\label{eq:def-Tgm}
  \end{equation}
  for $(\xi,U)$  a chart at the origin for $\bfDiff$.
  Denote by $\bfA'$ the differential of $\Phi$ at the origin:
  \begin{equation*}
    \begin{array}{cccc}
      \bfA': & \gdiff_H^e  & \rightarrow & \bfGamma(S^2T^*M)\times \bfOm^2\\
      &  (u,b)      & \mapsto     &  (\cL_u g, \cL_u\om - b).
    \end{array}
  \end{equation*}
  The proof will consist in checking the hypotheses of Theorem \ref{theo:Frobenius} for $\bfA'$. For some technical results regarding elliptic operator theory we will refer to Appendix \ref{sec:Appendix}.

  In order to apply elliptic operator theory, we parametrize $\gdiff_H^e$ by $\Gamma(TM+T^*M)$. This space of sections has an ILH structure by the ILH chain $\bfGamma(TM+T^*M):=\lbrace \bfGamma(TM)\times \bfOm^1,\;  \Gamma(TM)^{k+1}\times\Om^{1,k+1}, \; k\geq n+5\rbrace$. Consider the surjective morphism
  \begin{equation*}
    \label{eq:injection map exact case}
    \begin{array}{cccc}
      \iota_e : & \bfGamma(TM+T^*M) & \rightarrow & \gdiff_H^e \\
      & u+\al & \mapsto & (u,\iota_u H - d\al)
    \end{array}
  \end{equation*}
  Set $\bfA:=\bfA'\circ \iota_e$, which is a first-order differential operator from
  $\bfGamma(TM)\times\bfOm^1$ to $\bfGamma(S^2T^*M)\times\bfOm^2 $:
  \begin{equation}\label{eq:defA}
    \begin{array}{cccc}
      \bfA: &  \bfGamma(TM+T^*M)  & \rightarrow & \bfGamma(S^2T^*M)\times\bfOm^2\\
      &  u+\al      & \mapsto     &  (\cL_u g,  \cL_u\om - \iota_uH +d\al)
    \end{array}
  \end{equation}
  We then define the first-order operator $\bfB$ on the ILH chain
  $\lbrace \bfOm^0,\; \Om^{0,k+2}, k\geq n+5 \rbrace$ by
  $$
  \begin{array}{cccc}
    \bfB : & \bfOm^0 & \rightarrow & \bfGamma(TM+T^*M)\\
    &  f      &  \mapsto &  (0,df).
  \end{array}
  $$ 
  As $u\mapsto \cL_ug$ has injective symbol, the sequence
  \begin{equation}
    \label{eq:complex proof isom ILH exact}
    \bfOm^0 \stackrel{\bfB}{\rightarrow} \bfGamma(TM+T^*M) \stackrel{\bfA}{\rightarrow} \Gamma(S^2T^*M)\times \Om^2
  \end{equation}
  is elliptic in the middle, that is the range of the principal symbol of $\bfB$ equals the kernel of the principal symbol of $\bfA$. We can then use elliptic operator theory to prove that the operator $\bfA$ satisfies the properties $(a), (b), (c)$ and $(d)$ of Proposition \ref{prop:elliptic implies Frobenius} (compare with hypotheses of Theorem \ref{theo:Frobenius}).

  Now we check the hypotheses for $\bfA'$. From the surjectivity of $\iota_e$, we deduce
  \begin{align*}
    \gdiff_H^e=\ker \bfA' \oplus \iota_e(\Im \bfA^*), && \bfGamma(S^2T^*M) \times
    \bfOm^2=\Im \bfA' \oplus \ker \bfA^*,
  \end{align*}
  where $\bfA^*$ denotes the adjoint operator of $\bfA$ for the $L^{2,0}$ pairing (see  Appendix \ref{sec:Appendix}). Note that $\iota_e$ admits a continuous right inverse onto a complementary subspace of $\ker \iota_e$:
  $$
  \iota_e^{-1}:(u,b)\mapsto u-\bbG d^*(b-\iota_u H).
  $$
  Using $\iota_e^{-1}$, the closedness of $\Im \bfA^*$, and the continuity of $\iota_e$, we deduce that $\iota_e(\Im \bfA^*)$
  is a closed subspace of $\gdiff_H^e$, and (a) in Theorem \ref{theo:Frobenius} is satisfied. Since $\Im \bfA'= \Im \bfA$, and as $\Im \bfA$ is closed, we have that $\Im \bfA'$ is a closed
  subspace of $\bfGamma(S^2T^*M)\times\bfOm^2$, so (b) is satisfied. To conclude, as $\bfA$ satisfies the hypotheses  $(c)$ and $(d)$ in Proposition \ref{prop:elliptic implies Frobenius}, we find the estimates for $\bfA'$ (hypotheses (c) and (d) in Theorem \ref{theo:Frobenius}) by using the ILH linear normal maps $\iota_e$ and $\iota_e^{-1}$.
\end{proof}

Now we can prove:

\begin{proposition}
  \label{prop:isotropy ILH exact case full group}
  Let $V_+\in\GM$. Then $\Isom_H(V_+)$ is a strong ILH Lie subgroup of $\bfGDiff_H$. Moreover,
  the quotient space $\bfGDiff_H/\Isom_H(V_+)$ carries an ILH manifold structure.
\end{proposition}

\begin{proof}
  We use the proof of Proposition \ref{prop:isotropy ILH exact case}.  Consider this time the differential of
  \begin{equation}\label{eq:Phi-action-2}
    \begin{array}{cccc}
      \Phi : & \GDiff_H  & \rightarrow & \bfGamma(S^2T^*M) \times \bfOm^2\\
      &  (\phi,B) & \mapsto & \rho_{\GM}((\phi,B),(g,\om)),
    \end{array}
  \end{equation}
  which defines a right-invariant  $\cC^{\infty,\infty}$ ILH normal bundle homomorphism from $T\bfGDiff_H$  to $B(\bfGamma(S^2T^*M)\times\bfOm^2,\bfGDiff_H,T_{gm})$, where $T_{gm}$ is defined as in \eqref{eq:def-Tgm}. Denote by $\bfA'$ the differential of $\Phi$ at the origin:
  \begin{equation*}
    \begin{array}{cccc}
      \bfA': & \gdiff_H  & \rightarrow & \bfGamma(S^2T^*M)\times\bfOm^2\\
      &  (u,b)      & \mapsto     &  (\cL_u g,\cL_u\om-b)
    \end{array}
  \end{equation*}
  As in Proposition \ref{prop:isotropy ILH exact case}, the proof consists in checking the hypotheses  of Theorem \ref{theo:Frobenius} for $\bfA'$ and we rely on Appendix \ref{sec:Appendix} for elliptic operator theory.

  Consider the map 
  \begin{equation*}
    \begin{array}{cccc}
      \Psi : &\gdiff_H & \rightarrow & \bfGamma(TM) \times ( d\bfOm^1\oplus  \cH^2) \\
      &  (u,b)  & \mapsto &     (u,\iota_uH - b) ,
    \end{array}
  \end{equation*}
  where $\cH^2$ is the space of harmonic $2$-forms with respect to the metric $g$.

  Then $\Psi$ is a linear ILH continuous isomorphism, with continuous inverse given by the map 
  $(u,\beta)\mapsto (u,\iota_uH - \beta)$ for $u\in \bfGamma(TM)$ and $\beta\in d\bfOm^1 \oplus \cH^2$.
  Thus
  $$
  \gdiff_H=\Psi^{-1}( \bfGamma(TM)\times d\bfOm^1 )\oplus \Psi^{-1}( \cH^2).
  $$
  But $\Psi^{-1}( \bfGamma(TM)  \times d\bfOm^1)=\gdiff_H^e$, and $\Psi^{-1}( \cH^2)=\cH^2$ so
  $$
  \gdiff_H=\gdiff_H^e\oplus \cH^2.
  $$
  We are in the setting of Section \ref{sec:direct sums}, with $\bfE=\gdiff_H$, $\bfE_0=\gdiff_H^e$
  and $\cH=\cH^2$, as, by the proof of Proposition \ref{prop:isotropy ILH exact case},  the restriction of $\bfA'$ to $\gdiff_H^e$ satisfies the hypotheses of Theorem \ref{theo:Frobenius}.
  
  By Lemma \ref{lem:kernel decomposition abstract lemma}, the subspace $\ker \bfA'$ admits a complementary closed subspace
  $\bfE_2'$. We want to apply Lemma \ref{lem:kernel estimate abstract lemma} to $\bfA'$. First, note that, by definition, $\vert\vert \bfA' h \vert\vert_k=\vert\vert h \vert\vert_k$  for all $h\in \cH^2$.  We consider the norm $\vert\vert \cdot \vert\vert_0$ on the space $\gdiff_H$ (this makes sense
  even if the ILH chain defining $\gdiff_H$ starts at $n+5$).
  Let $z\in \gdiff_H$, which we decompose as $z=(u,\iota_uH - da)+(0,h)$ in $\gdiff_H^e\oplus \cH^2$.
  Then, by Hodge decomposition into orthogonal components,
  $$\vert\vert z \vert\vert_0^2 \geq \vert\vert h + h_u \vert\vert_0^2$$
  where $h_u$ denotes the harmonic part of $\iota_uH$. Using the Cauchy-Schwarz inequality,
  $$
  \vert\vert h\vert\vert_0^2\leq \vert\vert z\vert\vert^2+ 2\vert\vert h \vert \vert_0\vert\vert h_u\vert\vert_0.
  $$
  There is a uniform constant $C$ depending on $H$ such that
  $
  \vert\vert h_u\vert\vert_0\leq \vert\vert \iota_u H\vert\vert_0\leq C \vert\vert u\vert\vert_0.
  $
  As $\vert\vert u\vert\vert_0\leq \vert\vert z\vert\vert_0$, we have 
  $
  \vert\vert h\vert\vert_0^2\leq \vert\vert z\vert\vert^2+ 2C\vert\vert h \vert \vert_0\vert\vert z\vert\vert_0,
  $
  and thus, $$\vert\vert h\vert\vert_0\leq (1+2C)\vert\vert z \vert\vert_0.$$ We can now apply Lemma \ref{lem:kernel estimate abstract lemma} and get the linear estimates $(c)$ of Theorem \ref{theo:Frobenius} for $\bfA'$. By Lemma \ref{lem:decomposition image abstract lemma}, we have an orthogonal decomposition
  $$
  \bfGamma(TM)\times\Om^2=\Im\bfA'\oplus \bfF_3
  $$
  into closed subspaces, so (a) and (b) are satisfied. Finally, by Lemma \ref{lem: linear estimate image abstract lemma}, $\bfA'$ satisfies (d). 
\end{proof}

One of the main ingredients in the proof of Theorem \ref{theo:slice exact case} is a $\GDiff_H$-invariant metric on the space $\GM$. Take a splitting $\s$ so that we have $\GM\simeq_\s \cM\times \Omega^2$ and define, for $(g,\omega)\in\GM$ and $(\dot g,\dot\om)\in T_{(g,\om)}\GM$, the pairing
\begin{equation}
  \label{eq:invariant metric exact case}
  \la (\dot\om,\dot g),(\dot \om,\dot g) \ra_{g,\omega}=\int_M \la\dot\om,\dot\om\ra_g \:\dvol_g + \int_M \la\dot g,\dot g\ra_g \: \dvol_g
\end{equation}
We thus have a smooth $\GDiff_H$-invariant (weak) Riemannian metric on $\GM$.
\begin{remark}
  This metric only gives the (non-complete) $L^2$-topology on the tangent bundle of $\GM$, hence the denomination {\it weak} metric.
  As in \cite[Sec. 4]{Eb}, $(g,\om)\mapsto \la\cdot,\cdot\ra_{g,\om}$ is a smooth metric.
  It can be extended to a smooth metric on the spaces $\GM^k$ for all $k\geq n+5$.
  Moreover, each of these extended metrics admits a Levi-Civita connection and we
  can define the associated exponential maps $\exp^k_{\GM}$. Note that the definition of
  the metric only involves $g$, thus the proofs of these facts follow readily from Ebin's work.
  An important issue is the existence of an exponential map for $(\GM,\la\cdot,\cdot\ra)$.
  The proof in \cite{Eb} enables to define a Levi-Civita connection for $\la\cdot,\cdot\ra$ on $\GM$, whose
  extension to the spaces $\GM^k$ gives the Levi-Civita connections of the extended metrics.
  On the spaces $\GM^k$, the exponential is everywhere a local diffeomorphism from $T\GM^k$ to $\GM^k$,
  by the implicit function theorem in Hilbert spaces. However, this does not ensure
  the existence of an exponential map on $\GM$. Fixing $x\in\GM$, for each $k$ we have a
  neighbourhood of zero $\cV^k$ in $T_x\GM^k$ such that the exponential map is a diffeomorphism from $\cV^k$
  onto its image. However, the inverse limit of $(\cV^k)_{k\in\NN(d)}$ could shrink to zero.
  We will see in the proof of Theorem \ref{theo:slice exact case} how to overcome this difficulty.
\end{remark}

We are now ready to prove Theorems \ref{theo:slice exact case} and \ref{theo:slice exact case full group}, following \cite[Thm. 7.1 and Thm. 7.4]{Eb}.
\begin{proof}[Proof of Theorem \ref{theo:slice exact case}]
  The proof is organized in three steps.
  \subsubsection*{First step: ILH structure on the orbit}
  The action $\rho_{\GM}$ of $\bfGDiff_H^e$ on $V_+\in \GM$ induces a map
  $$\rho_{V_+}:\bfGDiff_H^e/\Isom_H^e(V_+)\rightarrow \cO_{V_+},$$
  where $\cO_{V_+}:=\bfGDiff_H^e\cdot V_+$ denotes the orbit. The map $\rho_{V_+}$ is injective and an immersion by the proof of Proposition \ref{prop:isotropy ILH exact case}. We will prove that $\rho_{V_+}$ is an homeomorphism onto  $\cO_{V_+}$ by showing that its image is closed.  Recall that $\rho_{\GM}(\phi,B)=(\phi^*g,\phi^*\om-B)$. From Ebin's work, we know that the map
  $\phi\mapsto \phi^*g$ has a closed image and is an homeomorphism from $\bfDiff/\Isom(g)$ onto its image.
  Thus, if a sequence $(\phi_j^*g,\phi_j^*\om-B_j)$ converges to $(g_\infty,\om_\infty)$, we can find $\phi_\infty\in \bfDiff$ such that $(\phi_j)$ converges to $\phi_\infty$.  But then $(\phi_j^*\om)$ converges to $\phi_\infty^*\om$, and there exists a $2$-form $B_\infty$ such that $(B_j)$ converges to $B_\infty$ and $\om_\infty=\phi_\infty^*\om-B_\infty$. As $\bfGDiff_H^e$ is a closed subgroup
  of $\bfDiff\ltimes \Om^2$, the map $\rho_{V_+}$ is a homeomorphism onto $\cO_{V_+}$, which thus becomes a closed ILH submanifold of $\GM$.

  \subsubsection*{Second step: construction of the normal bundle to the orbit}
  We define the normal bundle $\nu$ of the submanifold $\cO_{V_+}$ in $T\GM$ 
  to be the orthogonal to $T\cO_{V_+}$ with respect to the invariant metric (\ref{eq:invariant metric exact case}).
  As the metric is only a weak one, $\nu$ might not be a smooth subbundle of $T\GM_{\vert\cO_{V_+}}$.
  To show that $\nu$ is indeed a smooth subbundle, we will
  prove that it is the kernel of a smooth map $P$
  of bundles in the ILH category.
  Define a surjective bundle map
  $$
  P: T\GM_{\vert\cO_{V_+}} \rightarrow T\cO_{V_+}
  $$
  by transporting along the orbit the map $\bfA\circ \bbG \circ \bfA^*$, where $\bfA$ is the operator in \eqref{eq:defA}, and $\bbG$ is the Green
  operator of the complex (\ref{eq:complex proof isom ILH exact}). We want to show that $P$ is smooth. Fix $k\geq n+5$ and work in the Hilbert category on $\GM^k$, considering
  the $\GDiff_H^{e,k}$ orbit $\cO_{V_+}^k$ of $V_+$. The superscript $k$
  will refer to the Hilbert completion of the object in this category. The map $P$ extends to a map
  $$
  P^k: T\GM^k_{\vert\cO_{V_+}} \rightarrow T\cO_{V_+}^k.
  $$
 If $\eta\in \GDiff_H^{e,k}$, as in \cite[proof of Thm. 7.1]{Eb} we have that $P^k_{\eta\cdot V_+}$ equals $\bfA_\eta\circ (\bfA_\eta^*\circ\bfA_\eta)^{-1}\circ \bfA_\eta^*$
 where $\bfA_\eta = \eta^*\circ \bfA\circ (d R_\eta)^{-1}$
 and $R_\eta$ is the right translation. The operator
 $\bfA_\eta$ is the infinitesimal action of $\GDiff_H^{e,k}$ at $\eta\cdot V_+$. As the action of $\GDiff_H^{e,k}$
 on $\cG\cM^k$ is smooth, the map $\eta \mapsto \bfA_\eta$ is smooth. If $\eta \mapsto \bfA_\eta^*$ is smooth, then $\eta \mapsto\bfA_\eta^*\circ \bfA_\eta$ will be smooth and so will be $P$ as the inverse map is smooth.
 We now compute the adjoint of $\bfA$. For $(u,\alpha)\in \Gamma(TM)^{k+1}\times
  \Om^{1,k+1}$ and $(\dot g, \dot \om)\in\Gamma(S^2T^*M)^k\times\Om^{2,k}$, 
 if $\la \cdot , \cdot \ra_{L^2_g}$ denotes the $L^2$ inner product induced by $g$,
 \begin{equation*}
  \begin{array}{ccc}
   \la\bfA(u,\alpha), (\dot g, \dot \om) \ra_{g,\omega} & = & \la \cL_u\om -\iota_u H +d\alpha ,\dot\om\ra_{L^2_g}  +  \la \cL_u g,\dot g\ra_{L^2_g}\\
    & = & \la \iota_u(d\om - H)  ,\dot\om\ra_{L^2_g}+ \la d\iota_u\om  ,\dot\om\ra_{L^2_g}  +  \la \cL_u g,\dot g\ra_{L^2_g}  +  \la d\alpha  ,\dot\om\ra_{L^2_g}\\
    & = & \la u  ,F_1(\dot\om)\ra_{L^2_g}+ \la u  , F_2(d^*\dot\om)\ra_{L^2_g}  +  \la u ,F_3(\dot g)\ra_{L^2_g}  +  \la \alpha  ,d^*\dot\om\ra_{L^2_g}
  \end{array}
 \end{equation*}
 where $F_1$ and $F_2$ are zero order differential operators whose coefficients are rational functions of $(g,\om, d\om, H)$ and $F_3$ is the adjoint of $u\mapsto \cL_ug$.
Thus the adjoint of $\bfA$ can be written
  $$
  \bfA^*(\dot g,\dot \om)=( F_1(d^*\dot \om)+F_2(\dot\om)+F_3(\dot g),d^* \dot\om ).
  $$
  The smoothness of $\bfA_\eta^*$ follows from its local 
  expression as in \cite[Thm. 7.1.]{Eb}, using \cite[Lem. 3.2, Thm. 3.3]{Eb}. 
  Then, by definition of $P$ and invariance of the metric (\ref{eq:invariant metric exact case}) on $\cG\cM$, $\nu$ is the kernel
  of $P$, and thus it is a smooth subbundle of $T\GM_{\vert\cO_{V_+}}$.
  
  \subsubsection*{Third step: construction of the slice}
  Here, we fix $k\geq n+5$ and work in the Hilbert category on $\GM^k$, considering
  the $\GDiff_H^{e,k}$ orbit $\cO_{V_+}^k$ of $V_+$.
  We will use the invariant exponential map $\exp^k_{\GM}$
  associated to the weak metric (\ref{eq:invariant metric exact case}).
  
  We have to find small neighbourhoods $\cV^k$ of zero in the 
  fibre $\nu_{V_+}^k$ of $\nu^k$ over $V_+$ and $\cU^k$ of $V_+$ in $\cO_{V_+}^k$
  with a cross-section $\chi:\cU^k\rightarrow \GDiff_H^{e,k}$
  such that
  \begin{equation}
    \label{eq:neighbourhoods last step}
    \cW^k:=\lbrace d_2\rho_{\GM}(\phi,x), \phi\in \chi(\cU^k) \textrm{ and } x \in \cV^k \rbrace\subset \nu^k,
  \end{equation}
  where $d_2\rho_{\GM}$ is the differential of $\rho_{\GM}$ with respect to the second variable.
  Then, the restriction of $\exp^k_{\GM}$ to $\cW^k$ defines
  a diffeomorphism onto a small neighbourhood of $\cO_{V_+}^k$ in $\GM^k$, and the slice will be $\exp_{\GM}(\cV^k)\cap\GM$.
  Note that this last step needs a little explanation to be adapted from Ebin's proof.
  Indeed, the fact that $\exp^k_{GM}(\cV^k)\cap\GM$ is homeomorphic to $\exp^k_{\GM}(\cV^k\cap T\GM)$
  is not straightforward as we work
  in the category of ILH manifolds. The solution to this problem in Ebin's work is 
  provided by \cite[Thm. 7.5, Cor. 7.6]{Eb}.
  
  We adapt here the argument. The construction of the neighbourhoods $\cV^k$, $\cU^k$ and $\cW^k$
  satisfying (\ref{eq:neighbourhoods last step}) is done exactly as in \cite{Eb}.

  One needs to show that
  $$
  \exp^k_{\GM}: \cW^k\cap T\GM \rightarrow \exp_{\GM}(\cW^k)\cap \GM
  $$
  is well defined and a homeomorphism (for the smooth topology). 
  Elements in $\GDiff_H^e$ act on $T\GM^k$ and $\GM^k$. Denote by $\cL_{(u,b)}$
  the infinitesimal action, or Lie derivative. We restrict ourselves to the action of elements
  of the form $(u,\iota_uH)\in\gdiff_H^e$, for $u\in \Gamma(TM)$.
  The action at $(g,\om)\in\GM^k$ reads
  $$
  \cL_{u,\iota_uH}(g,\om)=(\cL_u g,\iota_uH+\cL_u\om).
  $$
  As $H$ is smooth, we deduce that $(g,\om)$ is in $\GM$ if and only if all its iterated Lie derivatives  with respect to elements of the form $(u,\iota_uH)\in\gdiff_H^e$  exist.
  As $\exp^k_{\GM}$ is $\GDiff_H^e$-invariant, it commutes with these Lie derivatives. Hence, $\exp_{\GM}$ sends smooth elements to smooth elements. Moreover, a sequence $(g_n,\om_n)$ converges to $(g,\om)$ in $\GM$ if and only if all its Lie derivatives do. As $\cW^k$ is stable under the action of small elements in $\bfGDiff_H^e$, $\exp^k_{\GM}$ is a continuous map from $\cW^k\cap T\GM$ to $\exp^k_{\GM}(\cW^k)\cap \GM$. The same argument applied to the inverse map $(\exp^k_{\GM})^{-1}$ implies surjectivity of $\exp^k_{\GM}$ onto smooth elements and that $\exp^k_{\GM}$ is an homeomorphism from $\cW^k\cap T\GM$ to $\exp^k_{\GM}(\cW^k)\cap \GM$.
  
  Then, one can set $\cS=\exp^k_{\GM}(\cV^k)\cap\GM$. The required properties for the slice follow from the invariance of the metric under  the $\bfGDiff_H^e$-action, and the regularities of the action and the exponential map.
\end{proof}

\begin{proof}[Proof of Theorem \ref{theo:slice exact case full group}]
  The proof follows the proof of Theorem \ref{theo:slice exact case}. The only step which deserves some clarification is the second step: the construction of the normal bundle to the orbit as a smooth subbundle of $T\GM$.
  We use the notations from the proofs of Propositions \ref{prop:isotropy ILH exact case}
  and \ref{prop:isotropy ILH exact case full group}.
  Denote by $\nu'$ the normal bundle (with respect to the invariant metric) to the $\GDiff_H$-orbit $\cO_{V_+}$.
  First, we proceed as for the case of exact diffeomorphisms, and consider the orthogonal splitting
  $$
  T_{V_+}\GM=\Im\bfA \oplus \ker \bfA^*
  $$
  where $\bfA$ is the operator of Proposition \ref{prop:isotropy ILH exact case}.
  Transport equivariantly along the $\bfGDiff_H$-orbit the operator $\bfA\circ \bbG\circ \bfA^*$, where $\bbG$ is the Green
  operator of the complex (\ref{eq:complex proof isom ILH exact}). We thus obtain a smooth bundle homomorphism from $T\GM_{\vert\cO_{V_+}}$ to $T\cO_{V_+}$, whose kernel defines a smooth ILH bundle $\nu$ over the orbit.
  
  From the proof of Proposition \ref{prop:isotropy ILH exact case full group}
  and by Lemma \ref{lem:decomposition image abstract lemma}, we have the orthogonal decompositions
  \begin{align*}
    T_{V_+}\GM=\Im \bfA' \oplus \bfF_3=\Im \bfA \oplus \ker\bfA^*, &&
    \Im \bfA'=\Im \bfA \oplus \cF,
  \end{align*}
  where $\bfF_3$ is the kernel
  of the orthogonal projection $p_0: \ker \bfA^* \rightarrow \cF$ onto a finite-dimensional subspace $\cF\subset \ker \bfA^*$ defined as follows. Let $(h_i)$ be an orthonormal basis of $\cH^2$. Decompose $\bfA' h_i=\bfA'x_i \oplus f_i$
  in the direct sum decomposition
  $\bfGamma(TM)\times\bfOm^2=\bfA'\gdiff_H^e \oplus \ker \bfA^*$, and define $\cF$ to be the span of the $f_i$. As $\Im \bfA$ is orthogonal to $\ker \bfA^*$, and
  as for all $h\in \cH^2$, we have $\bfA' h= (h,0)$,
  the projection operator can be written $p_0(x)=\sum_i \la x, h_i\ra_0 h_i$. This is nothing but the projection onto the space of harmonic $2$-forms.

  Note that, by construction, $\nu_{V_+}=\ker\bfA^*=\bfF_3\oplus \cF$, and $\nu'_{V_+}=\bfF_3=\ker p_0$.

  We can extend equivariantly the operator $I_0$ along the orbit $\cO_{V_+}$
  and define a bundle homomorphism $\tilde I_0$ from $\nu$ to the trivial bundle $\RR^{b_2(M)}$
  over $\cO_{V_+}$. As the kernel of
  the laplacian operator on $2$-forms varies smoothly with the metric, $\tilde I_0$
  is a smooth ILH bundle homomorphism.
  Then, the kernel of $\tilde I_0$ is by construction the normal bundle $\nu'$ to the orbit,
  and this defines a smooth ILH bundle.

  The other steps of the proof are analogous to Theorem \ref{theo:slice exact case}. 
\end{proof}

\begin{remark}
  The space $\GM$ is a strong ILH manifold, but Ebin's proof does not allow us to conclude that the slices we built are strong ILH submanifolds.
  So far, we have proved that they are ILH submanifolds of $\GM$. The problem to obtain a strong ILH manifold structure is related to the regularity of the exponential map
  $\exp_{\GM}$. To obtain a stronger structure for $\cS$, one would need linear estimates on the connection related to the weak inner product in order to ensure regularity of geodesics, as in \cite{Om73}
  or \cite[Sec. VI.1]{Om97}.
\end{remark}

\section{The ILH stratification of the moduli space}
\label{sec:stratification}

Let $E\rightarrow M$ be an exact Courant algebroid. Its automorphism group, or group of generalized diffeomorphisms $\GDiff$, acts on the set of generalized metrics $\GM$ by inverse image, as in (\ref{eq:action on metrics}). The {\bf moduli space of generalized metrics} $\GR$  is defined by
$$
\GR:=\GM/\GDiff.
$$
In this section we will exhibit an ILH stratification for $\GR$.

\begin{definition}
  Let $\cT$ be a topological space and $A$ be a countable partially ordered set. A partition  $(\cT_\alpha)_{\alpha\in A}$ of $\cT$ is an {\bf ILH stratification} of $\cT$ if
  \begin{enumerate}
  \item For each $\alpha\in A$, the space $\cT_\alpha$ is a strong ILH manifold whose topology is induced by the topology on $\cT$.
  \item Whenever $\cT_\alpha\cap \overline{\cT_\beta}\neq \emptyset$, then $\beta<\alpha$  and $\cT_\alpha \subset \overline{\cT_\beta}$. 
  \end{enumerate}
  The spaces $\cT_\alpha$ are called the strata of the stratification.
\end{definition}

For any compact subgroup $G\subset \GDiff$, denote by $\GM_G$ the space of generalized metrics whose generalized isometry group is $G$ and by $N^G$ the normalizer of $G$ in $\GDiff$. There is an action of $G\backslash N^G$ on $\GM_G$, whose quotient space we denote by $\GR_{(G)}$.
This is the space of generalized metrics whose isotropy group is conjugated to $G$. We will show the following theorem:

\begin{theorem}
  \label{theo:Stratification ILH exact case}
  The space $\GR$ admits an ILH stratification $(\GR_{(G)})_{(G)\in \cA}$
  where $\cA$ is the set of conjugacy classes of generalized isometry groups of $E$.
  Moreover, if $(G)\subsetneq (G')$, then $\GR_{(G')}\subset \overline{\GR_{(G)}}$, i.e., the strata intersect as much as possible.
\end{theorem}

Again, in order to provide the proof for this theorem, we choose a splitting   $\s$ so that $$E\simeq_\s (TM+T^*M)_H$$ and we have the identifications \eqref{eq:with-splitting} and \eqref{eq:with-splitting-2}.

\subsection{A description of generalized isometry groups}
Following \cite{Bou}, to prove Theorem \ref{theo:Stratification ILH exact case}, we first describe the generalized isometry
groups of $E$. Denote by $\Pi$ the projection
$$
\Pi: \Diff\ltimes \Om^2 \rightarrow \Diff.
$$

Note that $\Diff_{[H]}$, the space of diffeomorphisms that preserve the cohomology class $[H]$ is a strong ILH Lie subgroup
of $\bfDiff$ (these two groups share the same Lie algebra, $\bfDiff_{[H]}$ contains the connected component of identity
in $\bfDiff$). Recall the following, which follows implicitly from Proposition \ref{prop:isotropy ILH exact case} and recalls Proposition \ref{prop:isotropy compact exact case}.

\begin{lemma}
  \label{lem:description generalized isometries}
  Let $(g,\om)\in\GM$. Set $G=\Isom_H(g,\om)$. Then $G$, respectively $\Pi(G)$, is a compact Lie subgroup of $\bfGDiff_H$,
  respectively $\bfDiff_{[H]}$. The restriction of $\Pi$ to $G$ is an isomorphism of Lie groups and
  its inverse reads:
  \begin{equation}
    \label{eq:isom between G and Pi(G)}
    \Pi^{-1}(\phi)=(\phi,\om-\phi^*\om)
  \end{equation}
\end{lemma}

\begin{proof}
  The fact that $\Pi$ is an isomorphism of groups from $G$ to its image follows from the definition of $\GDiff_H$
  and its action on $\GM$. Then $\Pi(G)$ is a compact subgroup of $\bfDiff_{[H]}$, thus a Lie subgroup of $\bfDiff_{[H]}$
  by a theorem of Myers and Steenrod, see \cite[Prop. III.2]{Bou}.
\end{proof}

\begin{corollary}
\label{cor:countable set}
  The set $\cA$ of conjugacy classes of generalized isometry groups is countable.
\end{corollary}

\begin{proof}
  It follows as in \cite[Lem. VI.3]{Bou} from the classification of reductive Lie groups and their actions \cite{Pal}.
\end{proof}

In the reverse direction, we have the following partial result:
\begin{lemma}
  \label{lem: any cpct isometry group}
  For any compact subgroup $G$ of $\bfGDiff_H$, there exists $(g,\om)\in\GM$ such that $G\subset \Isom_H(g,\om)$.
\end{lemma}

\begin{proof}
  Average any generalized metric $V_+$ with respect to the Haar measure $d\mu_G$ on $G$,
  $$
  V_+^a:=\int_{\psi\in G} \psi\cdot V_+\: d\mu_G(\psi),
  $$
  so that the averaged metric $V_+^a$ is given by a pair $(g,\om)$ such that  $G\subset \Isom_H(g,\om)$.
\end{proof}

We will obtain a more precise statement describing the groups $G$ that are generalized 
isometry groups for a generalized metric,
and not just subgroups of generalized isometry groups.

Let $G$ be any compact subgroup of $\bfGDiff_H$, isomorphic to $\Pi(G)$. 
The group $G$ is then a compact Lie group and acts
on $M$ via the (left) action of $\Pi(G)$ on $M$:
$$
\begin{array}{ccc}
  G\times M & \rightarrow & M \\
  (\psi,m) & \mapsto & \Pi(\psi)(m).
\end{array}
$$
\begin{remark}
  Note that all the results from \cite{Bou} on compact Lie groups acting on compact
  manifolds extend to $G$, and hence to any generalized isometry group, by Lemma
  \ref{lem:description generalized isometries}.
\end{remark}
For any $m\in M$, denote by $G_m$ the stabilizer of $m$ under the $G$-action.
A {\it principal orbit} for the $G$-action is the orbit of a point $m\in M$
such that the image of $G_m$ in $GL(T_mM/T_m(G\cdot m))$ is the identity (this property does not depend on the choice of the point in the orbit).

The group $G_m$ acts on $ S^2T_m^*M\times \Lambda^2T_m^*M$ with twisted action:
\begin{equation}
  \label{eq:action on fixed fiber generalized metric}
  \begin{array}{ccc}
    G_m \times (S^2T_m^*M\times \Lambda^2T_m^*M) & \rightarrow & S^2T_m^*M\times \Lambda^2T_m^*M\\
    ((\phi,B)  ,\: g_m,\om_m) & \mapsto & ( \phi\cdot g_m, \phi\cdot \om_m+ B(m)).
  \end{array}
\end{equation}
where  $\phi\cdot g_m=g_m((D\phi)_m\:\cdot\:,(D\phi)_m\:\cdot\:)$ and $\phi\cdot \om_m=\om_m((D\phi)_m\:\cdot\:,(D\phi)_m\:\cdot\:)$ come from the action by pullback on tensors.
We can then define $(S^2T_m^*M\times \Lambda^2T_m^*M)_{G_m}$ to be the space
of invariants (fixed elements) under the $G_m$-action.

We will need the following proposition, which is an adaptation of \cite[Prop. III.12]{Bou}:
\begin{proposition}
  \label{prop:sub group isometries}
  Consider $G'=\Isom_H(g,\om)$ for some generalized metric $(g,\om)$, and $G\subset G'$ a compact subgroup.
  The following are equivalent:
  \begin{enumerate}
  \item[i)] There exists a generalized metric $(g,\om)\in\GM$ that is not invariant under the action of $G'$ or one of its conjugate but that is $G$-invariant.
  \item[ii)] There is a $G$-orbit (for the left
    action on $M$) that is not a $G'$-orbit or, there is a point $m$ of a principal orbit for $G$ such that 
    $(\Lambda^2T_m^*M\times S^2T_m^*M)_{G_m}\neq (\Lambda^2T_m^*M\times S^2T_m^*M)_{G'_m}$.
  \end{enumerate}
\end{proposition}

\begin{proof}
  The proof follows the proof of \cite[Prop. III.12]{Bou}, 
  considering elements of
  $S^2T^*M\times \Om^2$ instead of $S^2T^*M$, and
  using the twisted action (\ref{eq:action on fixed fiber generalized metric}). 
  
  The proofs of \cite[Prop. III.12, $(i)\Longrightarrow(ii)$]{Bou}
  and \cite[Prop. III.12, $(ii)\Longrightarrow(i)$]{Bou} in the case of $G$ and $G'$ having the same orbits
  extend directly to our setting. 
  
  We shall give some details to extend \cite[Prop. III.12, $(ii)\Longrightarrow(i)$]{Bou}
  in the case when $G'$ and $G$ have different orbits. In this situation, as in the proof of
  \cite[Prop. III.12, $(ii)\Longrightarrow(i)$]{Bou}, consider a perturbation
  of $(g,\om)$ of the form $(g,\om)+t(h,\om_h)$ for small $t$. Using the slice result for generalized metrics,
  a necessary condition for $\Isom_H(g+t h,\om+t\om_h)$ to be conjugated to $G'$ is that
  \begin{equation}
    \label{eq:necessary condition conjugacy metric perturbation}
    \exists u\in \Gamma(TM) \textrm{ such that }   \phi^*h - h = \cL_u g
  \end{equation}
  for some element $(\phi,B)\in G'$. 
  To see this, just restrict to the metric component in the Lie derivative considerations.  Thus, if we can find $(h,\om_h)$ that is $\Pi(G)$-invariant, and $(\phi,B)\in G'$ such that (\ref{eq:necessary condition conjugacy metric perturbation})
  is not satisfied,  then  $G\subset \Isom_H( g+t h,\om+t\om_h)$ for small $t$,  $\Isom_H(g+t h,\om+t\om_h)$ is not conjugated to $G'$, and we are done.
  
  Note that $\Pi(G)\subset \Isom_H(g)\subset \Isom(g)$.  Moreover, as $G$ and $G'$ have
  different orbits on $M$, so do $\Pi(G)$ and $\Isom(g)$. As in the proof of \cite[Prop. III.12]{Bou},
  there exists $h\in S^2T^*M$ and $\phi\in\Isom(g)$ such that $h$ is $\Pi(G)$-invariant
  and (\ref{eq:necessary condition conjugacy metric perturbation}) is not satisfied. As $\Pi(G')\subset\Isom(g)$
  and $\Pi(G)$ and $\Pi(G')$ have different orbits, we can actually take $\phi\in \Pi(G')$.
  Using the Haar mesure on $G$, and the average trick,
  we can find a $\Pi(G)$-invariant pair $(h,\om_h)$. Setting $(\phi,\om-\phi^*\om)\in G'$,
  the equation (\ref{eq:necessary condition conjugacy metric perturbation}) is not satisfied and
  the perturbation $(g+t h,\om+t\om_h)$ gives the required generalized metric.
  
  The rest of the proof readily follows Bourguignon's arguments. 
\end{proof}

\begin{remark}
  In the proof of Proposition \ref{prop:sub group isometries}, for $(ii)\Longrightarrow(i)$,
  one starts with a metric $(g',\om')$ such that $G\subset \Isom_H(g',\om')$. Then
  the metric $(g,\om)$ such that $\Isom_H(g,\om)=G$ can be constructed as close as we want
  to $(g',\om')$ in $\GM$. 
\end{remark}

We are now ready to give a characterization of the compact subgroups $G$ of $\bfGDiff_H$ such that
there is a generalized metric $(g,\om)$ with $G=\Isom_H(g,\om)$:

\begin{proposition}
  \label{prop:caracterization isometries}
  Let $G$ be a compact subgroup of $\bfGDiff_H$. The following are equivalent:
  \begin{enumerate}
  \item[i)] There exists a generalized metric $(g,\om)\in\GM$ such that $G=\Isom_H(g,\om)$.
  \item[ii)] For any compact subgroup $G'$ of $\bfGDiff_H$ such that $G\subsetneq G'$ and having the same orbits 
    as $G$, there is a point $m$ of a principal orbit for $G$ such that 
    $(\Lambda^2T_m^*M\times S^2T_m^*M)_{G_m}\neq (\Lambda^2T_m^*M\times S^2T_m^*M)_{G'_m}$.
  \end{enumerate}
\end{proposition}

\begin{proof}
The proof follows the proof of \cite[Thm. III.23]{Bou}.
\end{proof}

\subsection{Centralizer and normalizer of a generalized isometry group}
In this subsection we show that the spaces
$\cR_{(G)}$ are ILH manifolds. First, we need a strong ILH Lie group structure on the groups acting on $\GM_G$.
Set $G=\Isom_H(g,\om)$ for some generalized metric $(g,\om)$. We begin with the study of 
the centralizer $\GDiff_H^G$ of $G$ in $\bfGDiff_H$:

\begin{lemma}
  \label{lem:invariant GDiff}
  The centralizer $\GDiff_H^G$ of $G$ in $\bfGDiff_H$ is a strong ILH Lie subgroup of $\bfGDiff_H$.
\end{lemma}

\begin{proof}
  We first reduce to the case where $\om=0$. To do this consider the map 
  \begin{equation}
    \label{eq:change of trivialisation of E}
    \begin{array}{cccc}
      \cA_\om:  & \Diff\ltimes \Om^2 & \rightarrow & \Diff\ltimes \Om^2 \\
      &  (\phi,B) & \mapsto & ( \phi,B-\om+\phi^*\om),
    \end{array}
  \end{equation}
  which is just the conjugation by the $B$-field $(\Id,-\omega)$. The maps $\cA_\om$ and its inverse $\cA_{-\omega}$ are $\cC^{\infty,\infty}$ ILH normal diffeomorphisms
  by Proposition \ref{prop:Diff action forms}. Moreover, they are topological group isomorphisms.
  We deduce that the precomposition by $\cA_\om$ or its inverse preserve the strong ILH Lie group structure of $\bfDiff^G\ltimes (\bfOm^2)^G$ (that is the regularity
  of product and inverse in local charts at the origin).
  Note that under $\cA_\om$, the group $\GDiff_H$ is sent to $\GDiff_{H+d\om}$ and the group $\Isom_H(g,\om)$ is sent to $\Isom_{H+d\om}(g,0)$.
  Similarly, the centralizer of $G=\Isom_H(g,\om)$ in $\bfDiff^G\ltimes (\bfOm^2)^G$
  is sent to the centralizer of $\Isom_H(g,0)$ in $\bfDiff^G\ltimes (\bfOm^2)^G$. Then, we can assume
  without loss of generality that $\om=0$ and $G=\Isom_H(g,0)$.

  Consider now the centralizer of $G$ in $\bfDiff^G\ltimes (\bfOm^2)^G$, denoted $(\bfOm^2)^G\rtimes \bfDiff^G$.
  This is the semi-direct product of $G$-invariant $2$-forms in $\bfOm^2$ and the centralizer
  of $\Pi(G)$ in $\bfDiff$. By \cite[Prop. V.8]{Bou}, $\bfDiff^G$ is a strong ILH Lie sub-group of $\bfDiff$,
  and its connected component of identity is totally geodesic in $\bfDiff$ with respect to an $L^2$ metric
  induced by $g$. Similarly, using the projection operator
  \begin{equation}
    \begin{array}{ccc}
      \bfOm^2 & \rightarrow & \bfOm^2 \\
      B    &   \mapsto   &  \int_G \phi^*B \: d\mu_G((0,\phi)),
    \end{array}
  \end{equation}
  where $d\mu_G$ denotes the Haar mesure on $G$, the group $(\bfOm^2)^G$ is a totally geodesic strong ILH Lie subgroup of $\bfOm^2$.
  Using the chart at the identity (\ref{eq:chart exact case}), we see that
  $(\bfOm^2)^G\rtimes \bfDiff^G$ is a (totally geodesic) strong ILH Lie-subgroup of $\bfDiff^G\ltimes (\bfOm^2)^G$.

  Finally, we can use the same argument as in the proof of Theorem \ref{theo:strong ILH subgroup exact case} to obtain
  that $\GDiff_H^G$ is a strong ILH Lie subgroup of $\bfDiff^G\ltimes (\bfOm^2)^G$.
  As the change of chart to obtain $\bfGDiff_H^G$ as a strong ILH Lie subgroup of $\bfDiff^G\ltimes (\bfOm^2)^G$
  and the change of chart to obtain $\bfGDiff_H$ as a strong ILH Lie subgroup of $\bfDiff^G\ltimes (\bfOm^2)^G$
  are both obtained by applying the implicit function theorem to the same map, we conclude that
  $\bfGDiff_H^G$ is a strong ILH Lie subgroup of $\bfGDiff_H$.
\end{proof}

\begin{remark}
  In the proof of Theorem \ref{theo:Stratification ILH exact case}, we do not really need the fact that
  $\bfGDiff_H^G$ is a strong ILH Lie subgroup of $\bfGDiff_H$, we just need the fact that this is a strong ILH Lie subgroup
  of $\bfDiff^G\ltimes (\bfOm^2)^G$.
\end{remark}

Let $N_G$ be the normalizer of $G$ in $\bfGDiff_H$. Then:

\begin{corollary}
  The quotient $G\backslash N_G$ is an ILH Lie group whose topology is induced by the inclusion
  $G\backslash N_G \rightarrow G\backslash \bfGDiff_G$.
\end{corollary}

\begin{proof}
  This follows as in \cite[Thm. V.22]{Bou}, and the discussion on abstract Lie groups in 
  \cite[Sec. V]{Bou}.
\end{proof}

\subsection{The ILH stratification}
Let $G=\Isom_H(g,\om)$ as before.
We turn now to the quotient space $\GR_{(G)}$:

\begin{proposition}
  \label{prop:ILH generalized G invariant metrics}
  The space $\GR_{(G)}$ is a strong ILH manifold.
  Moreover, its topology coincides with the topology induced by the inclusion in $\cG\cR$.
\end{proposition}
\begin{proof}
  We use the slice result \cite[Prop. II.16]{Bou} for the action of $G\backslash N_G$ on $\GM_G$.
  The proof follows as in \cite[Thm. V.24]{Bou}.
\end{proof}

Now we can prove Theorem \ref{theo:Stratification ILH exact case}:

\begin{proof}[Proof of Theorem \ref{theo:Stratification ILH exact case}] The proof follows the one of \cite[Thm. VI.4]{Bou}.

  From Corollary \ref{cor:countable set}, the set $\cA$ of conjugacy classes of generalized isometry groups is countable. It is partially ordered by inclusion.
  It is also clear that $(\cG\cR_{(G)})_{(G)\in \cA}$ is a partition of $\cG\cR$.
  
  From Proposition \ref{prop:ILH generalized G invariant metrics}, the spaces $\GR_{(G)}$ are strong ILH manifolds whose topology coincides with the induced topology from $\cG\cR$.
  
  Then, one has to show the intersection properties of the strata. Let $(G)$ and $(G')$ be two conjugacy classes
  in $\cA$. Assume that $\cG\cR_{(G')}\cap\overline{\cG\cR_{(G)}}\neq \varnothing$. Let $V_+'$ be a generalized metric whose class is in $\cG\cR_{(G')}\cap\overline{\cG\cR_{(G)}}$, with isometry group $G'$.
  Applying the slice Theorem \ref{theo:slice exact case full group} at $V_+'$, we find a slice $\cS$ parametrizing generalized metrics nearby $V_+'$.
  Then, any element in this slice has generalised
  isometry group conjugated to a subgroup of $G'$.
  From this we deduce that $G$ is conjugated to a subgroup of $G'$,
  hence $(G)\subsetneq (G')$. 
  
Last, let $(G)$ and $(G')$ be two conjugacy classes
  in $\cA$ such that $(G)\subsetneq (G')$. Then $G'$ is the isometry group of some generalized metric $V_+'$.
  We can assume up to conjugation that $G$ is a subgroup of $G'$ and the generalized isometry group of another generalized metric.
  To conclude, one needs to construct a generalized metric $V_\ep$, arbitrarily close to $V_+'$, such that the generalized
  isometry group of $V_\ep$ is $G$.
  This is done by following the argument of the proof of Proposition \ref{prop:caracterization isometries}.
\end{proof}

As a corollary, we also obtain generalized versions of other results in \cite{Bou}.
\begin{corollary}
  We have the following:
  \begin{enumerate}
  \item The space $\GM_G$ is an open dense subset of $\GM_G^1$, where $\GM_G^1$ is the space of generalized metrics whose generalized isometry
    group contains $G$.
  \item The space $\GM_{\lbrace 1\rbrace}$ is an open dense subset of $\GM$ as long as $\dim M\geq 2$.
  \item In a neighbourhood of a generalized metric with non-trivial generalized isometry group, $\GM$ is not an ILH manifold.
  \item In each homotopy class of a loop in $\GR$, there is a loop that can be lifted to $\GM$.
  \end{enumerate}
\end{corollary}

\subsection{Relation between the strata of the moduli spaces}
\label{sec:relations moduli}

In this section we relate the moduli space of generalized metrics of an exact Courant algebroid, $\GR=\GM/\GDiff$, with the moduli space of usual metrics, or Riemannian structures, of the base manifold, $\cR=\cM/\Diff$, and other moduli spaces naturally arising. In particular, we study the relation between the isometries of a generalized metric and the isometries of the usual metric to which it projects, and use this to compare the strata of the moduli spaces.

Since the definition of a generalized metric on a Courant algebroid does not depend on the bracket, we start by introducing the orbit space for the action of the group $\OO_\pi$,
$$
\cGR^\pi:=\cGM/\OO_\pi.
$$
Following Sections \ref{sec:moduliGM} and \ref{sec:stratification}, one can show that $\cGR^\pi$ admits an ILH stratification
$$
\cGR^\pi=\bigcup_{(G)\in\cA_\pi}(\cGR^\pi)_{(G)},
$$
where $(G)\in \cA_\pi$ runs through all the conjugacy classes of isotropy groups, $\Isom(V_+)$, for the $\OO_\pi$-action in $\cGM$. Again, we choose a splitting $\s$, in such a way that 
\begin{align*}
  E\simeq_\s (TM+T^*M)_H,&& V_+\simeq_\s (g,\om),
\end{align*}
to provide proofs of the results. Recall that $g$ does not depend on the splitting.

\begin{proposition}
  The projection $\cGR^\pi\rightarrow\cR$ defined by $[(g,\om)]\mapsto [g]$ is an isomorphism of ILH stratified topological spaces.
\end{proposition}

\begin{proof}
  This map and its inverse $[g]\mapsto [(g,\om)]$ are well defined and continuous. On the other hand, an element $(\phi,B)\in\OO_\pi$ stabilizing $(g,\om) \in {\GM}$ should satisfy $\phi^*g=g$, i.e., $\phi$ is an isometry of $g$, and $B=\phi^*\omega-\omega$, i.e., $B$ is completely determined by $\omega$ and $\phi$. We thus have that $\Isom_H(g,\omega)\cong \Isom(g)\subset \Diff_{[H]}$, so the ILH stratification is preserved.
\end{proof}

We are mainly concerned with the moduli space of generalized metrics
$$\cGR^H := \frac{\GM}{\GDiff_{H}},$$
which, as $\GDiff_{H}$ is a subgroup of $\OO_\pi$, projects onto $\cGR^\pi\cong\cR$. 

We have another natural projection of $\cGR^H$. Consider the action of the ILH Lie group $\DH$ on the space of metrics $\cM$. Its orbit space is, by following \cite{Bou} and Sections \ref{sec:moduliGM} and \ref{sec:stratification}, an ILH stratified space 
$$
\cRuH:=\cM/\DH=\bigcup_{(G)\in \cBH} \cRuH_{(G)}
$$
where $\cBH$ is the set of conjugacy classes of isometry groups in $\DH$.
Again, the projection
$$
\pi_H: \cGR^H \rightarrow \cRuH
$$
is well defined and maps strata inside strata. In order to express more precisely the relation between the two stratifications, we look at  the isometry groups under the conjugation by $\OO_\pi$.

\begin{proposition}\label{prop:action-isom-groups}
  For $(g,\om)\in \GM$ and $(\psi,C)\in \OO_\pi$ we have 
  \begin{equation}
    (\psi,C)^{-1}\Isom_H((g,\om))(\psi,C)=\Isom_{\psi^*(H+dC)}((\psi^*g,\psi^*\omega-C)),\label{eq:conjugation-isom}
  \end{equation}
  where both isotropy groups are subgroups of $\OO_\pi$.
\end{proposition}

\begin{proof}
  On the one hand, $(\phi,B)$ fixes $(g,\om)$ if and only if $(\psi,C)^{-1}(\phi,B)(\psi,C)$ fixes $(\psi,C)\cdot (g,\om)$, as the action is on the right. Note that
  $$  (\psi,C)^{-1}(\phi,B)(\psi,C)= (\psi^{-1}\phi\psi,\psi^*(B-\phi^*C+C)),$$ On the other hand, the condition $\phi^*H-H=dB$ is equivalent to
  $$(\psi^{-1}\phi\psi)^*(\psi^*(H+dC))-(\psi^*(H+dC))=d(\psi^*(B+\phi^*C-C)).$$
\end{proof}

From the projection $\pi_H([(g,\om)])=[g]$, we can express the preimage of a stratum $\cRuH_{(G)}$ in $\cRuH$, for $G=\Isom_{[H]}(g)\subset\DH$, as  
$$
\pi_H^{-1}(\cRuH_{(G)})=\bigcup_{\om\in\Om^2}\cGR^H_{(\Isom_{H}(g,\om))}.
$$
The following proposition uses Hodge theory to give a more concrete description. Note first that 
$$\Isom_{[H]}(g,\om)=\{(\phi,\phi^*\om-\om)\st \phi\in\Isom_{[H]}(g), \phi^*(H-d\om)=H-d\om\}.$$

\begin{proposition}
  \label{prop:image inverse stratum}
  Given $g\in\cM$, there exists $C\in \Omega^2$ such that 
  $$ \Isom_{H}(g,\om) = (\Id,-C) G(\om-C) (\Id,C),$$
  where we set $G(\om):=\lbrace \phi \in G \;\vert\; \phi^*\om = \om\rbrace,$ and identify $G(\om)\subset\Diff_{[H]}$ with $G(\om)\times\lbrace 0 \rbrace\subset\GDiff_H$.
  As a consequence, we have the description
  $$
  \pi_H^{-1}(\cRuH_{(G)})=\bigcup_{\om\in\Om^2}\cGR^H_{\left((\Id,-C) G(\om)(\Id,C)\right)}.
  $$
\end{proposition}

\begin{proof}

  We will use the conjugation in \eqref{eq:conjugation-isom}. We average $H$ under the $\Isom_{[H]}(g)$-action to obtain $H'$, which belongs to the same cohomology class of $H$ and hence there is $C_1\in\Omega_2$ such that $H'=H+dC_1$. The conjugation of $\Isom_H(g,\om)$ by $(\Id,C_1)$ is then
  $$\Isom_{H'}(g,\om-C_1).$$
  The metric $g$ being fixed, we have a Hodge decomposition on forms
  $$
  \Om^2=d^*\Om^3 \oplus \ker(d).
  $$
  Let $C_2$ be the projection of $\omega-C_1$ onto $\ker(d)$. The conjugation of $\Isom_{H'}(g,\om-C_1)$ by $(\Id,C_2)$ is $\Isom_{H'}(g,\om')$, where $\om'=\om-C_1-C_2$ belongs to $d^*\Om^3$ and can be written as $\om'=d^*d\tilde\om$ for $\tilde\om\in\Om^2$.
  We now have the group
  $$
  \Isom_{H'}(g,\om')=\lbrace (\phi,\phi^*\om'-\om') \st \phi\in \Isom_{[H]}(g), \phi^*(d\om')=d\om'\rbrace.
  $$
  Take  $\phi\in G$. The condition $\phi^*d\om'=d\om'$ becomes $\phi^*\Delta_gd\tilde\om'=\Delta_gd\tilde\om'$.  As $\phi\in \Isom(g)$, we deduce $\Delta_g\phi^*d\tilde\om'=\Delta_gd\tilde\om'$. The laplacian being an isomorphism on the orthogonal of harmonic forms,  $\phi^*d\tilde\om'=d\tilde\om'$. But $\phi$ commutes with the Hodge dual for $g$, so $\phi^*\om'=\om'$. 
  Thus,
  $$
  \Isom_{H'} (g,\om')=\lbrace (\phi,0) \st \phi\in \Isom_{[H]}(g), \phi^*\om'=\om'\rbrace=G(\omega')\times \{0\},
  $$
  and $C=C_1+C_2$ gives the first part of theorem. As $\omega$ running through $\Omega^2$ corresponds to $\omega-C$ running through $\Omega^2$, the second part follows.
\end{proof}

\begin{remark}
  \label{rem:carac-isom(om)}
  Consider $g\in\cM$ and $G\subset \Isom_{[H]}(g)$. By adapting the arguments in \cite[Prop. III.12]{Bou}, there exists $\om\in\Om^2$ such that $G=G(\om)$ if and only if, for all compact subgroups $G'\subset\Isom_{[H]}$ containing $G$ and having the same orbits as $G$, there is a point $m$ of a principal orbit for $G$ such that $(\Lambda^2T_m^*M)_{G_m}\neq (\Lambda^2T_m^*M)_{G'_m}$. 
\end{remark}

The relation between $\cGRuH$ and $\cR$, can be described by the sequence
$$
\cGR^H \rightarrow \cRuH \rightarrow \cR
$$
where the last arrow is the projection $\cM/\DH \rightarrow \cM/\Diff$, whose fibre is isomorphic to the {\it mapping class type} group $\Diff/\DH$. The preimage of a strata $\cR_{(G)}\subset \cR$ consists of 
$$\bigcup_{(G')\in \mathcal{A_{[H]}^G}} \cRuH_{(G')} $$
where $\mathcal{A_{[H]}^G}$ is the set of $\Diff_{[H]}$-conjugacy classes for the groups in the $\Diff$-conjugacy class $(G)$.

Summarizing, we have described the relation between the strata of the following spaces
$$ \cGR^H \to \cR^{[H]} \to \cR\cong \cGR^\pi.$$

\section{Odd exact Courant algebroids}
\label{sec:odd-exact}

We have focused so far on exact Courant algebroids, as this is the best-known case. A future aim of our work is the exploitation of these results for the class of transitive Courant algebroids, those of the form $TM+\ad P + T^*M$ for $P$ a principal $G$-bundle. The simplest instance in this class is the bundle $TM+1+T^*M$, where $1$ denotes a trivial line bundle over $M$. Indeed, as $TM+\ad P + T^*M$ is related to solutions of the Strominger system \cite{grt}, $TM+1+T^*M$ corresponds to the abelian case $G=\SSS^1$.

We briefly introduce the generalized geometry arising from $TM+1+T^*M$, and state the main results, drawing the analogies with the exact case and stressing the relevant technical or conceptual differences.

\subsection{Basics on $B_n$-generalized geometry}

The vector bundle $TM+1+T^*M$, whose sections we denote by $u+f+\alpha$ and $v+g+\beta$, has a Courant algebroid structure consisting of the pairing $$\la u+f + \alpha, u+f+\alpha\ra = i_u\alpha + f^2,$$ the projection $\pi$ to $TM$ as the anchor map, and the Dorfman bracket
\begin{equation}\label{eq:bracket-Bn}
 [u+f+\alpha,v+g+\beta] =  [u,v] + u(g)-v(f)+ L_u \beta - i_v d\alpha  + 2g df.
 \end{equation}

The study of the generalized geometry on this bundle and its twisted versions is called generalized geometry of type $B_n$, as the pairing has signature $(n+1,n)$, whose group of special orthogonal transformations is $\SO(n+1,n)$, a real form of a complex group of type $B_n$. This geometry was defined in \cite{Bar} and developed in \cite{Rubio}, \cite{Rubioth} and \cite{Rubiopr}.

An {\bf odd exact Courant algebroid} \cite{Rubiopr} is a transitive Courant algebroid $E$, i.e., $\pi$ is surjective, such that $\rk E = 2\dim M + 1$ is satisfied\footnote{An exact Courant algebroid (Section \ref{sec:GDIff}) is equivalently defined as a transitive Courant algebroid $E$ such that $\rk E=2\dim M$.}. They fit into the diagram  
\begin{equation}\label{eq:diag-odd-exact}
\xymatrix{
	T^*M \ar[d] & &  \\ 												
	Q^* \ar[r] \ar[d]   & E \ar[r]^\pi  & TM,   \\
	1           & &}
\end{equation}
where $Q$ is the Lie algebroid $E/T^*M$ and the row and the column are exact. Choosing an isotropic splitting $\s:T\to E$ determines a unique splitting $1\to Q^*$ whose image is orthogonal to $\s(TM)$, and hence gives an isomorphism of vector bundles $E\simeq_{\s} TM+1+T^*M$. The resulting bracket on $TM+1+T^*M$ is the bracket in \eqref{eq:bracket-Bn} twisted, by  $H\in \Omega^3$ and $F\in \Omega_2^{cl}$ such that $dH+F\wedge F=0$, as follows: 
\begin{align*}
[u+f+\alpha,v+g+\beta]_{H,F} = {} & [u,v] + u(g)-v(f)+ i_ui_vF  + L_u \beta - i_v d\alpha  \\
&   + 2g df + 2(g i_u F-fi_vF  ) + i_ui_vH. 
\end{align*}
We refer to this Courant algebroid as $(TM+1+T^*M)_{H,F}$.



\subsection{The ILH Lie group structure of the automorphism group}

One of the main novelties of $B_n$-geometry is the appearance, apart from $B$-fields, of $A$-fields, which show non-abelianness. Define $\Omega^{2+1}$ to be the space $\Omega_2\times \Omega_1$. When considered as a Lie group, the product in $\Omega^{2+1}$ is  $$(B,A)(B',A')=(B+B'+A\wedge A',A+A').$$
The group of orthogonal transformations of an odd exact Courant algebroid commuting with the anchor are then described, by choosing a splitting, as
$$\OO_\pi(E) \simeq_{\s} \Diff \ltimes \Omega^{2+1}.$$
 
Analogously, the automorphism group $\Aut(E)$ of an odd exact Courant algebroid is isomorphic to the automorphism group  of $(TM+1+T^*M)_{H,F}$, which is given by
\begin{align*}
\GDiff_{H,F} =\{  \psi\ltimes (B,A)\in {} & \Diff \ltimes \Omega^{2+1} \st
\psi^*H-H=dB-A\wedge(2F+dA), \psi^*F - F = dA         \},
\end{align*}
with product
\begin{align}\label{eq:product abelian case}
(\psi\ltimes (B,A))(\psi'\ltimes (B',A')) & = \psi \psi' \ltimes (\psi'^*B+B'+\psi'^*A \wedge
A',\psi'^*A+A'),
\end{align}
and Lie algebra
$$\gdiff_{H,F} = \{(u,(b,a))\in \Gamma(TM)\times \Omega^{2+1} \st  d(\iota_u H-b)+2(\iota_uF+a)\wedge F=0, d(\iota_uF-a)=0\}.$$

Analogously to Proposition \ref{prop:ILH-on-Opi}, the group $O_\pi(E)\simeq_\s \Diff\ltimes \Omega^{2+1}$ carries a strong ILH Lie group structure, independent of the choice of splitting $\s$, modelled on
$$\lbrace \bfGamma(TM)\times\bfOm^{2+1}, \Gamma(TM)^{k+1}\times\Om^{2+1,k}, k\geq n+5 \rbrace.$$  A chart at the origin $(\xi',U')$ on $U'=U\times \Om^{2+1,n+5}$, with $(\xi,U)$ a chart of $\Diff$, is given by
\begin{equation}
\label{eq:chart abelian case} 
\begin{array}{cccc}
\xi': & U'\cap(\bfGamma(TM) \times \bfOm^{2+1}) & \rightarrow &  \bfDiff\ltimes \bfOm^{2+1} \\
& (u,(b,a)) & \mapsto & (\xi(u),(b,a)).
\end{array}
\end{equation}


Denote the space $\Omega^3\times \Omega^2$ by $\Omega^{3+2}$. In the remaining of this section, the ILH chains associated to the spaces $\bfGamma(TM)$, $\bfOm^{2+1}$ and $\bfOm^{3+2}$ are $\lbrace \bfGamma(TM), \Gamma^{k+1}, k\geq n+5 \rbrace$, $\lbrace \bfOm^{2+1}, \Om^{2+1,k}, k\geq n+5 \rbrace$ and $\lbrace \bfOm^{3+2}, \Om^{3+2,k-1}, k\geq n+5 \rbrace$. Moreover, we fix a twist $(H,F)\in \Omega^{3+2}$ with $dF=0$ and $dH+F\wedge F=0$.

We describe $\GDiff_{H,F}$ as $\tilde\rho((H,F),\cdot)^{-1}(H,F)$ for the map
\begin{equation}
\label{eq:defining map abelian case}
\begin{array}{cccc}
\tilde\rho : & \bfOm^{3+2} \times (\bfDiff \ltimes \bfOm^{2+1})   & \rightarrow & \bfOm^{3+2}\\
&  ((H',F'), (\phi,(B,A))) & \mapsto & (\phi^*H' - dB + A\wedge(2F'+dA),\phi^*F'-dA).
\end{array}
\end{equation}
In order to endow  $\GDiff_{H,F}$ with an ILH Lie group structure, we use the implicit function theorem. First, we look at the map $\tilde\rho((H,F),\cdot)$ using the chart $(\xi',U')$. Define
\begin{equation*}
\begin{array}{cccc}
\Phi: & U'\cap (\bfDiff \ltimes \bfOm^{2+1}) & \rightarrow & \Om^{3+2} \\
&     (u,(b,a))          &  \mapsto &   \tilde\rho((H,F),\xi'(u,(b,a))),
\end{array}
\end{equation*}
whose derivative at zero is given, as $(d\xi)_0=Id$, by  
\begin{equation}
\label{eq:dphi abelian case}
d\Phi_0 (u,(b,a))=(d(\iota_uH - b)+2(\iota_uF+A)\wedge F,d(\iota_uF-A)).
\end{equation}
The map $\Phi$ is a $\cC^{\infty,\infty}$ ILH normal map, as in Proposition \ref{prop:Diff action forms}.  A right inverse for $d\Phi_0$ in the category of $\cC^{\infty,2}$ ILH normal maps is given, using the notation of \eqref{eq:right-inverse}, by the restriction of the following operator to the image of $d\Phi_0$:
\begin{equation*}
\begin{array}{cccc}
\bfB: & \bfOm^{3+2} & \rightarrow & \bfGamma(TM)\ltimes \bfOm^{2+1}  \\
&  (h,f)    & \mapsto     & (0,-d^*\bbG(h+2(d^*\bbG f)\wedge F) ,-d^*\bbG f).
\end{array}
\end{equation*}

From the Implicit Function Theorem \ref{theo:implicitfunctiontheorem}, $\GDiff_{H,F}$ is a strong ILH (sub)manifold of $\Diff\ltimes \Omega^{2+1}$. As we have proved that $\Diff\ltimes \Omega^{2+1}$ is a strong ILH Lie group, the regularity of the local inverse, product, etc. follow on $\GDiff_{H,F}$, thus obtaining the following.
\begin{proposition}
	The group $\bfGDiff_{H,F}$ is a strong ILH Lie subgroup of $\bfDiff\ltimes \bfOm^{2+1}$.
\end{proposition}

\begin{remark}
	\label{rem:elliptic complex F twisted}
	Implicitly, we have used here the fact that the complex
	\begin{equation*}
	\label{eq:complex elliptic F twisted}
	\begin{array}{cccc}
	d_F:& \Om^{i+1,i} & \rightarrow & \Om^{i+2,i+1} \\
	&(b,a ) & \mapsto & (db- 2(-1)^i a\wedge F, da)
	\end{array}
	\end{equation*}
	is an elliptic complex of differential operators. This complex encodes some variations of Courant algebroids structures
	on $TM+1+T^*$ modulo symmetries fixing the anchor.
\end{remark}

Denote by $\GDiff_{H,F}^e$ the group of exact generalized diffeomorphisms \cite{Rubioth}. We will show that this is a strong ILH Lie
subgroup of $\bfGDiff_{H,F}$ by means of Frobenius' theorem. Its Lie algebra  $\gdiff_{H,F}^e$   is given by
\begin{equation*}
\label{eq:liegdiffexact abelian case}
\lbrace (u,(b,a))\in \Gamma(TM)\times\Om^{2+1} \;\vert\; \iota_uH-b=d\xi+2fF, \iota_uF-a = df \textrm{ for some } \xi\in\Omega^1, f\in\cCi(M)  \rbrace.
\end{equation*}
As a subalgebra of $\gdiff_{H,F}$, it can be described as the space of triples $(u,(b,a))$ such that
\begin{equation*}
\left\{
\begin{array}{ll}
\iota_uH-b-2(\bbG d^*(\iota_uF-a))F\in  (\cH^2)^\perp\\
\iota_uF-a \in (\cH^1)^\perp
\end{array} 
\right.
\end{equation*}
where $\cH^i$ denotes the space of harmonic $i$-forms.
Let $s=h^2(M)+h^1(M)$, where $h^i(M)=\dim H^i(M)$, and take $(e_i)_{1\leq i\leq h^2}$ and $(f_i)_{1\leq i\leq h^1}$, 
bases of the harmonic $2$-forms and $1$-forms, respectively. Define the map:
\begin{equation*}
\begin{array}{cccc}
I : & \bfOm^{2+1} & \rightarrow & \RR^s\\
&  (\om_2,\om_1)  &  \mapsto    &  (\langle \om_2 , e_i \rangle_{L^2_g}, \langle \om_1 , f_i \rangle_{L^2_g}).
\end{array}
\end{equation*}
Set also
\begin{equation*}
\begin{array}{cccc}
\kappa_1 : & \gdiff_{H,F} & \rightarrow & \bfOm^{2+1} \\
&  (u,(b,a))   & \mapsto     &  (\iota_uH-b,\iota_uF-a)
\end{array}
\end{equation*}
and
\begin{equation*}
\begin{array}{cccc}
\kappa_2 : & \bfOm^{2+1} & \rightarrow & \bfOm^{2+1} \\
&  (\beta,\alpha)   & \mapsto     &  (\beta-2\bbG d^*(\alpha) F,\alpha).
\end{array}
\end{equation*}
Note that $\gdiff_{H,F}^e=(I\circ\kappa_2\circ\kappa_1)^{-1}(0)$. 
We will consider an extension of $I\circ\kappa_2\circ\kappa_1$ to $T\bfGDiff_{H,F}$, which will be a $\cC^{\infty,\infty}$ ILH normal homomorphism to the trivial bundle over $\bfGDiff_{H,F}$.

First, consider the bundle $B(\bfOm^{2+1},\bfDiff,\tilde T_{\bfOm^{2+1}})$ over $\bfDiff$, where
$\tilde T_{\bfOm^{2+1}}$ is defined as in Definition \ref{def:vector bundle parallel}.
Denote by $\pi: \bfDiff\ltimes \bfOm^{2+1}\rightarrow\bfDiff$ the projection. Then the restriction of
$\pi^*\tilde T_{\bfOm^{2+1}}$ to $\bfGDiff_{H,F}$ defines an ILH vector bundle with fibre $\bfOm^{2+1}$, whose
transition functions are defined by parallel transport on $M$. Extend $\kappa_2$ and $I$
to right-invariant bundle homomorphisms 
$$\tilde \kappa_2:B(\bfOm^{2+1},\bfGDiff_{H,F},\pi^*\tilde T_{\bfOm^{2+1}}) \rightarrow B(\bfOm^{2+1},\bfGDiff_{H,F},\pi^*\tilde T_{\bfOm^{2+1}})$$
and 
$$\tilde I: B(\bfOm^{2+1},\bfGDiff_{H,F},\pi^*\tilde T_{\bfOm^{2+1}}) \rightarrow \bfGDiff_{H,F}\times \RR^s,$$
by setting $\tilde \kappa_2=R_\phi\circ\kappa_2\circ R_\phi^{-1}$ and $\tilde I=R_\phi\circ I\circ R_\phi^{-1}$ for $\phi\in\bfGDiff_{H,F}$.
\begin{lemma}
	\label{lem:extensions of kappa and I}
	The bundle homomorphisms $\tilde I$ and $\tilde k_2$ are $\cC^{\infty,\infty}$ ILH normal.
\end{lemma}
\begin{proof}
	This is a direct consequence of Theorem \ref{theo:extension theorem}[(2)] for $\tilde \kappa_2$
	and Theorem \ref{theo:extension theorem}[(3)] for $\tilde I$. Note that Theorem \ref{theo:extension theorem}
	is stated for $\bfDiff$, but the bundles we consider here are pullbacks of bundles defined on $\bfDiff$, 
	so the extension results hold in our setting as well.
\end{proof}
To apply Frobenius' theorem, it remains to extend the operator $\kappa_1$.
As before, extend  $\kappa_1$ to a right-invariant bundle homomorphism 
$$\tilde \kappa_1:B(\gdiff_{H,F},\bfGDiff_{H,F},T_{\theta_{H,F}})\rightarrow B(\bfOm^{2+1},\bfGDiff_{H,F},\pi^*\tilde T_{\bfOm^{2+1}}),
$$
where $T_{\theta_{H,F}}$ is the defining map for $T\bfGDiff_{H,F}$.
\begin{lemma}
	\label{lem:tangent bundle abelian case}
	The defining map for $T\bfGDiff_{H,F}$ is given,
	for $(u,(b,a))\in \gdiff_{H,F}$, $(\xi(v),(B,A))$ and $(\xi(w),(B',A'))$ in $U'\cap \bfGDiff_{H,F}$, by
	$$
	T_{\theta_{H,F}}( (u,(b,a)), (\xi(v),(B,A)) ,(\xi(w),(B',A')))=(T_\theta(u,\xi(v),\xi(w)),\xi(w)^*b+\xi(w)^*a\wedge A',\xi(w)^*a)
	$$
	where $T_\theta$ is the defining map for $T\bfDiff$, and $(\xi,U)$ is a chart at the origin for $\bfDiff$.
\end{lemma}

\begin{proof}
	Recall the definition of the defining map (\ref{eq:defining map tangent}) for the tangent bundle.
	The result follows from a direct computation of $\theta$ (see Definition \ref{def:ILHLiegroup}),
	using the chart (\ref{eq:chart abelian case}) and the product~(\ref{eq:product abelian case}).
\end{proof}

\begin{lemma}
	\label{lem:extension kappa1}
	The bundle  homomorphism $\tilde \kappa_1$ is $\cC^{\infty,\infty}$ ILH normal.
\end{lemma}

\begin{proof}
	By right invariance, we only need to check regularity of the local expression of $\tilde\kappa_1$ at the identity.
	Set $g=(v,(b,a))\in U'\cap \gdiff_{H,F}$,
	and $(w,(B,A))\in \gdiff_{H,F}$. By definition, the local expression for $\tilde \kappa_1$ at the identity is given by:
	$$
	\Phi_{\kappa_1}(g)(w,(B,A))=(\pi^*\tilde T_{\bfOm^{2+1}}(e,g))\circ R_g\circ \kappa_1\circ R_g^{-1}\circ T_{\theta_{H,F}}(e,g)^{-1}(w,(B,A)).
	$$
	Recall that $T_\theta$ is the defining map for the tangent bundle of $\bfDiff$, and $\tau(\exp_xv(x))$ denotes the parallel transport on the bundle $\Lambda^{2}T^*M \oplus T^*M$ along the geodesic
	$t\mapsto \exp_x(tv(x))$, from $t=0$ to $t=1$.
	Setting $w'=T_\theta(e,\xi(v))^{-1}w$, we compute using Lemma \ref{lem:tangent bundle abelian case}:
	$$
	\Phi_{\kappa_1}(g)(w,(B,A))=(\Phi_1(g)+\Phi_2(g))(w,(B,A))
	$$
	where
	$$
	\Phi_1(g)(w,(B,A))(x)=\tau(\exp_xv(x))^{-1}( \iota_{w'}H,\iota_{w'}F)(\exp_xv(x))
	$$
	and
	$$
	\Phi_2(g)(w,(B,A))(x)= -\tau(\exp_xv(x))^{-1} (\xi(v)^{-1})^*(B-A\wedge a,A)(\exp_xv(x)).
	$$
	The map $\Phi_1$ is the local expression for the (right-invariant pullback of) the right-invariant extension of the map
	$$
	\begin{array}{ccc}
	\bfGamma(TM) & \rightarrow & \bfOm^{2+1}\\
	w &\mapsto & (\iota_w H, \iota_w F)
	\end{array}
	$$
	from $B(\bfGamma(TM),\bfDiff,T_\theta)$ to $B(\bfOm^{2+1},\bfDiff,\tilde T_{\bfOm^{2+1}})$.
	By Theorem \ref{theo:extension theorem}[(1)], this defines a $\cC^{\infty,\infty}$
	ILH normal map. For $\Phi_2$, consider, for example, its second component
	$$
	-\tau(\exp_xv(x))^{-1} (\xi(v)^{-1})^*A(\exp_xv(x))=-A(x)(D(\xi(v)^{-1})\tau(\exp_xv(x))\cdot \;).
	$$
	Thus, using $\Psi_{-1}$ as defined in (\ref{eq:psioperator}),
	$$
	\Phi_2(w,(B,A))=\Psi_{-1}(v,(B-A\wedge a,A)).
	$$
	From Lemma \ref{lem:psioperator}, the map $\Phi_2$ has the required regularity, which ends the proof.
\end{proof}

We can then conclude the following.
\begin{proposition}
	\label{cor:exact subgroup abelian case}
	The group $\bfGDiff_{H,F}^e$ is a strong ILH Lie subgroup of $\bfGDiff_{H,F}$. 
\end{proposition}

\begin{proof}
	From Lemmas \ref{lem:extension kappa1} and \ref{lem:extensions of kappa and I}, the operator $\tilde I\circ\tilde\kappa_2\circ\tilde\kappa_1$,
	right-invariant by construction, is $\cC^{\infty,\infty}$ ILH normal.
	Then the result follows from the lighter version of Frobenius' theorem \ref{theo:Frobenius} referred to in Remark \ref{rem:hypo+trivial}, recalling that $\fgdiff_{H,F}^e=\ker(I\circ\kappa_2\circ\kappa_1)$. 
	The maps
\begin{align*}
	\bfB_I(t_i,s_i) & {} =(\sum t_ie_i, \sum s_i f_i),& 		\bfB_2(\gamma,\delta) & {} =( \gamma , \delta- 2 \bbG d^*\gamma F), & 		\bfB_1(\beta,\alpha) & {} =(0,(\beta,\alpha))
\end{align*}
	provide $\cC^{\infty,\infty}$ ILH normal right inverses for $I$, $\kappa_2$ and $\kappa_1$, respectively.
	If $\bfB$ denotes the composition $\bfB_1\circ \bfB_2\circ I$, then $\gdiff_{H,F}=\ker (I\circ\kappa_2\circ\kappa_1)\oplus \bfB\RR^s$.
	The regularity of $\bfB$ implies that the hypotheses of Frobenius' theorem are satisfied.
	Then, $\ker (\tilde I\circ\tilde\kappa_2\circ\tilde\kappa_1)$ can be integrated to a strong ILH Lie subgroup of
	$\bfGDiff_{H,F}$.
\end{proof}

\begin{remark}
	By definition of $\gdiff_{H,F}$ and $\gdiff_{H,F}^e$ and by constrution of the associated Lie groups, one can check that if $H^1(M,\RR)$ and $H^2(M,\RR)$ vanish, then 	$\bfGDiff_{H,F}^e=\bfGDiff_{H,F}$.
\end{remark}


\subsection{The moduli space of $B_n$-generalized metrics}

A generalized metric on an odd exact Courant algebroid $E$ over an $n$-dimensional manifold $M$ is given by a rank $n+1$ positive definite subbundle $V_+\subset E$. We make use now of diagram \eqref{eq:diag-odd-exact}. The intersection $U:=V_+\cap Q^*$ must be of constant rank $1$ and hence is a line subbundle. Its orthogonal complement $U^\perp\subset V_+$ is positive definite of rank $n$ and projects isomorphically to $TM$ via $\pi_{U^\perp}$, thus yielding a usual metric $g$. The map $\s: u\mapsto i_u g - \pi_{U^\perp}^{-1}u \in E$ is an isotropic splitting and the pair $(g,\s)$ determines $V_+$ completely. In this case, $\s$ determines a compatible splitting $1\to Q^*$, which comes from the isomorphism $V_+\cap Q^*\cong 1$.

The space of isotropic splittings $\Lambda$ is an $\Omega^{2+1}$-torsor, as the difference of two of them gives a map $TM\to Q^*$. This map is determined by the projection $TM\to 1$ and a skew-symmetric map $TM\to T^*M$, i.e., an element of $\Omega^{2+1}$.

Summarizing, with the notation $\cM := \{ g\in \Gamma(S^2 T^*M)\st g \textrm{ is positive definite} \}$,  we have
$$\cGM\cong \cM\times \Lambda \simeq_{\s} \cM \times \Omega^{2+1},$$
showing that $\cGM$ is an ILH manifold modelled on the ILH chain
$$\lbrace  \bfGamma(S^2T^*M)\times\bfOm^{2+1}, \Gamma(S^2T^*M)^k\times\Om^{2+1,k}, k\geq n+5 \rbrace.$$

As before, we have a right ILH action $\rho_{\GM}$ of $F\in O_\pi$ on $V_+\in \cGM$ given by $F^{-1}(V_+)$. 
We fix a splitting $\s$ to write isomorphisms
\begin{align*}\label{eq:with-splitting-Bn}
& \OO_\pi  \simeq_\s \Diff\ltimes \Omega^{2+1}, \qquad \qquad \qquad \cGM\simeq_\s \cM \times \Omega^{2+1},\\
&  V_+\simeq_\s\{u  + r + \gamma(u) + i_u g + i_u \omega - \gamma(u)\gamma -2 r\gamma  \st X\in TM, r\in 1\},
\end{align*}
where $V_+$ corresponds to $\exp(\om,\gamma)\lbrace \iota_Xg +r+X \st X\in TM, r\in 1\rbrace$ with
\begin{equation*}
\begin{array}{ccc}
\exp(\om,\gamma) & = & 
\left(
\begin{array}{ccc}
1 & 0 & 0 \\
\gamma &  1        & 0 \\
\om - \gamma\otimes\gamma & -2 \gamma & 1
\end{array}
\right) \in \Gamma(\SO(TM+ 1 + T^*M)).
\end{array} 
\end{equation*}

The action $\rho_{\GM}$ then reads \begin{equation*}
\label{eq:action exact case}
\begin{array}{cccc}
\rho_{\GM}:& ( \Diff\ltimes \Omega^{2+1} )\times\GM& \rightarrow & \GM \\
&((\phi,(B,A)),(g,(\omega,\gamma)) & \mapsto & (\phi^*g, (\phi^*\om - B - A\wedge\phi^*\gamma,\phi^*\gamma-A)),
\end{array}
\end{equation*}
which restricts to actions of $\bfGDiff_{H,F}$ and $\bfGDiff_{H,F}^e$ on $\GM$.


We denote the group of {\it generalized isometries} of $V_+\simeq_{\s} (g,(\omega,\gamma))$ 
by $$\Isom(V_+)\simeq_{\s}\Isom_{H,F}(g, (\om,\gamma)),$$ and analogously for the group of exact generalized isometries  $\Isom^e(V_+)\simeq_{\s}\Isom^e_{H,F}(g,(\om,\gamma))$.
As in Proposition \ref{prop:isotropy compact exact case}, $\Isom(V_+)$ and $\Isom^e(V_+)$ are isomorphic to compact subgroups of $\Isom(g)$. 

We will prove the following result. 

\begin{theorem}
	\label{theo:slice abelian case}
	Let $V_+$ be a generalized metric on an odd exact Courant algebroid $E$. There exists an ILH submanifold $\cS$ of $\GM$ such that:
	\begin{enumerate}
		\item[a)] For all $\phi\in \Isom^e(V_+)$, $\phi(\cS)=\cS$.
		\item[b)] For all $\phi \in \GDiff^e$, if $\phi(\cS)\cap\cS\neq \emptyset$, then $\phi\in\Isom^e(V_+)$.
		\item[c)] There is a local cross-section $\chi$ of the map $\phi \mapsto \rho_{\GM}(\phi,V_+)$ on a neighbourhood
		$\cU$ of $V_+$ in $\GDiff^e\cdot V_+$ such that the map from $\cU\times\cS$ to $\GM$ given by
		$(V_1,V_2)\mapsto \rho_{\GM}(\chi(V_1),V_2)$ is a homeomorphism onto its image.
	\end{enumerate}
The same result holds for the action of $\Isom(V_+)\subset \GDiff$ on $\GM$.
\end{theorem}



The proof of Theorem \ref{theo:slice abelian case}  follows Theorems \ref{theo:slice exact case} and \ref{theo:slice exact case full group}. We work on a splitting $\s$, for which $V_+\simeq_{\s} (g,(\omega,\gamma))$. We need an invariant metric on $\GM$, which we construct now.  First, given a Riemannian metric $g$ on $M$, consider a metric on the Lie algebra of
the group $\Om^{2+1}$:
$$
(b,a) \mapsto \int_M \vert b\vert_g^2 + \vert a\vert_g^2 \: \dvol_g
$$
where $\vert b\vert_g^2 + \vert a\vert_g^2$ is the metric on $\Lambda^{2}T^*M + T^*M$ induced by $g$.
Transport this metric on $T\Om^{2+1}$ to obtain a left-invariant metric on
 $\Om^{2+1}$. Explicitly, for any $(\om,\gamma)\in\Om^{2+1}$ and $(\dot\om,\dot \gamma)\in T_{\om,\gamma}\Om^{2+1}$, set
$$
\vert\vert (\dot \om,\dot\gamma)\vert\vert_{(\om,\gamma),g}=\int_M \vert\dot \om - \gamma\wedge \dot\gamma\vert_g^2+\vert \dot\gamma\vert_g^2\: \dvol_g.
$$
Then, define a metric on $\GM$
for $V_+=(g,(\om,\gamma))\in\GM$ and $(\dot g,(\dot\om,\dot \gamma))\in T_{(g,(\om,\gamma))}\GM$:
\begin{equation}
\label{eq:invariant metric abelian case}
\vert\vert(\dot g,(\dot\om,\dot\gamma))\vert\vert^2_{V_+}=
\int_M \vert\dot \om - \gamma\wedge \dot\gamma\vert_g^2+\vert \dot\gamma\vert_g^2 +  \vert\dot g\vert_g^2 \: \dvol_g.
\end{equation}
By construction, this defines a $\GDiff_{H,F}$-invariant (actually, $\Diff\ltimes \Om^{2+1}$-invariant) Riemannian metric on $\GM$.
Just as in \cite{Eb}, this metric is smooth. Even if this is a weak Riemannian metric,
it admits a Levi-Civita connection. The proof of this fact causes no difficulties, and can be done
following Ebin \cite{Eb}. Analogously to the exact case, the metric and the connection extend to the spaces $\GM^k$ and
we obtain $\GDiff_{H,F}$-invariant exponential maps on $T\GM^k$.
We define invariant strong continuous Riemannian metrics on the spaces $\GM^k$~by
\begin{equation}
\label{eq:invariant strong metric abelian case}
\vert\vert (\dot g,(\dot\om,\dot\gamma))\vert\vert^2_{L^{2,k}_{V_+}}=
\sum_{i=0}^k\int_M \vert \nabla_g^i(\dot \om - \gamma\wedge \dot\gamma)\vert_g^2+\vert \nabla^i_g(\dot\gamma)\vert_g^2 +  \vert \nabla^i_g(\dot g)\vert_g^2 \: \dvol_g.
\end{equation}
\begin{remark}
	\label{rem:twisted metrics}
	For a generalized metric $V_+=(g,(\om,\gamma))$,
	the invariant metrics (\ref{eq:invariant strong metric abelian case}) are nothing but the usual $L^{2,k}_g$ metrics defined by equation (\ref{eq:norms}) on the space of sections
	of the Riemannian vector bundle $(E,h_{V_+})$ where $E= S^2T^*M\oplus \Lambda^2T^*M\oplus T^*M$ and the metric $h_{V_+}$ is twisted by $(\om,\gamma)$:
	$$h_{V_+}(\dot g,(\dot\om,\dot\gamma))^2=  \vert\dot \om - \gamma\wedge \dot\gamma\vert_g^2+\vert \dot\gamma\vert_g^2 +  \vert\dot g\vert_g^2. $$
	We will consider the bundle $ \bfGamma(S^2T^*M)\times \bfOm^{2+1}$ as the tangent bundle to $\GM$ at a given point $V_+$. Then, the underlying Riemannian structure
	will be given by the twisted metric $h_{V_+}$. This will be taken into account when computing adjoint of operators that take value into $\bfGamma(S^2T^*M)\times \bfOm^{2+1}$.
\end{remark}

We denote, for simplicity, $\Isom_{H,F}(g,(\omega,\gamma))$ by $\Isom_{H,F}(V_+)$, and similarly for the exact case. 
\begin{proposition}
	\label{prop:isotropy ILH abelian case}
	Let $V_+\in\GM$. Then $\Isom^e_{H,F}(V_+)$ is a strong ILH Lie subgroup of $\bfGDiff_{H,F}^e$. Moreover,
	the quotient space $\bfGDiff_{H,F}^e/\Isom^e_{H,F}(V_+)$ carries an ILH manifold structure.
\end{proposition}

\begin{proof}
	The proof follows the outline of the one of Proposition \ref{prop:isotropy ILH exact case} 
	and is an application of Theorem \ref{theo:Frobenius}.
	Consider the map
	\begin{equation}\label{eq:Phi-Bn}
	\begin{array}{cccc}
	\Phi : & \bfDiff\ltimes \bfOm^{2+1} & \rightarrow &  \bfGamma(S^2T^*M)\times\bfOm^{2+1}\\
	&  (\phi,(B,A)) & \mapsto & \rho_{\GM}((\phi,(B,A)),(g,(\om,\gamma)))
	\end{array}
	\end{equation}
	The differential of $\Phi$ restricted to $\bfGDiff_{H,F}^e$ 	defines a right-invariant $\cC^{\infty,\infty}$ ILH normal bundle homomorphism
	from $T\bfGDiff_{H,F}^e$ to the bundle $B(\bfGDiff_{H,F}^e,\bfOm^{2+1}\times\bfGamma(S^2T^*M),T_{gm})$,
	with
	$$
	T_{gm}((\dot g,(\dot\om,\dot \gamma)),(\xi(u),(B,A)),(\xi(v),(B',A')))=\xi(v)^*(\dot g,(\dot\om,\dot \gamma)),
	$$
	where $(\xi,U)$ is a chart at the origin for $\bfDiff$.
	Denote by $\bfA_e$ this differential at the origin:
	\begin{equation*}
	\begin{array}{cccc}
	\bfA_e: & \gdiff_{H,F}^e & \rightarrow & \bfGamma(S^2T^*M) \times \bfOm^{2+1}\\
	&  (u,(b,a))     & \mapsto     &    (\cL_u g,(\cL_u\om-b-a\wedge\cL_u\gamma,\cL_u\gamma -a ))  
	\end{array}
	\end{equation*}

	We need to check hypotheses $(a)$ to $(d)$ of Frobenius' Theorem \ref{theo:Frobenius} 	for $\bfA_e$. 	To do this, we consider another parameterization of $\gdiff_{H,F}^e$.
	We endow $\Gamma(TM+1+T^*M)$
	with an ILH structure using the ILH chain 
	$\bfGamma(E):=\lbrace  \bfGamma(TM)\times\bfOm^{1+0},\;  \gamma(TM)^{k+1}\times\Om^{1+0,k+1}, \; k\geq n+5\rbrace$, with $\Omega^{1+0}$ the space $\Omega^1 \times \cCi(M)$.
	By construction of $\bfGDiff_{H,F}^e$, there is a surjective morphism
	\begin{equation*}
	\label{eq:injection map abelian case}
	\begin{array}{cccc}
	\iota_e : & \bfGamma(TM+1+T^*M) & \rightarrow & \gdiff_{H,F}^e \\
	& u+f+\al & \mapsto & (u, (\iota_u H + 2f F -d\al,\iota_u F-df)).
	\end{array}
	\end{equation*}
	Set $\bfA:=\bfA_e\circ \iota_e$, that is, the first-order differential operator 
	\begin{equation*}
	\begin{array}{cccc}
	\bfGamma(TM+1+T^*M)& \stackrel{\bfA}\rightarrow &  \bfGamma(S^2T^*M)\times\bfOm^{2+1}\\
	  u + f+ \al     & \mapsto     &    (\cL_u g,(\cL_u\om-\iota_uH-2f F +d\al+(\iota_uF-df)\wedge\gamma,\cL_u\gamma-\iota_uF+df)).  
	\end{array}
	\end{equation*}
	The differential $f\mapsto df\in \bfGamma(TM+1+T^*M)$ defines a first-order operator on the ILH chain
	$\lbrace \bfOm^0,\; \Om^{0,k+2}, k\geq n+5 \rbrace$. It is routine to check that the sequence
		\begin{equation*}
	\label{eq:complex proof isom ILH abelian}
	\bfOm^0 \stackrel{d}{\rightarrow} \bfGamma(TM+1+T^*M) \stackrel{\bfA}{\rightarrow} \Gamma(TM)\times\Om^{2+1}
	\end{equation*}
	is elliptic at the middle.
	By Proposition \ref{prop:elliptic implies Frobenius} from the Appendix \ref{sec:Appendix}, the operator $\bfA$ satisfies the hypotheses $(a)$ to $(d)$ of Proposition \ref{prop:elliptic implies Frobenius}. Moreover, $\iota_e$ admits a continuous right inverse:
	$$
	\iota_e^{-1}:(u,(b,a))\mapsto u - \bbG d^*(a - \iota_u F) -  \bbG d^*(b-\iota_u H-2\bbG d^*(\iota_u F-a)\wedge F).
	$$
	Using the Hodge decomposition, we see that $\Im(\iota_e^{-1})$ is a complementary subspace of $\ker \iota_e$.
	Then, as in the proof of Proposition \ref{prop:isotropy ILH exact case},
	we deduce the decompositions into closed subspaces
	\begin{align*}
	\gdiff_{H,F}^e & {} =\ker \bfA_e \oplus \iota_e(\Im \bfA^*),&
	\bfGamma(S^2T^*M)\times\bfOm^{2+1}&{}=\Im \bfA_e \oplus \ker \bfA^*,
	\end{align*}
	and that $\bfA_e$ satisfies hypotheses $(a)$ to $(d)$ of Theorem \ref{theo:Frobenius}.
\end{proof}
A similar statement holds for the full group of generalized diffeomorphisms:
\begin{proposition}
	\label{prop:isotropy ILH full group abelian case}
	Let $V_+\in\GM$. Then $\Isom_{H,F}(V_+)$ is a strong ILH Lie subgroup of $\bfGDiff_{H,F}$. Moreover,
	the quotient space $\bfGDiff_{H,F}/\Isom_{H,F}(V_+)$ carries an ILH manifold structure.
\end{proposition}
\begin{proof}
	Restrict the map $\Phi$ from \eqref{eq:Phi-Bn} to $\bfGDiff_{H,F}$ and denote by $\bfA'$ its differential at the origin,
	\begin{equation*}
	\begin{array}{cccc}
	\bfA': & \gdiff_{H,F}  & \rightarrow &  \bfGamma(S^2T^*M)\times \bfOm^{2+1}\\
	&  (u,(b,a))      & \mapsto     &  (\cL_u g, (\cL_u\om-b-a\wedge\cL_u\gamma,\cL_u\gamma-a)).
	\end{array}
	\end{equation*}
	We will use the elliptic complex from Remark \ref{rem:elliptic complex F twisted}. We have a Hodge
	decomposition
	$$
	\Om^{2+1}=\cH_F^2\oplus d_F\Om^{1+0}\oplus d_F^*\Om^{3+2}
	$$
	with $\cH_F^2$ the harmonic elements with respect to the corresponding Laplacian. Consider the map 
	\begin{equation*}
	\begin{array}{cccc}
	\Psi : &\gdiff_{H,F} & \rightarrow & \bfGamma(TM) \times (d\bfOm^{1+0} \oplus \cH_F^2) \\
	&  (u,(b,a))  & \mapsto &     (u,(\iota_uH - b ,\iota_uF-a)).
	\end{array}
	\end{equation*}
	Then $\Psi$ is a linear ILH continuous isomorphism, with continuous inverse, and
	$$
	\gdiff_{H,F}=\Psi^{-1}( \bfGamma(TM)\times d\bfOm^{1+0} )\oplus \Psi^{-1}( \cH_F^2)=\gdiff_{H,F}^e\oplus \cH_F^2.
	$$
	As in the proof of Proposition \ref{prop:isotropy ILH exact case full group}, we conclude the proof using the lemmas of Section \ref{sec:direct sums}.
\end{proof}

\begin{proof}[Proof of Theorem \ref{theo:slice abelian case}] Fix $V_+\in\GM$. Then, following Remark \ref{rem:twisted metrics}, the decomposition
$$
\bfGamma(S^2T^*M)\times \bfOm^{2+1} = \Im \bfA_e \oplus \ker \bfA^*
$$
corresponds to an orthogonal decomposition:
\begin{equation*}
T_{V_+}\GM=T_{V_+}(\bfGDiff_{H,F}^e\cdot V_+) \stackrel{\perp}{\oplus}\nu_{V_+}
\end{equation*}
into the tangent space of the orbit and its the normal bundle (with respect to the metric (\ref{eq:invariant metric abelian case})).
Using the invariant metrics (\ref{eq:invariant strong metric abelian case}) and Proposition \ref{prop:isotropy ILH abelian case}, the proof follows as in Section \ref{sec:moduliGM}.
\end{proof}

A corollary of Theorem \ref{theo:slice abelian case}, whose proof follows as in \cite[Corollary 8.2]{Eb}, is the following.

\begin{corollary}
	The space of generalized metrics with trivial isometry group is open in $\GM$.
\end{corollary}

With the results above, a generalization of the main result of Section \ref{sec:stratification} is straightforward.
Define the moduli space of generalized metrics $\GR$ of an odd exact Courant algebroid by
$$
\GR:=\GM/\GDiff \simeq_{\s} (\cM \times \Omega^{2+1})/ \GDiff_{H,F}.
$$

\begin{theorem}
The space $\GR$ admits an ILH stratification $(\GR_{(G)})_{(G)\in \cA}$
where $\cA$ is the set of conjugacy classes of generalized isometry groups of an odd exact Courant algebroid.
Whenever $(G)\subsetneq (G')$, we have $\GR_{(G')}\subset \overline{\GR_{(G)}}$, i.e., the strata intersect as much as possible. Moreover,
\begin{itemize}
	\item The space $\GM_G$ is an open dense subset of $\GM_G^1$, where $\GM_G^1$ is the space of generalized metrics whose generalized isometry
group contains $G$.
\item The space $\GM_{\lbrace 1\rbrace}$ is an open dense subset of $\GM$ as long as $\dim M\geq 2$.
\item In a neighbourhood of a generalized metric with non-trivial generalized isometry group, $\GM$ is not an ILH manifold.
\item In each homotopy class of a loop in $\GR$, there is a loop that can be lifted to $\GM$.
\end{itemize}
\end{theorem}



\begin{appendices}

  \section{Appendix}
  \label{sec:Appendix}

  We give some analytical results on elliptic operator theory and ILH chains which are used to prove Propositions \ref{prop:isotropy ILH exact case} and \ref{prop:isotropy ILH exact case full group}, leading to the proof of the slice theorem (Theorems \ref{theo:slice exact case} and \ref{theo:slice exact case full group}).
  \subsection{Some results on elliptic complexes}
  Let $(E_i,h_i), i\in\lbrace 0,1,2\rbrace$ be smooth Riemannian vector bundles over a compact Riemannian manifold $(M,g)$, and let
  \begin{equation}\label{eq:elliptic-complex-app}
    \Gamma(E_0) \stackrel{\bfB}{\rightarrow} \Gamma(E_1) \stackrel{\bfA}{\rightarrow} \Gamma(E_2)
  \end{equation}
  be a complex of first-order differential operators. We assume that this complex is elliptic in the middle, that is, at the level of principal symbols, the range of $\bfB$ equals the kernel of $\bfA$ (see, e.g., \cite[Sec. VII.5]{Om97}).  Using the metrics on $E_i$, and the metric $g$, we can endow the spaces $\Gamma(E_i)$ with $L^{2,k}$ Hilbert norms
  $\vert\vert\cdot\vert\vert_k$ as in  (\ref{eq:norms}). The completion of $\Gamma(E_i)$ for this $L^{2,k}$ norm will be denoted by $\Gamma(E_i)^k$. We then consider the continuous extensions of $\bfB$ and $\bfA$ on these spaces,
  as well as the ILH chains $\lbrace \bfGamma(E_i), \Gamma(E_i)^{k+2-i}\st k\in \NN(d) \rbrace$
  for some $d\in \NN$, $i\in\lbrace 0,1,2\rbrace$. Note the shift in the norms in the definition of the ILH chains, so that
  $\bfA$ and $\bfB$ define $\cC^\infty$ ILH maps. We denote by $\bfA^*$ and $\bfB^*$ the adjoint operators of $\bfA$ and $\bfB$, 
  defined on smooth sections by the $L^{2,0}$ pairing: for $f_i\in \Gamma(E_i)$,
  $$
  \int_M h_2(\bfA f_1,f_2) \; \dvol_g=\int_M h_1 (f_1,\bfA^*f_2) \; \dvol_g,
  $$
  and similarly for $\bfB$.  Let $\Delta=\bfA^*\bfA+\bfB\bf\bfB^*$ be the elliptic operator of the complex \eqref{eq:elliptic-complex-app} and denote by $\bbG$ the associated Green operator. The following lemma follows from the proof of \cite[VII, Lemma 5.10]{Om97}.

  \begin{lemma}
    \label{lem:estimates left inverse}
    The operator $\bbG\bfA^*: \bfA\bfGamma(E_1)\rightarrow \bfGamma(E_1)$ is an ILH linear normal left inverse for 
    $\bfA:\bfGamma(E_1)\rightarrow \bfA\bfGamma(E_1)$ on $\Im\bfA^*$,
    and there are constants $C$ and $D_k$ such that, for $w\in \bfGamma(E_2)$ and $k\in\NN(d)$,
    $$
    \vert\vert \bbG\bfA^* w\vert\vert_{k+1}\leq C \vert\vert w\vert\vert_k +D_k \vert\vert w\vert\vert_{k-1}.
    $$
  \end{lemma}
  \begin{proof}
    The fact that $\bbG\bfA^*$ is a left inverse follows from definition of $\Delta$ and $\bbG$.
    For the estimate, first apply the uniform G\aa{}rding's inequality \cite[VII, Lemma 5.2]{Om97} to
    $\bbG\bfA^*w$ to obtain
    $$
    \vert\vert \bbG\bfA^* w\vert\vert_{k+1}\leq C \vert\vert \bfA^*w\vert\vert_{k-1} +D_k \vert\vert \bbG\bfA^*w\vert\vert_{k}.
    $$
    Recall that $\bbG_k: \bfA^*\Gamma(E_2)^{k+1} \rightarrow \bfA^*\Gamma(E_2)^{k+3}$ is continuous.
    Thus,
    $$
    \vert\vert \bbG\bfA^* w\vert\vert_{k+1}\leq C \vert\vert \bfA^*w\vert\vert_{k-1} +D_k' \vert\vert \bfA^*w\vert\vert_{k-2},
    $$
    and the result follows, as $\bfA^*$ is a linear ILH normal operator.
  \end{proof}

  We now gather all the results that enable us to use Frobenius' theorem in the ILH category.
  \begin{proposition}
    \label{prop:elliptic implies Frobenius}
    The following statements hold:
    \begin{enumerate}
    \item[a)] The image $\Im \bfA^*$ is a closed subspace of $\Gamma(E_1)$ such that $\Gamma(E_1)=\ker \bfA \oplus \Im \bfA^*$. 
    \item[b)] The image $\Im \bfA$ is a closed subspace of $\Gamma(E_2)$ such that $\Gamma(E_2)=\Im \bfA \oplus \ker \bfA^*$.
    \item[c)] There are constants  $C$ and $D_k$, for $k\in\NN(d)$ with $D_d=0$, such that, for $u\in \Im(\bfA^*)$,
      $$\vert\vert \bfA u\vert\vert_k\geq C\vert\vert u\vert\vert_{k+1} - D_k \vert\vert u \vert\vert_{k}.$$
    \item[d)] There are constants  $C'$ and $D'_k$, for $k\in\NN(d)$ with $D'_d=0$, such that, the projection $p:\Gamma(E_2) \rightarrow \Im \bfA$ with respect to the decomposition in b) satisfies, for all $v\in\Gamma(E_2)$,
      $$\vert\vert pv\vert\vert_k\leq C' \vert\vert v\vert\vert_k + D_k'\vert\vert v\vert\vert_{k-1}. $$
    \end{enumerate}
  \end{proposition}

  \begin{proof}
    The proof follow from considerations in \cite[VII.5]{Om97} (note the switch of $\bfA$ and $\bfB$ in our notation). 
    First-order linear operators are linear ILH normal maps, as follows from  \cite[V, Thm. 3.1]{Om97}, the fact that the $1$-jet map is
    an ILH linear normal operator, and the stability of ILH linear normal operators by composition.
    In particular, $\bfA$ and $\bfA^*$ are linear normal. From \cite[VII.5]{Om97}, we also have linear estimates for $\bbG$: for all $v\in \bfGamma(E_1)$,
    $$
    \vert\vert \bbG v\vert\vert_{k+2} \leq c\vert\vert v\vert\vert_k+d_k\vert\vert v\vert\vert_{k-1}.
    $$

    Let $\bfA_k^*$ and $\bfB_k^*$ be the extensions of $\bfA^*$ and $\bfB^*$
    to the Hilbert completions $\Gamma(E_i)^k$, $i\in\lbrace 1,2\rbrace$. The extensions of $\Delta$ and $\bbG$ are denoted by $\Delta_k$ and $\bbG_k$.
    Then,  for $k\in\NN(d)$, we have $\Gamma(E_1)^k=\ker \bfA \oplus \Im \bfA^*_{k+1}$. 
    Taking the inverse limit, we obtain the decomposition
    $\Gamma(E_1)=\ker \bfA \oplus \Im \bfA^*$.
    From $\bfA^*\bfA \bbG=1$ on $\Im \bfA^*$, the fact that $\bfA$ and $\bfA^*$ are linear ILH normal and linear estimates for $\bbG$,
    we deduce the estimate
    $$
    \vert\vert v\vert\vert_k \leq \vert\vert \bfA^*\bfA \bbG v\vert\vert_k\leq c\vert\vert v\vert\vert_k+d_k\vert\vert v\vert\vert_{k-1}
    $$
    for all $v\in \Im(\bfA^*)$ and for a uniform constant $c$. Using the fact that
    $\Im \bfA^*_{k+1}= \Im \bfA^*_{k+1}\bfA \bbG_k$, we see that $\Im \bfA^*$ is closed and a) follows.

    For b), the decomposition $w=\bfA \bbG\bfA^*w+(1-\bfA \bbG\bfA^*)w$, for $w\in \Gamma(E_2)$, gives the decomposition
    $\Gamma(E_2)=\Im \bfA \oplus \ker \bfA^*$. To show that $\Im \bfA$ is closed, we will use the estimate
    $$
    \vert\vert w\vert\vert_k \leq \vert\vert \bfA \bbG\bfA^*w\vert\vert_k\leq c\vert\vert w\vert\vert_k+d_k\vert\vert w\vert\vert_{k-1}
    $$
    on $\Im \bfA$, where the first inequality follows from $\bfA \bbG\bfA^*=1$ on $\Im \bfA$, 
    and the second one from Lemma \ref{lem:estimates left inverse} and the fact that $\bfA$ is an ILH linear normal map.

    For c),  we use  $\bbG\bfA^*\bfA=1$ on $\Im \bfA^*$, together with Lemma \ref{lem:estimates left inverse}, to obtain
    $$
    \vert\vert u\vert\vert_{k+1}=\vert\vert \bbG\bfA^*\bfA u\vert\vert_{k+1}\leq C\vert\vert \bfA u\vert\vert_k + D_k \vert\vert \bfA u \vert\vert_{k-1}.
    $$
    As $\bfA$ is linear normal, the estimate in c) follows.

    Finally,  d) follows from $p=\bfA \bbG\bfA^*$, the fact that $\bfA$  is linear normal, and Lemma \ref{lem:estimates left inverse}.
  \end{proof}

  \subsection{On direct sum decompositions of ILH chains}
  \label{sec:direct sums}
  Let $\bfA$ be a linear normal operator between ILH chains. We assume that the restriction of $\bfA$ to a closed subspace of its domain
  satisfies the hypotheses of Theorem \ref{theo:Frobenius}. We consider in this section a special setting in which the hypotheses of Frobenius' Theorem (Theorem \ref{theo:Frobenius}) extend to $\bfA$.

  Let $\lbrace \bfE, E^k\st k\in\NN\rbrace$ and $\lbrace \bfF, F^k, k\in\NN \rbrace$ be two ILH chains. We assume in this section that, for each $k$ and all $x\in \bfE$,
  \begin{equation}
    \label{eq:ILH norm comparison}
    \vert\vert x\vert\vert_0 \leq \vert\vert x \vert\vert_k
  \end{equation}
  Since in this paper we only consider {\it graded Fr\'echet spaces}, satisfying $\vert\vert \cdot\vert\vert_k\leq\vert\vert\cdot\vert\vert_{k+1}$,
  equation (\ref{eq:ILH norm comparison}) is always satisfied.
  Let
  $$
  \bfA : \bfE \rightarrow \bfF
  $$
  be a linear continuous ILH operator. We assume that
  $$
  \bfE=\bfE_0\oplus \cH
  $$
  where $\cH$ and $\bfE_0$ are closed ILH subspaces of $\bfE$ with $\cH$ finite dimensional. Denote by $\bfA_0$ the restriction of $\bfA$ to $\bfE_0$.
  Assume that we have direct sum decompositions into closed subspaces in the ILH category:
  \begin{align*}
    \bfE_0=\ker \bfA_0 \oplus \bfE_2,&& \bfF=\Im \bfA_0 \oplus \bfF_2.
  \end{align*}
  We assume, moreover, that $\ker \bfA$ is finite dimensional. We then have the following.
  \begin{lemma}
    \label{lem:kernel decomposition abstract lemma}
    There exists a subspace $\cH'\subset\cH$ such that the following decomposition into closed subspaces holds in the ILH category:
    $$
    \bfE=\ker \bfA \oplus \bfE_2 \oplus \cH'.
    $$
    In particular, $\bfE_2\oplus \cH'$ is a closed complementary subspace of $\ker \bfA$ in $\bfE$.
  \end{lemma}
  \begin{proof}
    By hypothesis, we have a decomposition into closed ILH subspaces:
    $$
    \bfE=\ker \bfA_0 \oplus \bfE_2\oplus \cH.
    $$
    As $\ker \bfA$ is finite dimensional, we can complete a basis of $\ker \bfA_0$ into a basis of $\ker \bfA$ with elements $(f_i)_{1\leq i \leq r}$ of $\bfE_2\oplus \cH$.
    Write these elements as $f_i=e_i+h_i$ in the decomposition $\bfE_2\oplus \cH$ for $1\leq i\leq r$. We must have $r\leq \dim \cH$, as, otherwise, there would be a non-trivial linear combination $\sum_i \lambda_i h_i =0$. In that case, $\sum_i \lambda_i f_i=\sum_i \lambda_i e_i \in \ker \bfA\cap \bfE_2$ and by the decomposition of $\bfE_0$, we would have
    $\sum_i \lambda_i f_i=0$, which is a contradiction. We can then set $\cH'$ to be a complementary subspace of the span of $\lbrace h_i\rbrace_{1\leq i \leq r}$ in $\cH$. By construction, $\bfE=\ker \bfA \oplus \bfE_2 \oplus \cH'$ and the proof follows.
  \end{proof}
  Assume now that $\bfA_0$ satisfies the following linear estimate: there are constants $C$ and $D_k$, for $k\in\NN$ with $D_0=0$, such that, for $k\in\NN$ and $x\in \bfE_2$, 
  \begin{equation}
    \label{eq:linear estimate A0}
    \vert\vert x\vert\vert_k \leq C \vert\vert \bfA_0 x\vert\vert_k + D_k \vert\vert x\vert\vert_{k-1}.
  \end{equation}

  \begin{lemma}
    \label{lem:kernel estimate abstract lemma}
    With the previous notation, assume that, for each $k\in\NN$ and all $h\in\cH$, 
    \begin{equation}
      \label{eq:hypothesis 1 lemma kernel estimate}
      \vert\vert \bfA h \vert\vert_k=\vert\vert h \vert\vert_k.
    \end{equation}
    Assume, moreover, that there is a constant $C$
    such that for any element $x\in \bfE$ written as $x=x_0+h$ in the decomposition $\bfE_0\oplus \cH$, 
    \begin{equation}
      \label{eq:hypothesis 2 lemma kernel estimate}
      \vert\vert h\vert \vert_0\leq C \vert\vert x\vert\vert_0.
    \end{equation}
    Then $\bfA$ satisfies the following linear estimates: there are constants $C'$ and $D'_k$, for $k\in \NN$ and $D'_0=0$, such that, for $k\in\NN$ and $x\in \bfE_2\oplus\cH'$,
    $$
    \vert\vert x\vert\vert_k \leq C'\vert\vert \bfA x\vert\vert_k + D'_k \vert\vert x\vert\vert_{k-1}.
    $$
  \end{lemma}
  \begin{proof}
    Note that $\cH$ is finite dimensional, so there are constants $D_k$ such that for all $h\in \cH$,
    \begin{equation}
      \label{eq:H finite dimensional}
      \vert\vert h\vert \vert_k\leq D_k \vert\vert h\vert\vert_0.
    \end{equation}
    
    Let $x=x_0+h\in \bfE_2\oplus\cH'$. For the sake of simplicity, we denote by $C$ and $D_k$ positive constants that may vary along the estimates, but always with $C$ independent of $k$. For each $k\in\NN$,
    \begin{equation*}
      \begin{array}{cccl}
        \vert\vert x \vert\vert_k & \leq & \vert\vert x_0 \vert\vert_k + \vert\vert h\vert\vert_k & \\
        & \leq &  C\vert\vert \bfA x_0\vert\vert_k + D_k \vert\vert x_0\vert\vert_{k-1}+ \vert\vert h\vert\vert_k& \mathrm{ by  }\:(\ref{eq:linear estimate A0}) \\
        & \leq & C\vert\vert \bfA x - \bfA h\vert\vert_k + D_k \vert\vert x-h\vert\vert_{k-1}+ \vert\vert h\vert\vert_k& \\
        & \leq &  C\vert\vert \bfA x\vert\vert_k + C\vert\vert \bfA h\vert\vert_k + D_k \vert\vert x\vert\vert_{k-1}+ D_k\vert\vert h\vert\vert_0& \mathrm{ by  }\:(\ref{eq:H finite dimensional})\\
        & \leq & C\vert\vert \bfA x\vert\vert_k +  D_k \vert\vert x\vert\vert_{k-1}+  D_k\vert\vert h\vert\vert_0 & \mathrm{ by  }\:(\ref{eq:hypothesis 1 lemma kernel estimate})\: \mathrm{ and  }\:(\ref{eq:H finite dimensional})\\
        & \leq & C\vert\vert \bfA x\vert\vert_k +  D_k \vert\vert x\vert\vert_{k-1} & \mathrm{ by  }\:(\ref{eq:hypothesis 2 lemma kernel estimate})\: \mathrm{ and  }\:(\ref{eq:ILH norm comparison}).

      \end{array}
    \end{equation*}
  \end{proof}
  Recall the decomposition:
  $$
  \bfF=\Im \bfA_0 \oplus \bfF_2.
  $$
  We assume that this decomposition is orthogonal with respect to the inner product in $F^0$ (this is the case in the applications we consider in this paper, as the decompositions are of the form $\bfF=\Im \bfA \oplus \ker \bfA^*$
  where the adjoint is computed with respect to an $L^{2,0}$ inner product).
  \begin{lemma}
    \label{lem:decomposition image abstract lemma}
    There exist closed subspaces $\cF$ and $\bfF_3$ of $\bfF_2$, with $\cF$ finite dimensional, such that the following decompositions hold in the ILH category:
    \begin{align*}
      \rm F_2= \cF \oplus \bfF_3,&& \Im \bf A= \Im\bfA_0 \oplus \cF.
    \end{align*}
    In particular, the following is a decomposition into closed subspaces:
    $$
    \bfF=\Im \bfA \oplus \bfF_3.
    $$
    More precisely, $\bfF_3=\ker (p_0 : F_2^0 \rightarrow \cF)\cap \bfF_2$ where $p_0$ denote the orthogonal projection.
  \end{lemma}
  \begin{proof}
    Recall that
    $$
    \bfE=\bfE_0\oplus \cH
    $$
    with $\cH$ finite dimensional. For $(h_i)$ a basis of $\cH$, decompose $\bfA h_i= \bfA_0 y_i + f_i$ in the direct sum $\Im \bfA_0 \oplus \bfF_2$.
    Set $\cF$ to be the span of the elements $f_i$. Then $\cF$ is a finite dimensional subspace of $\bfF_2\cap \Im \bfA$.
    Then $F_2^0$ admits the orthogonal decomposition
    $$
    F_2^0=\cF\oplus \ker (p_0 : F_2^0 \rightarrow \cF)
    $$
    where $p_0$ denotes the orthogonal projection onto $\cF$. It is routine to verify that $\bfF_3=\ker (p_0 : F_2^0 \rightarrow \cF)\cap \bfF_2$
    satisfies the required properties.
  \end{proof}
  We conclude with linear estimates for the projection onto $\Im \bfA$. Assume that there are constants  $C$ and $D_k$, for $k\in\NN$ and $D_0=0$, such that the projection $\pi_0: \bfF \rightarrow \Im \bfA_0$ with respect to the decomposition $\Im \bfA_0 \oplus \bfF_2$ satisfies, for $k\in\NN$ and $x\in \bfF$,
  \begin{equation}
    \label{eq:linear estimate image A0}
    \vert\vert \pi_0x\vert\vert_k \leq C \vert\vert  x\vert\vert_k + D_k \vert\vert x\vert\vert_{k-1}.
  \end{equation}
  We then have the following lemma.
  \begin{lemma}
    \label{lem: linear estimate image abstract lemma}
    Denote by $\pi:\bfF \rightarrow \Im \bfA$ the projection with respect to the decomposition $\bfF=\Im \bfA \oplus \bfF_3$.
    Then there exists $C'$ and $D'_k$, for $k\in \NN$ with $D'_0=0$, such that, for $k\in\NN$ and $x\in \bfF$,
    $$
    \vert\vert \pi x\vert\vert_k \leq C' \vert\vert  x\vert\vert_k + D'_k \vert\vert x\vert\vert_{k-1}.
    $$
  \end{lemma}
  \begin{proof}
    Let $x\in \bfF$. From Lemma \ref{lem:decomposition image abstract lemma}, we have $\Im \bf A= \Im\bfA_0 \oplus \cF$ and $\bfF=\Im\bfA_0 \oplus \cF\oplus \bfF_3$.  As this decomposition is orthogonal with respect to the inner product on $F^0$, we have $\pi x= \pi_0 x+ p_0 x$ where $p_0$ denotes the orthogonal projection from $F^0$ to $\cH$. Thus,
    $$
    \vert\vert \pi x \vert\vert_k\leq   \vert\vert \pi_0 x \vert\vert_k +\vert\vert p_0 x \vert\vert_k.
    $$
    But $\cH$ is finite dimensional, so, if $(e_i)$ denotes an orthonormal basis of $\cH$ with respect to $\vert\vert\cdot\vert\vert_0$, 
    $$
    \vert\vert p_0 x \vert\vert_k\leq D_k \vert\vert x \vert\vert_0
    $$
    with $D_k = \sum_i \vert\vert e_i \vert\vert _k$.
    Using hypothesis (\ref{eq:linear estimate image A0}), the result follows.
  \end{proof}

\end{appendices}

\end{document}